\documentclass[]{article}
\usepackage[utf8]{inputenc}
\usepackage[a4paper,top=2.5cm,bottom=2.5cm,left=2.25cm,right=2.25cm]{geometry}
\usepackage{graphicx} 
\usepackage{float}
\usepackage{textcomp}
\usepackage{xcolor}
\usepackage{amsmath}
\usepackage{amssymb}
\usepackage{amsthm}
\usepackage{amsfonts}
\usepackage[english]{babel}
\usepackage{verbatim}
\usepackage{array}
\usepackage{hyperref}
\usepackage{afterpage}
\numberwithin{equation}{section}
\newtheorem{theorem}{Theorem}[section]
\newtheorem{lemma}{Lemma}[section]
\newtheorem{corollary}{Corollary}[section]
\newtheorem{prop}{Proposition}[section]

\theoremstyle{definition}
\newtheorem{definition}{Definition}[section]

\theoremstyle{remark}
\newtheorem*{remark}{Remark}

\newcommand{\rchi}{\protect\raisebox{2pt}{$ \chi $}}
\newcommand{\ov}{\overline{v}}
\newcommand{\ow}{\overline{w}}
\newcommand{\oF}{\overline{F}}
\newcommand{\oFF}{\overline{\mathbb{F}}}
\newcommand{\ol}{\overline{\lambda}}

\newcommand{\tFF}{\tilde{\mathbb{F}}}

\newcommand{\FF}{\mathbb{F}}
\newcommand{\KK}{\mathbb{K}}
\newcommand{\eps}{\varepsilon}

\newcommand\pef[1]{(\ref{#1})}

\title{On a kinetic equation describing the behavior of a gas interacting mainly with radiation}
\author{Elena Demattè}

\begin{document}

\maketitle

\begin{abstract}
In this article we study a kinetic model which describes the interaction between a gas and radiation. Specifically, we consider a scaling limit in which the interaction between the gas and the photons takes place much faster than the collisions between the gas molecules themselves. We prove in the homogeneous case that the solutions of the limit problem solve a kinetic equation for which a well-posedness theory is considered. The proof of the convergence to a new kinetic equation is obtained analyzing the dynamics of the gas-photon system near the slow manifold of steady states.\\

\textbf{Acknowledgments:} The author thanks Juan J.L. Velázquez for the suggestion of
the problem and helpful discussions. The author has been supported by the Deutsche Forschungsgemeinschaft (DFG, German Research Foundation) under Germany’s Excellence Strategy – EXC-2047/1 – 390685813 as a member of the Bonn International Graduate School of Mathematics (BIGS). Furthermore, the work was supported by the
DFG through the collaborative research centre \textit{The mathematics of emerging effects} (CRC 1060, Project-ID 211504053). 
\end{abstract}
\tableofcontents
\section{Introduction}
In this paper we study a kinetic model that describes the interaction between the molecules of a gas and radiation. In the considered model it is assumed that the gas molecules can be in a ground and in an excited state. The transition between these two states takes place either due to nonelastic collisions between two molecules in the ground state, or by the absorption and emission of photons. The evolution of the radiation density is described also by means of a kinetic equation, specifically the radiative transfer equation in which the emission and absorption by the gas molecules are included. In spite of its simplicity the model provides valuable insights about the behavior of radiative gases.

We will denote as $ F^1=F^1(t,v) $ the density in the phase space of gas molecules in the ground state, and as $ F^2=F^2(t,v) $ the density of gas molecules in the excited state. We consider in this paper only homogeneous distributions in space (due to that $ F^1 $ and $ F^2 $ are independent of $ x $).

The most peculiar feature of the model considered in this paper is that it is assumed that the fastest kinetic process is the interaction between the photons and the gas molecules. As a consequence we have the following relation
\begin{equation}\label{introduction1}
F^2=\lambda F^1,
\end{equation} 
where $ \lambda=\lambda(t) $ is a function that depends on the intensity of radiation, which in fact depends on the initial amount of radiation present on the system. In the limit described above, which from now on will be referred to as the fast radiation limit, we will obtain that \pef{introduction1} holds approximately during most of the evolution. We will derive also equations describing the dynamic of $ \lambda(t)$ and $ F^1(t,v) $ which change in the time scale of the mean free time between the collisions of the gas.

\subsection{The model}
The model considered in the paper, which has been already introduced in \cite{rossanipolew,rossanispiga, paper}, is a simple model of gas-photon interaction. In this model it is assumed that the gas molecules can be in two states, namely the ground state $ A $ and the excited state $ \overline{A} $. The difference of energy between these states is given by $ \eps_0=h\nu_0 $, where $ \nu_0 $ is the frequency of the photons $ \gamma $. The radiative process in the system is assumed to be monochromatic. In particular all the effect yielding frequency dispersion like Doppler effect and line width are neglected. The considered system is also open with the exterior to the exchange of photons but not to the exchange of gas molecules. This is for example the situation of stellar atmospheres. 


In our model we only deal with three possible types of interaction: elastic collisions, nonelastic collisions and radiative processes. In this paper we neglect all the possible processes which involve three or more particles. 

The elastic collision is one of the following processes
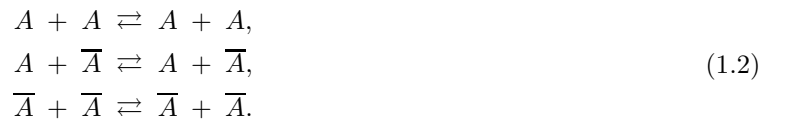
\begin{equation}\label{elcoll}
\begin{split}
A\;+\;A\;&\rightleftarrows\;A\;+\;A,\\
A\;+\;\overline{A}\;&\rightleftarrows\;A\;+\;\overline{A},\\
\overline{A}\;+\;\overline{A}\;&\rightleftarrows\;\overline{A}\;+\;\overline{A}.\\
\end{split}
\end{equation}
The nonelastic collisions are only between ground state gas molecules, which colliding lose kinetic energy. This is absorbed by an electron yielding one of the molecule to be in the excited state. They can be schematically represented by
\begin{equation}\label{nonelcoll}
A\;+\;A\;\rightleftarrows\;A\;+\;\overline{A}.
\end{equation}
Finally, the radiative processes are only of the following ones
\begin{equation}\label{rad}
A\;+\;\gamma\;\rightleftarrows\;\overline{A}.
\end{equation}
In this case the absorption of a photon by a ground state molecule creates an excited state molecule. Also the emission (spontaneous or stimulated) of a photon due to the de-excitation of a molecule might happen. This is the so called bound-bound transition as explained in \cite{mihalas, oxenius, Rutten}. We neglect the momentum of the photons in all the interactions. Moreover, we assume that for every interaction the inverse process is also possible.

Elastic collisions as in \pef{elcoll} are characterized by the relation between pre-collisional velocities ($ v_1 $ and $ v_2 $) and post-collisional velocities ($ v_3 $ and $ v_4 $). This is given by the conservation of momentum and of kinetic energy
\begin{equation}\label{elasticrel1}
v_1+v_2=v_3+v_4\;\;\;\;\;\;\;\;\;\;\text{ and }\;\;\;\;\;\;\;\;\;\;\;\arrowvert v_1\arrowvert^2+\arrowvert v_2\arrowvert^2=\arrowvert v_3\arrowvert^2+\arrowvert v_4\arrowvert^2.
\end{equation}
Therefore we conclude also
\begin{equation}\label{elasticrel2}
v_{3,4}=\frac{v_1+v_2}{2}\pm\frac{\arrowvert v_1-v_2\arrowvert}{2} \omega \;\;\;\;\;\;\;\;\text{ for }\omega\in\mathbb{S}^2.
\end{equation}

Similarly, nonelastic collisions as in \pef{nonelcoll} are also characterized by the relation between pre-collisional velocities ($ \ov_1 $ and $ \ov_2 $) and post-collisional velocities ($ \ov_3 $ and $ \ov_4 $). From the conservation of momentum and total energy we compute
\begin{equation}\label{nonelasticrel1}
\ov_1+\ov_2=\ov_3+\ov_4\;\;\;\;\;\;\;\;\;\;\text{ and }\;\;\;\;\;\;\;\;\;\;\;\arrowvert \ov_1\arrowvert^2+\arrowvert \ov_2\arrowvert^2=\arrowvert \ov_3\arrowvert^2+\arrowvert \ov_4\arrowvert^2+2\varepsilon_0.
\end{equation}
and therefore the relation between pre- and post-collisional velocities is given by
\begin{equation}\label{nonelasticrel2}
\ov_{3,4}=\frac{\ov_1+\ov_2}{2}\pm\omega\sqrt{\frac{\arrowvert \ov_1-\ov_2\arrowvert^2}{4}-\varepsilon_0}  \;\;\;\;\;\;\;\;\text{ for }\omega\in\mathbb{S}^2
\end{equation}
and 
\begin{equation}\label{nonelasticrel3}
\ov_{1,2}=\frac{\ov_3+\ov_4}{2}\pm\omega\sqrt{\frac{\arrowvert \ov_3-\ov_4\arrowvert^2}{4}+\varepsilon_0}  \;\;\;\;\;\;\;\;\text{ for }\omega\in\mathbb{S}^2.
\end{equation}

We have already introduced the gas density functions. The behavior of the photons is usually described by the intensity of radiation $ I(t,x,n) $, where $ n\in\mathbb{S}^2 $. This is the energy flux due to the radiative process. However, for the description of this model we will use the photon number density $ Q(t,x,n) $, which is the number of photons per unit volume per unit frequency. This latter function is related to the former one by the identity $ I= ch\nu_0 Q $ (cf. \cite{oxenius}), where $ c $ is the speed of light and $ h $ the Planck constant.  

In this paper we will restrict ourselves to the homogeneous problem only, so $ F^i $ and $ Q $ would be independent of $ x $. 
We combine the kinetic equation for the elastic and nonelastic collisions, which involve several Boltzmann operators, with the radiative transfer equation. According to paper \cite{paper} in the homogeneous case the following system of kinetic equation is then obtained, which describes the behavior of the gas densities and the photons in the model under consideration.
\begin{equation}\label{kinsystem}
\begin{split}
\frac{\partial F^1 (v)}{\partial t} = & \sum_{i=1}^{2} \int_{\mathbb{R}^3} dv_2 \int_{\mathbb{S}^2} d\omega\; B_{el}(v,v_2)\left(F^1(v_3)F^i(v_4)-F^1(v)F^i(v_2)\right)\\
\;\; & +2\int_{\mathbb{R}^3} d\ov_2 \int_{\mathbb{S}^2} d\omega\; B_{ne}^{12}(v,\ov_2)\left(F^2(\ov_3)F^1(\ov_4)-F^1(v)F^1(\ov_2)\right)\\
\;\; & +\int_{\mathbb{R}^3} d\ov_3 \int_{\mathbb{S}^2} d\omega\; B_{ne}^{34}(\ov_3,v)\left(F^1(\ov_1)F^1(\ov_2)-F^2(\ov_3)F^1(v)\right)\\
\;\; & + \int_{\mathbb{S}^2} dn\; A_0 F^2(v)+B_0 Q(n)\left(F^2(v)-F^1(v) \right)\\
\frac{\partial F^2 (v)}{\partial t} = & \sum_{i=1}^{2} \int_{\mathbb{R}^3} dv_2 \int_{\mathbb{S}^2} d\omega\; B_{el}(v,v_2)\left(F^2(v_3)F^i(v_4)-F^2(v)F^i(v_2)\right)\\
\;\; & +\int_{\mathbb{R}^3} d\ov_4 \int_{\mathbb{S}^2} d\omega\; B_{ne}^{34}(v,\ov_4)\left(F^1(\ov_1)F^1(\ov_2)-F^2(v)F^1(\ov_4)\right)\\
\;\; & - \int_{\mathbb{S}^2} dn\; A_0 F^2(v)+B_0 Q(n)\left(F^2(v)-F^1(v) \right)\\
\frac{\partial Q (n)}{\partial t} =& \int_{\mathbb{R}^3} dv\; A_0 F^2(v)+B_0 Q(n)\left(F^2(v)-F^1(v) \right)\\
\end{split}
\end{equation}
We define $ \mathbb{F}:= (F^1,F^2,Q) $ and we consider $ F^i(t,v)\geq 0 $, $ Q(t,n)\geq 0 $. Moreover, the collision kernels in the case of hard spheres are given by
\begin{equation}\label{kernels}
\begin{split}
B_{el}(v_1,v_2):=&\arrowvert v_1-v_2\arrowvert,\\
B_{ne}^{12}(\ov_1,\ov_2):=&C_0 \frac{\sqrt{\arrowvert \ov_1-\ov_2\arrowvert^2-4\varepsilon_0}}{2\arrowvert \ov_1-\ov_2\arrowvert}\arrowvert \ov_1-\ov_2\arrowvert,\\
B_{ne}^{34}(\ov_3,\ov_4):=&C_0 \frac{\sqrt{\arrowvert \ov_3-\ov_4\arrowvert^2+4\varepsilon_0}}{2\arrowvert \ov_3-\ov_4\arrowvert}\arrowvert\ov_3-\ov_4\arrowvert.\\
\end{split}
\end{equation}
For their full derivation we refer to \cite{paper}.

The constants in the radiative terms are defined as $ A_0=\frac{B_{12} 2 \varepsilon_0 \nu_0^2}{4\pi c^2}\geq 0 $ and $ B_0= \frac{B_{12} c \varepsilon_0}{4\pi}\geq 0 $, where $ B_{12} $ is the Einstein's absorption coefficient. For a detailed derivation of the radiative transfer equation we refer to \cite{mihalas,oxenius,Rutten}.\\

\begin{remark}
	It is well known that
	\begin{equation}\label{usual estimate 3}
	\arrowvert v-w\arrowvert\leq (1+\arrowvert v\arrowvert^2)^{1/2}(1+\arrowvert w\arrowvert^2)^{1/2}.
	\end{equation}
	This implies the following bounds for the kernels for the nonelastic collisions
	\begin{equation}\label{usual estimate}
	\begin{split}
	B_{ne}^{12}(\ov_1,\ov_2)\leq&C_0 \frac{\sqrt{\arrowvert \ov_1-\ov_2\arrowvert^2+4\varepsilon_0}}{2}\\
	\leq&C_0 \frac{\arrowvert \ov_1-\ov_2\arrowvert+2\sqrt{\varepsilon_0}}{2}\\
	\leq& \frac{C_0}{2}(1+2\sqrt{\varepsilon_0})(1+\arrowvert \ov_1\arrowvert^2)^{1/2}(1+\arrowvert \ov_2\arrowvert^2)^{1/2}, \\
	\end{split}
	\end{equation}
	\begin{equation}\label{usual estimate 2}
	B_{ne}^{34}(\ov_3,\ov_4)\leq\frac{C_0}{2}(1+2\sqrt{\varepsilon_0})(1+\arrowvert \ov_3\arrowvert^2)^{1/2}(1+\arrowvert \ov_4\arrowvert^2)^{1/2}.
	\end{equation}
	\\
\end{remark}
\begin{remark}
	In this paper we assume that the interaction of the gas molecules is described by means of hard spheres. We can obtain the same results also in the case of more general hard potentials. In that case instead of $ \arrowvert v_1-v_2\arrowvert $ we consider for the elastic kernel the function $ b(\cos \theta)\arrowvert v_1-v_2\arrowvert^\gamma $ for $ \gamma\in(0,2] $ and for the nonelastic kernels expressions of the form
	\begin{equation*}
	B_{ne}(\ov,\ow):=C_0 \frac{\sqrt{\arrowvert \ov-\ow\arrowvert^2\pm4\varepsilon_0}}{2\arrowvert \ov-\ow\arrowvert}b(\cos \theta)\arrowvert \ov-\ow\arrowvert^\gamma,
	\end{equation*}
	where $ \theta $ is the angle between the difference of pre- and post-collisional velocities.
\end{remark}
In order to simplify the reading we introduce the following notation
\begin{equation}\label{coll1}
\begin{split}
\mathbb{K}_1 \left[\FF,\FF\right](v)=& \sum_{i=1}^{2} \int_{\mathbb{R}^3} dv_2 \int_{\mathbb{S}^2} d\omega\; B_{el}(v,v_2)\left(F^1(v_3)F^i(v_4)-F^1(v)F^i(v_2)\right)\\
\;\; & +2\int_{\mathbb{R}^3} d\ov_2 \int_{\mathbb{S}^2} d\omega\; B_{ne}^{12}(v,\ov_2)\left(F^2(\ov_3)F^1(\ov_4)-F^1(v)F^1(\ov_2)\right)\\
\;\; & +\int_{\mathbb{R}^3} d\ov_3 \int_{\mathbb{S}^2} d\omega\; B_{ne}^{34}(\ov_3,v)\left(F^1(\ov_1)F^1(\ov_2)-F^2(\ov_3)F^1(v)\right)\\
\end{split}
\end{equation}
and similarly
\begin{equation}\label{coll2}
\begin{split}
\mathbb{K}_2\left[\FF,\FF\right](v)=&\sum_{i=1}^{2} \int_{\mathbb{R}^3} dv_2 \int_{\mathbb{S}^2} d\omega\; B_{el}(v,v_2)\left(F^2(v_3)F^i(v_4)-F^2(v)F^i(v_2)\right)\\
\;\; & +\int_{\mathbb{R}^3} d\ov_4 \int_{\mathbb{S}^2} d\omega\; B_{ne}^{34}(v,\ov_4)\left(F^1(\ov_1)F^1(\ov_2)-F^2(v)F^1(\ov_4)\right).\\
\end{split}
\end{equation}
Moreover we define the following operators
\begin{equation}\label{KR}
\mathbb{K}[\FF,\FF]:=\begin{pmatrix}\mathbb{K}_1[\FF,\FF]\\\mathbb{K}_2[\FF,\FF]\\0\end{pmatrix}\;\;\;\;\;\text{ and }\;\;\;\;\;\mathcal{R}[\FF]=\begin{pmatrix}\int_{\mathbb{S}^2} dn\; A_0 F^2(v)+B_0 Q(n)\left(F^2(v)-F^1(v) \right)\\-\int_{\mathbb{S}^2} dn\; A_0 F^2(v)+B_0 Q(n)\left(F^2(v)-F^1(v) \right)\\\int_{\mathbb{R}^3} dv\; A_0 F^2(v)+B_0 Q(n)\left(F^2(v)-F^1(v) \right)\end{pmatrix}
\end{equation}
so that the kinetic system can be written as 
\begin{equation}\label{compactequation}
\partial_t\FF(t,v)=\KK[\FF,\FF](t,v)+\mathcal{R}[\FF](t,v).
\end{equation}
In this paper we focus on the setting, when the gas is so much rarefied, that the collisions between gas molecules are much less frequent if compared to the radiative processes. We impose the following three conditions
\begin{equation}\label{conditions}
\begin{split}
(1)&\:\: m\arrowvert v\arrowvert^2 \approx h\nu_0,\\
(2)& \:\:Q\approx \frac{2\nu^2}{c^3},
\\
(3)& \:\:\alpha \int_{\mathbb{R}^3} \left(F^1+F^2 \right)\;dv\ll B_{12}\frac{2h\nu_0^3}{c^2},
\end{split}
\end{equation}
where $ \alpha $ is a parameter which measures the order of magnitude for the cross section for the collisions of the gas particles. With the relation $ \approx $ we mean that the quantities are of the same order. 
 
Condition $ (1) $ in \pef{conditions} means that the kinetic energy of the particles is comparable to the energy of the photons. Due to this assumption the Boltzmann ratio $ e^{-\frac{2\eps_0}{K_B T}} $ is of order one. 

The physical meaning of $ (2) $ is that for the photon density in our model the number of occupied quantum states is of the same order of available quantum states. It indicates that we need to use the Bose-Einstein quantum statistic. In order to give a more mathematical meaning of condition $ (2) $, we recall the form of the radiative transfer equation in both the terms for the evolution of the density of gas molecules and of the photons. This is given by
\begin{equation}\label{radiation}
\begin{split}
\partial_t F^i(t,v)&= \mathbb{K}_i\left[\FF,\FF\right] \pm \int_{0}^\infty \int_{\mathbb{S}^2} \frac{B_{12}}{4\pi}\delta(\nu-\nu_0)\left(\frac{2h\nu^3}{c^2}F^2\left(1+\frac{c^3}{2\nu^2} Q \right)-ch\nu Q F^1\right)\;dn\;d\nu\\
\partial_t Q(t,n,\nu)&=\int_{\mathbb{R}^3}\frac{B_{12}}{4\pi}\left(\frac{2h\nu^3}{c^2}F^2\left(1+\frac{c^3}{2\nu^2} Q \right)-ch\nu Q F^1\right)dv
\end{split}
\end{equation}
Imposing on $ Q $ the second condition yields the fact that both terms in $ \left(1+\frac{c^3}{2\nu^2} Q \right) $ have the same order of magnitude. The term $ 1 $ is associated to the spontaneous emission, while $ \frac{c^3}{2\nu^2} Q $ is related to the stimulated emission process. Assumption $ (2) $ means that both gas molecule densities $ F^1 $ and $ F^2 $ contribute equally to the radiative process. Most likely this assumption is not needed in order to derive the model but it simplifies the third and most important assumption in \pef{conditions}. In a situation where $ Q\ll \frac{2\nu^2}{c^3} $ the spontaneous radiation would be the dominant contribution, while whenever $ Q\gg \frac{2\nu^2}{c^3} $ the stimulated radiation would be the most important term. Also in these limit situations we could obtain similar results as those we will obtain in this paper.

Finally, condition $ (3) $ implies exactly that the collision terms between gas molecules are much smaller than the terms describing the interaction between gas and radiation, i.e. \\$ \left\arrowvert\mathbb{K}_i\left[\FF,\FF\right]\right\arrowvert\ll\left\arrowvert\int_{0}^\infty\;d\nu\delta(\nu-\nu_0)\int_{\mathbb{S}^2} dn\; A_0 F^2(v)+B_0 Q(n)\left(F^2(v)-F^1(v) \right)\right\arrowvert$. This is true, since with $ (2) $ in \pef{conditions} we have 
\begin{equation}
\int_{0}^\infty\;d\nu\delta(\nu-\nu_0)\int_{\mathbb{S}^2} dn\; A_0 F^2(v)+B_0 Q(n)\left(F^2(v)-F^1(v) \right)\approx B_{12}\frac{2h\nu_0^3}{c^2}F^i(v).
\end{equation}
Moreover, if $ \alpha $ is the order of magnitude of the cross section, for the Boltzmann operators we see
\begin{equation}
\mathbb{K}_i\left[\FF,\FF\right]\approx\alpha \left(\int_{\mathbb{R}^3} F^1+F^2 \;dv\right)F^i(v).
\end{equation}
This concludes the justification for condition $ (3) $ in \pef{conditions}, due to which we are assuming that the gas is very rarefied.

\begin{remark}
	We are not aware of a specific physical system, which can be described by these assumptions. Anyway, these conditions apply for example to systems of very rarefied hydrogen with an incoming radiation of frequency of visible light (more precisely $ \nu_0\approx10^{15}\;s^{-1} $ corresponding to the wavelength $ 122 \;nm $) and with temperature below $ 7000 \;K $. Already the sun photosphere, which is by far not very rarefied, with its density of order $ 10^{-4}\;kg\; m^{-3} $ yields the ratio $ \frac{\alpha \int_{\mathbb{R}^3} \left(F^1+F^2 \right)\;dv}{B_{12}\frac{2h\nu_0^3}{c^2}} $ of order $ 10^{-28} $. We remark also that the temperature in the photosphere is lower than $ 6000\;K $ and therefore the hydrogen is largely non-ionized.
\end{remark}

It is convenient to introduce a scaling which defines more mathematically the model under the assumption in \pef{conditions}. Hence, we rewrite the equation \pef{kinsystem} under the scaling $ \mathbb{K}_i\mapsto\varepsilon\mathbb{K}_i  $. The dynamic of the pair $ \left(\lambda(t),\; F^1(t)\right) $, introduced in \pef{introduction1}, will evolve with a time scale given by the relation between the collision and the radiative terms. Therefore we scale the time like $ t\mapsto \varepsilon t $ and, setting $ A_0=B_0=1 $ for simplicity, we study the equation
\begin{equation}\label{epsequation}
\begin{split}
\partial_t F^1(v)&= \mathbb{K}_1\left[\FF,\FF\right](v)+ \frac{1}{\varepsilon}\int_{\mathbb{S}^2} \left[F^2(v) +Q(n)(F^2(v)-F^1(v))\right]\;dn\\
\partial_t F^2(v)&= \mathbb{K}_2\left[\FF,\FF\right](v)- \frac{1}{\varepsilon}\int_{\mathbb{S}^2} \left[F^2(v) +Q(n)(F^2(v)-F^1(v))\right]\;dn\\
\partial_t Q(n)&= \frac{1}{\varepsilon}\int_{\mathbb{R}^3} \left[F^2(v) +Q(n)(F^2(v)-F^1(v))\right]\;dv
\end{split}
\end{equation}
Our goal is to find a new kinetic equation which defines the solution of this system as $ \eps\to0 $.\\

The mathematical properties of the homogeneous Boltzmann equation have been intensively studied (cf. \cite{cercignani}). The techniques for the homogeneous Boltzmann equation can be adapted for the case of nonelastic collision terms. The radiative transfer equation is linear in $ Q $ and has good properties. Therefore \pef{kinsystem} has all usual conservation laws of mass and momentum of the gas and of the energy of the system. We recall again that we neglect the momentum of the photons. Moreover it is possible to define an entropy related to the equation \pef{kinsystem}, which satisfies an H-theorem.

In this paper we will not prove the well-posedness theory for \pef{kinsystem}, which has been developed in \cite{tesi}, but we focus on the analysis of the solution of \pef{epsequation} when $ \eps $ tends to zero. We notice that formally if $ \eps=0 $ in equation \pef{epsequation} the possible solutions should satisfy the relation $ F^2=\lambda F^1 $ and $ Q=\frac{\lambda}{1-\lambda} $ for some $ \lambda(t)\in [0,1) $. We can call this set $ \mathcal{M} $ and it will be referred to as the manifold of steady states. We can expect that the solutions are close to $ \mathcal{M} $ for every time.

With a perturbative expansion we will be able to derive initially formally a new equation for the vector $ (\lambda, F) $, from which we can recover the interaction of the gas molecules with the photons by the relation $ \mathbb{F}=\begin{pmatrix}F\\\lambda F\\ \frac{\lambda}{1-\lambda}\end{pmatrix} $. The vector $ (\lambda, F) $ takes value in $ [0,1)\times L^1_2\left(\mathbb{R}^3\right) $. The kinetic equation we will derive is defined as follows
\begin{equation}\label{newsystem}
\begin{split}
\partial_t \lambda&= -\frac{(1-\lambda)^2(1+\lambda)}{(1+\lambda)+(1-\lambda)^2\int_{\mathbb{R}^3}F}\int_{\mathbb{R}^3} \mathbb{K}_1\left[F,\lambda\right]dv\\
\partial_t F&= \frac{1}{1+\lambda}\left(\mathbb{K}_1\left[F,\lambda\right]+\mathbb{K}_2\left[F,\lambda\right]\right)+\frac{(1-\lambda)^2F}{(1+\lambda)+(1-\lambda)^2\int_{\mathbb{R}^3}F}\int_{\mathbb{R}^3} \mathbb{K}_1\left[F,\lambda\right]dv\\
\end{split}
\end{equation}
Notice that this method can be thought as a generalized Chapman-Enskog expansion (cf. \cite{chap}). An important difference with the classical Chapman-Enskog expansion is that we do not derive an equation for the densities (mass, momentum and energy) but we construct a new equation for $ F $ and $ \lambda $. Moreover, the manifold of steady states for the homogeneous Boltzmann equation is only five dimensional, whereas here it is infinite dimensional. It is indeed characterized by the function $ F $ and $ \lambda $. Another way to understand the result of this paper is in terms of the theory of dynamical system. More precisely, we can understand the solution described in this paper by means of the dynamics of a system near a slow manifold. In the context of mathematical physics this approach wa used to described the dynamic of a small charges particle in an electromagnetic field (cf. \cite{spohn}). 

\subsection{Definitions}
Before stating the main results of this paper we shall give some important definitions and notation. First of all, we will assume throughout the paper without loss of generality that the measure of the sphere is $ 1 $, i.e. $ \mathcal{L}(\mathbb{S}^2)=1 $. As in the usual theory for the Boltzmann equation we will also have to work with some weighted spaces.
\begin{definition}
	We define the Banach Space $ L_k^1\left(\mathbb{R}^3\right) $ as the space of all integrable functions $ f$ such that $ \Arrowvert f \Arrowvert_{L_k^1}:=\int_{\mathbb{R}^3} dx \arrowvert f\arrowvert (1+|x|^2)^{k/2}<\infty $
\end{definition}
This is the subspace of integrable functions used also in the well-posedness theory for the classical Boltzmann equation. We will also say that a function $ f\in L_k^1\left(\mathbb{R}^3\right) $ has a bounded k-th moment. The next lemma gives us a useful interpolation result for these spaces.
\begin{lemma}\label{lem0}
	Let $ k\leq m $. If  $ f\in L_m^1\left(\mathbb{R}^3\right) $ then also $ f\in L_k^1\left(\mathbb{R}^3\right) $ with $ \Arrowvert f \Arrowvert_{L_k^1}\leq\Arrowvert f \Arrowvert_{L_m^1} $. Moreover the following interpolation formula holds 
	\begin{equation}
	\Arrowvert f \Arrowvert_{L_3^1}\leq\Arrowvert f \Arrowvert_{L_2^1}^{1/2}\Arrowvert f \Arrowvert_{L_4^1}^{1/2}.
	\end{equation}
	\begin{proof}
		This is a consequence of the H\"older Inequality. See for example \cite{Bressan}.
	\end{proof}
\end{lemma}
Another helpful notation that will be used in order to simplify the reading, is that given a vector $ \mathbb{F}=\left(F^1,F^2,Q\right) $ we will denote $ F $ as the vector of the first two components: $ F=\left(F^1,F^2\right) $. So that we can also write $ \mathbb{F}=(F,Q) $. Moreover we will also write for simplicity $ \Arrowvert F\Arrowvert_{L_2^1}= \Arrowvert F^1\Arrowvert_{L_2^1}+\Arrowvert F^2\Arrowvert_{L_2^1} $. With this notation we can define the space from where we consider the initial data.
\begin{definition}\label{X}
	We define the Banach space $ \mathcal{X}=\left\{\FF= \begin{pmatrix}F^1\\F^2\\Q\end{pmatrix}\in L^1_2\left(\mathbb{R}^3\right)\times L^1_2\left(\mathbb{R}^3\right)\times L^1\left(\mathbb{S}^2\right)\right\}  $ and we call $ \mathcal{X}_+=\left\{ \FF\in\mathcal{X}: \FF\geq0\right\}\subset\mathcal{X} $. 
\end{definition}We briefly give the definition of solutions for the kinetic equation \pef{kinsystem}.
\begin{definition}
	We call a non-negative vector $ \mathbb{F}=\left(F^1,F^2,Q\right)\in C^1\left([0,\infty), \mathcal{X}\right) $ a strong solution if $ \FF(t)\in\mathcal{X}_+ $ for all $ t $ and it solves \pef{kinsystem} almost everywhere.
\end{definition}
Weak solutions are defined testing \pef{kinsystem} against some proper test-functions.
\begin{definition}
	We say that a vector of non negative functions $ \mathbb{F}\in C\left([0,\infty), \mathcal{X}\right) $ is a weak solution of the kinetic system \pef{kinsystem}, if for all $ \varphi^i(v)\in C_c^{\infty}(\mathbb{R}^3)$ ,  $i=1,2 $ and $ \varphi^3(n)\in C_c^{\infty}(\mathbb{S}^2) $ the following holds
	\begin{equation}\label{weak}
	\begin{split}
	&\partial_t\left[\int_{\mathbb{R}^3} dv\; \left(F^1\varphi^1(v)+F^2\varphi^2(v)\right)+\int_{\mathbb{S}^2}dn \;Q\varphi^3(n)\right]\\
	&= \frac{1}{2}\sum_{i=1}^2\sum_{j=1}^2\int_{\mathbb{R}^3}dv_1\int_{\mathbb{R}^3}dv_2\int_{\mathbb{S}^2}d\omega\; \arrowvert\omega\cdot(v_1-v_2)\arrowvert F^i(v_1)F^j(v_2)\left(\varphi^i(v_3)+\varphi^i(v_4)-\varphi^i(v_1)-\varphi^i(v_2)\right)\\
	&+\int_{\mathbb{R}^{12}} d\ov_1 d\ov_2 d\ov_3 d\ov_4 \delta(\ov_1+\ov_2-\ov_3-\ov_4)\delta(\arrowvert\ov_1\arrowvert^2+\arrowvert\ov_2\arrowvert^2-\arrowvert\ov_3\arrowvert^2-\arrowvert\ov_4\arrowvert^2-2\varepsilon_0)W_{ne}(\ov_1,\ov_2;\ov_3,\ov_4)\\
	&\left[F^1(\ov_1)F^1(\ov_2)\left(\varphi^2(\ov_3)+\varphi^1(\ov_4)-\varphi^1(\ov_1)-\varphi^1(\ov_2)\right)+F^2(\ov_3)F^1(\ov_4)\left(\varphi^1(\ov_1)+\varphi^1(\ov_2)-\varphi^2(\ov_3)-\varphi^1(\ov_4)\right)	\right]\\
	&+\int_{\mathbb{R}^3}dv\int_{\mathbb{S}^2}dn\;h_{rad}\left[\mathbb{F}\right]\left[\varphi^1(v)-\varphi^2(v)+\varphi^3(n)\right], 
	\end{split}
	\end{equation}
	where $ W_{ne}(\ov_1,\ov_2;\ov_3,\ov_4)= C_0 \frac{\arrowvert\omega \cdot (\ov_1-\ov_2)\arrowvert}{\arrowvert \ov_1-\ov_2\arrowvert}  $.
\end{definition}
We obtained this definition testing each side of the system \pef{kinsystem} with the proper choice of $ \varphi^i $ and then integrating both sides with respect to $ v $ and $ n $. As it can be read in the Appendix B of \cite{paper} evaluating the delta functions with respect to $ \ov_1,\ov_2 $ correspondingly $ \ov_3,\ov_4 $ we obtain the desired nonelastic kernels \pef{kernels}. We will refer to \pef{weak} often as to the weak formulation of \pef{kinsystem}.\\

\begin{remark}\label{rem1}
	Using the inequalities \pef{usual estimate 3}, \pef{usual estimate} and \pef{usual estimate 2} we see that for \linebreak$ \mathbb{F}\in L^{\infty}\left( [0,\infty), L_4^1\left(\mathbb{R}^3\right)\times L_4^1\left(\mathbb{R}^3\right)\times L^1\left(\mathbb{S}^2\right)\right) $ the right hand side of the equation \pef{weak} is well defined also for $ \varphi^i(v)=C (1+\arrowvert v\arrowvert^2) $ if $ i=1,2 $ and $ \varphi^3(n)=C $ for all $ C>0 $.
\end{remark}

Moreover solutions with the control of at least the fourth moment satisfy also the conservation of mass, momentum and energy, which can be deduced by an approximating argument
\begin{equation}
\begin{split}
\partial_t\left[\int_{\mathbb{R}^3}dv \left(F^1+F^2\right)\right]=0 & \;\;\;\;\;\;\;\; \;\;\;\;\;\;\;\;\text{ (mass),}\\
\text{for }i=1,2,3\;\;\;\;\;\;\partial_t\left[\int_{\mathbb{R}^3}dv \left(F^1+F^2\right)v_i\right]=0 &\; \;\;\;\;\;\;\;\;\;\;\;\;\;\;\;\text{ (momentum),}\\
\partial_t\left[\int_{\mathbb{R}^3} \left(F^1\arrowvert v\arrowvert^2+F^2\left(\arrowvert v\arrowvert^2+2\varepsilon_0\right)\right)dv\;+2\varepsilon_0\int_{\mathbb{S}^2}Q\;dn\right]=0&\; \;\;\;\;\;\;\;\;\;\;\;\;\;\;\;\text{ (energy).}\\
\end{split}
\end{equation}
Finally, we can also define an entropy.
\begin{definition}\label{entropydef}
	We define for $ \mathbb{F} $ a solution to the kinetic equation \pef{kinsystem} the entropy at time $ t<\infty $ by
	\begin{equation}\label{entropy}
	\mathcal{H}(t):=\int_{\mathbb{R}^3} dv \; \left(F^1\log(F^1)+F^2\log(F^2)\right)+\int_{\mathbb{S}^2}dn \;\left(Q\log(Q)-\left(\frac{A_0}{B_0}+Q\right)\log\left(\frac{A_0}{B_0}+Q\right)\right).
	\end{equation}
\end{definition}
The entropy and its dissipation are important tools for the development of the well-posedness theory. We will be able to define them also for the new equation we will derive later.	
\subsection{Main results of the paper}
We conclude the introduction with the following two theorems, which are the main results of this paper. The first one is about the well-posedness theory for the kinetic system \pef{newsystem}.
\begin{theorem}\label{infinityexistence}
	Let $ (\lambda_0,F_0)\in [0,1)\times L^1_2\left(\mathbb{R}^3\right) $ with $ F_0\log(F_0)\in L^1\left(\mathbb{R}^3\right) $ and $ F_0\in L^1_4\left(\mathbb{R}^3\right) $. Let the initial energy $ E_0 $ be finite. Assume $ (\lambda_0,F_0) $ satisfies $ F_0(v)\geq Ce^{-\beta\arrowvert v\arrowvert^2} $. Let $ T>0 $ then there exists a weak solution $ (\lambda, F)\in C\left([0,T], \mathbb{R}\times L^1_2\left(\mathbb{R}^3\right)\right)$ with $ \lambda(t)\in[0,1) $ of the kinetic system \pef{newsystem} for the initial value $ (\lambda_0, F_0) $. This is the unique solution with $ F\in L^{\infty}\left([0,T], L^1_4\left(\mathbb{R}^3\right)\right) $. 
\end{theorem}

We will prove also rigorously that the kinetic system \pef{newsystem} really describes the behavior of the solutions to \pef{epsequation} as $ \eps\to0 $.

\begin{theorem}\label{thmrigorous}
	Let $ \eps>0 $ and $ T>0 $. Let $ \FF_0 $ defined by $ (\lambda_0,F_0) $ and with $ \FF_0>0 $ and $ F_0(v)\geq Ce^{-\beta\arrowvert v\arrowvert^2} $. Assume it has mass $ \kappa_0>0 $, energy $ E_0:=E\left(\FF_0\right)\leq \frac{1}{2} $ and $ F_0\in L^1_4 $.
	
	Let $ \delta<\delta_0 $ small enough for some $ \delta_0 $ depending on the initial value $ \FF_0 $. We define $ \tFF_0=\FF_0+W_0$ such that $ \tFF_0>0 $ and $ \Arrowvert W_0\Arrowvert<\delta $. 
	
	Let $ \FF^\eps\in C^1\left([0,\infty),L^1_2\left(\mathbb{R}^3\right)\times L^1_2\left(\mathbb{R}^3\right)\times L^1\left(\mathbb{S}^2\right)\right) $ be the solution to the equation \pef{epsequation} for the initial value $ \tFF_0 $. 
	
	Let $ \oFF $ be defined by $ (\ol,\oF)\in C\left([0,\infty), \mathbb{R}\times L^1_2\left(\mathbb{R}^3\right)\right) $, which are the solution of the derived equation \pef{newsystem} for the initial value $ \FF_0 $.
	
	Then for all $ s>0 $ we have
	\begin{equation}
	\lim\limits_{\eps\to 0}\sup_{s\leq t\leq T}\Arrowvert \FF^\eps-\oFF\Arrowvert_\mathcal{X}= 0.
	\end{equation}
\end{theorem}

\subsection{A summary of mathematical models for photon-gas systems}
The behavior of a gas interacting with photons is besides its mathematical interest it is also relevant in astrophysical applications. The radiative transfer equation is the main tool for the study of this problem. A detailed explanation of its derivation and its properties can be found in \cite{Chandrasekhar, mihalas,oxenius,Rutten}.

The first studies of the interaction between gas and radiation can be found in \cite{compton, Milne}. These models describe the diffusion of radiation through a gas which is trapping it. However, they do not consider collisions between gas molecules. The only interactions taken into account are the absorption and the successive emission of photons by the gas molecules.

Some papers which have studied the radiative transfer equation interacting with matter in a rigorous mathematical way are the following. In \cite{Bardos, rossanipolew} the interaction between a fluid system and radiation has been considered.
Models close to the spirit of this paper, where the interaction is due to collisions and radiative processes, are developed and discussed in \cite{Burgers,Nouri2,Nouri1,rossanispiga}. In particular in \cite{Nouri2,Nouri1} the author studied also the well-posedness theory for a kinetic equation describing a system where the collisions between the gas molecules are only elastic and where in addition to the radiative process \pef{rad} the photon-gas interaction can be also of the form $ \overline{A} + \gamma \rightarrow A + 2\gamma $. The author of \cite{Burgers} develops a model which considers only two-levels gas molecules and monochromatic radiation, which are assumptions we also make in this paper. In the model in \cite{rossanispiga} all three interactions \pef{elcoll}, \pef{nonelcoll} and \pef{rad} are included. 

In \cite{paper} the authors studied the same model we will work with. They studied different scaling limits yielding either LTE (local thermal equilibrium) or Non-LTE situations and they developed via a hydrodynamic limit Euler equations coupled with the radiative transfer equation.

\subsection{Outline of the paper}   
In Section 2 we formally derive equation \pef{newsystem}. Section 3 deals with the well-posedness theory for this equation, while in Section 4 the rigorous proof for the derivation of equation \pef{newsystem} will be shown. We postpone until Section 4 the proof of Theorem \ref{thmrigorous} since it requires the well-posedness theory developed in Section 3.
\section{The fast radiation limit}

It is possible to develop a good well-posedness theory for the kinetic equation \pef{kinsystem} when the collisions between gas molecules and the radiation processes are of the same order. In this paper we aim to study the case when the collision terms are negligible compared to the terms describing the interaction between the gas and the radiation. In this Section we will formally derive the new kinetic equation as it is stated in \pef{newsystem}.

\subsection{The behavior of the solutions}
We recall the equation we stated in \pef{epsequation}, which describes the behavior of the gas and the photons shortly after the initial time. After times of order $ \eps $, we can expect the distribution of gas molecules and radiation to solve  
\begin{equation}
\begin{split}
0&=\int_{\mathbb{S}^2} \left[F^2(v) +Q(n)(F^2(v)-F^1(v))\right]\;dn\\
0&=\int_{\mathbb{R}^3}\left[F^2(v) +Q(n)(F^2(v)-F^1(v))\right]\;dv
\end{split}
\end{equation}
Hence, $ Q $ is independent of $ n $ with $ Q(n)=\frac{\int_{\mathbb{R}^3} F^2(v)\;dv}{\int_{\mathbb{R}^3}\left[ F^1(v)-F^2(v)\;dv\right]} $. This is possible only assuming \linebreak$ \int_{\mathbb{R}^3} F^2(v)\;dv<\int_{\mathbb{R}^3} F^1(v)\;dv $, since $ Q $ is non-negative. Moreover, substituting this photon number density into the first equation we see
\begin{equation*}
F^2(v)=\frac{\int_{\mathbb{S}^2} Q(n)\;dn}{\int_{\mathbb{S}^2} \left[1+ Q(n)\;dn\right]}F^1(v)=\frac{\int_{\mathbb{R}^3} F^2(v)\;dv}{\int_{\mathbb{R}^3} F^1(v)\;dv}F^1(v).
\end{equation*}
Hence we have to look for solutions living in the manifold given by
\begin{equation}\label{manifold}
\begin{cases}
\overline{F}^2=\lambda\overline{F}^1\\
\overline{Q}=\frac{\lambda}{1-\lambda}
\end{cases},
\end{equation}
where $\overline{F}^1\geq 0 $, $ \overline{F}^1\in L_2^1\left(\mathbb{R}^3\right) $ and $ \lambda\in[0,1) $ for all $ t\geq 0 $.
In order to simplify the notation we define the manifold of steady states $ \mathcal{M} $ as follows.
\begin{definition}\label{manifold1}
	We call $ \mathcal{M}=\left\{\FF\in\mathcal{X}: \FF= \begin{pmatrix}F\\\lambda F\\\frac{\lambda}{1-\lambda}\end{pmatrix}, \lambda<1\right\}  $ and for the non-negative $ F $ we denote \linebreak$ \mathcal{M}_+=\left\{ \FF\in\mathcal{M}: F\geq 0, \lambda\in[0,1)\right\} $.
\end{definition}
$ \mathcal{M}_+ $ plays the role of the so called critical manifold for the perturbative problem considered in \cite{spohn}, but we will call it the manifold of steady states.

We look for solutions of \pef{epsequation} that remain close to the manifold $ \mathcal{M}_+ $ for arbitrary times in powers of $ \eps $. Imposing suitable compatibility conditions we will obtain a kinetic equation which describes the evolution of functions in $ \mathcal{M}_+ $. The procedure is reminiscent to the classical Chapman-Enskog expansion which yields to the hydrodynamic limits of the Boltzmann equation. We call the entire approach a generalized Chapman-Enskog expansion. We will hence study the new kinetic system, its entropy and we will develop a weak well-posedness theory for it. We will then prove Theorem \ref{thmrigorous}.

\subsection{The generalized Chapman-Enskog expansion}
Setting $ \overline{F}^i $, $ \lambda $ and $ \overline{Q} $ as in \pef{manifold}, we look for a perturbative expansion for small variation of these functions. Therefore taking $ H(t,n) $, $ G^i(t,v) $ we define
\begin{equation}\label{expansion1}
\begin{split}
F^1&=\overline{F}^1\left(1+\varepsilon G^1\right),\\
F^2&=\overline{F}^2\left(1+\varepsilon G^2\right),\\
Q&=\frac{\lambda}{1-\lambda}\left(1+\varepsilon H\right).
\end{split}
\end{equation}
Putting \pef{expansion1} into the kinetic system \pef{kinsystem} and neglecting all terms of order $ \varepsilon $ or small we obtain
\begin{equation}\label{expansion2}
\begin{split}
\partial_t \overline{F}^1&=\mathbb{K}_1\left[\oFF,\oFF\right]+\int_{\mathbb{S}^2}\left[\frac{\lambda}{1-\lambda} \overline{F}^1 \left(G^2-G^1\right)-\lambda H \overline{F}^1\right]\;dn\\
\partial_t \overline{F}^2&=\mathbb{K}_2\left[\oFF,\oFF\right]-\int_{\mathbb{S}^2}\left[\frac{\lambda}{1-\lambda} \overline{F}^1 \left(G^2-G^1\right)-\lambda H \overline{F}^1\right]\;dn\\
\partial_t \frac{\lambda}{1-\lambda}&=\int_{\mathbb{R}^3}\left[\frac{\lambda}{1-\lambda} \overline{F}^1 \left(G^2-G^1\right)-\lambda H \overline{F}^1\right]\;dv\\
\end{split}
\end{equation}
We can neglect such terms because the collision terms and all the other functions are of order 1. The collision terms $ \mathbb{K}_i\left[\oFF,\oFF\right] $ are as the one defined in \pef{coll1} and \pef{coll2} for the function $ \oFF=\left(\overline{F}^1,\overline{F}^2,\overline{Q}\right) $. When the vector $ \mathbb{F} $ is an element of the steady manifold $ \mathcal{M}_+ $ we write $ \mathbb{K}_i\left[F,\lambda\right] $ in order to stress out the fact that $ F^1=F $ and $ F^2= \lambda F $.

Let us fix $ t\in [0,\infty) $. Let us consider the space $ \textbf{H}:= L^2\left(\mathbb{R}^3\right)\times L^2\left(\mathbb{R}^3\right)\times L^2\left(\mathbb{S}^2\right) $. Let us fix $ \lambda\in[0,1) $ and $ \overline{F}^1 $. For each $ \lambda$, $ \oF^1 $ we define an operator $ L_{\lambda, \oF^1} $ by means of
\begin{equation}\label{expansion3}
\begin{split}
L_{\lambda,\overline{F}^1}\;\;\; :\;\;\;&\textbf{H}\;\;\;\;\;\;\longrightarrow\;\;\;\;\;\;\;\;\;\;\;\;\;\;\;\;\;\;\;\;\; \textbf{H}\\
&\begin{pmatrix} G^1\\ G^2\\ H\end{pmatrix}\mapsto L_{\lambda,\oF^1}\begin{pmatrix} G^1\\ G^2\\ H\end{pmatrix}=\begin{pmatrix} L_1\left(G^1, G^2, H\right)\\ L_2\left(G^1, G^2, H\right)\\ L_3\left(G^1, G^2, H\right)\end{pmatrix}=\begin{pmatrix}\int_{\mathbb{S}^2}\left[\frac{\lambda}{1-\lambda} \overline{F}^1 \left(G^2-G^1\right)-\lambda H \overline{F}^1\right]\;dn \\-\int_{\mathbb{S}^2}\left[\frac{\lambda}{1-\lambda} \overline{F}^1 \left(G^2-G^1\right)-\lambda H \overline{F}^1\right]\;dn \\ \int_{\mathbb{R}^3}\left[\frac{\lambda}{1-\lambda} \overline{F}^1 \left(G^2-G^1\right)-\lambda H \overline{F}^1\right]\;dv\end{pmatrix}
\end{split}
\end{equation}
On the Hilbert space with inner product $ \left(\textbf{H}, \langle , \rangle_{\mathbb{R}^3}+\langle , \rangle_{\mathbb{R}^3}+ (1-\lambda)\langle , \rangle_{\mathbb{S}^2})\right) $ the operator $ L_{\lambda,\overline{F}^1} $ is self-adjoint, indeed for $\left( \tilde{G}^1,\tilde{G}^2,\tilde{H}\right)\in \textbf{H} $ an easy computation shows
\begin{equation}\label{expansion4}
\begin{split}
&\langle \tilde{G}^1, L_1(G^1,G^2,H)\rangle_{\mathbb{R}^3}+\langle \tilde{G}^2, L_2(G^1,G^2,H)\rangle_{\mathbb{R}^3}+(1-\lambda)\langle \tilde{H}, L_3(G^1,G^2,H)\rangle_{\mathbb{S}^2}\\
=&\langle L_1(\tilde{G}^1,\tilde{G}^2,\tilde{H}), G^1\rangle_{\mathbb{R}^3}+\langle L_2(\tilde{G}^1,\tilde{G}^2,\tilde{H}), G^2\rangle_{\mathbb{R}^3}+(1-\lambda)\langle L_3(\tilde{G}^1,\tilde{G}^2,\tilde{H}), H\rangle_{\mathbb{S}^2}.
\end{split}
\end{equation}
In order to find a new kinetic equation, we need to eliminate the operator $ L_{\lambda, \oF^1} $ in equation \pef{expansion2}. Hence, we start looking at the tangent space of $ \oFF:=\left(\oF^1,\;\oF^2,\; \overline{Q}\right) $ in the Hilbert space $ \textbf{H} $. To do that we consider $ F^i $ and $ \tilde{\lambda} $ given by small perturbation of these $ \overline{F}^1 $ and $ \lambda $ in the manifold.
\begin{equation}\label{expansion5}
\begin{cases}
F^1= \overline{F}^1(1+\xi_1)\\
F^2= \overline{F}^2(1+\xi_2)\\
Q= \frac{\lambda}{1-\lambda}(1+\eta) 
\end{cases}
\end{equation}
Defining $ \tilde{\lambda}=\frac{Q}{1+Q} $ and expanding with Taylor we see therefore that $ \tilde{\lambda}=\lambda(1+(1-\lambda)\eta) $. Hence imposing $ F^2=\tilde{\lambda}F^1 $ we obtain the condition for the tangent space: $ \xi_2=\xi_1+(1-\lambda)\eta $. Therefore we formally define the tangent space $ T_{\oFF}\mathcal{M} $ at $ \oFF $ by all vectors
\begin{equation}\label{tangentspace}
\begin{pmatrix}\xi\\\xi+(1-\lambda)\eta\\ \eta \end{pmatrix} \;\;\;\;\;\;\;\;\;\;\;\;\;\;\;\;\text{for} \;\;\eta\in\mathbb{R}\text{, }\xi\in L^2\left(\mathbb{R}^3\right).
\end{equation}

\begin{remark}
	Since the kernel of the operator $ L_{\lambda,\oF^1} $ are the vector $ \left(\xi_1,\;\xi_2,\;\eta\right)^\top $ satisfying the relation $ \xi_2=\xi_1+(1-\lambda)\eta $, we see that $  L_{\lambda, \overline{F}^1}\left(\xi_1,\xi_1+(1-\lambda)\eta, \eta \right)=0 $. Therefore testing the equation with these vectors we can eliminate the operator in equation \pef{expansion2}.\\
	
\end{remark}
It is interesting also to look at the orthogonal space of $ \oFF:=\left(\oF^1,\;\oF^2,\; \overline{Q}\right) $ in the Hilbert space $ \textbf{H} $. It is not difficult to see that, by the definition of inner-product and the structure of the tangent space, the orthogonal space to $ \mathcal{M}_+ $ is formally given by all vectors
\begin{equation}\label{orthogonalspace}
\begin{pmatrix}\alpha\\-\alpha\\ \theta \end{pmatrix} \;\;\;\;\;\;\;\;\;\;\text{such that} \;\;\int_{\mathbb{R}^3}\alpha\;dv=\int_{\mathbb{S}^2} \theta\;dn\;.
\end{equation}
The structure of the orthogonal space will play an important role for the proof of Theorem \ref{thmrigorous}. We will indeed see that any vector $ \FF>0 $ in a small neighborhood of the manifold $ \mathcal{M} $ can be decomposed as the sum of a vector in the manifold and a remainder, which has the same structure as the vectors in the orthogonal space.
\\

Let us now return to the formal derivation of the kinetic equation for $ \lambda $ and $ \oF^1 $. Let $ \xi_1\in C_c^\infty \left(\mathbb{R}^3\right)$ and $ \eta $ be a constant. Using the property that the self-adjoint operator $ L_{\lambda, \overline{F}^1} $ is zero on the tangent space $ T_{\oFF}\mathcal{M} $ we obtain that for any $ (G^1,G^2,H)\in \textbf{H} $
\begin{equation}\label{expansion6}
\left\langle \begin{pmatrix} \xi_1 \\\xi_1+(1-\lambda)\eta \\\eta\end{pmatrix}, L_{\lambda, \overline{F}^1}\begin{pmatrix} G^1\\ G^2\\ H\end{pmatrix}\right\rangle_\textbf{H}=0.
\end{equation}
Putting that into the initial equation \pef{expansion2} we obtain the following result
\begin{equation}\label{expansion7}
\left\langle \begin{pmatrix} \xi_1\\\xi_1+(1-\lambda)\eta \\\eta\end{pmatrix}, \begin{pmatrix} \partial_t \overline{F}^1\\ \partial_t \overline{F}^2\\ \partial_t \frac{\lambda}{1-\lambda}\end{pmatrix}\right\rangle_\textbf{H}=\left\langle \begin{pmatrix} \xi_1\\x_1+(1-\lambda)\eta \\\eta\end{pmatrix}, \begin{pmatrix} \mathbb{K}_1[\overline{F}^1,\lambda] \\\mathbb{K}_2[\overline{F}^1,\lambda]\\0\end{pmatrix}\right\rangle_\textbf{H}.
\end{equation}
From now on we define $ F:=\overline{F}^1 $.

Suppose now $ \eta=0 $. Keeping in mind that
\begin{equation*}
\int_{\mathbb{R}^3} dv\left(\partial_t (1+\lambda)F\right)\xi_1=\int_{\mathbb{R}^3}dv\; \xi_1\left(\mathbb{K}_1\left[F,\lambda\right]+\mathbb{K}_2\left[F,\lambda\right]\right)
\end{equation*} and since $ \xi_1\in C_c^\infty \left(\mathbb{R}^3\right) $ is arbitrary we conclude that
\begin{equation}\label{expansion8}
\partial_t \left(1+\lambda\right)F=\mathbb{K}_1\left[F,\lambda\right]+\mathbb{K}_2\left[F,\lambda\right].
\end{equation}
Similarly we now consider $ \xi=0 $ and $ \eta=\frac{1}{1-\lambda} $. Putting them in \pef{expansion7} we obtain
\begin{equation}\label{expansion9}
\int_{\mathbb{R}^3} \partial_t (\lambda F) \;dv \;+\partial_t \frac{\lambda}{1-\lambda}=\int_{\mathbb{R}^3} dv\; \mathbb{K}_2\left[F,\lambda\right].
\end{equation}
Using the weak formulation as in Section 1 we have $ \int_{\mathbb{R}^3} dv\left( \mathbb{K}_1\left[F,\lambda\right]+\mathbb{K}_2\left[F,\lambda\right]\right)=0  $, which implies \linebreak$  \int_{\mathbb{R}^3} dv\; \mathbb{K}_2\left[F,\lambda\right]=-\int_{\mathbb{R}^3}dv\; \mathbb{K}_1\left[F,\lambda\right] $. Using this and \pef{expansion8}, we can rewrite \pef{expansion9} as
\begin{equation}\label{expansion10}
\partial_t \frac{\lambda}{1-\lambda}=\partial_t\int_{\mathbb{R}^3} dv\; F -\int_{\mathbb{R}^3} dv\; \mathbb{K}_1\left[F,\lambda\right].
\end{equation}
We now rewrite \pef{expansion10} with the goal to find an equation for the derivative of $ \lambda $ only. Surely, since $ \partial_t \frac{\lambda}{1-\lambda}=\frac{\partial_t\lambda}{(1-\lambda)^2} $ we see that
\begin{equation}\label{expansion11}
\partial_t \lambda= (1-\lambda)^2 \int_{\mathbb{R}^3} \left(\partial_t F-\mathbb{K}_1\left[F,\lambda\right]\right)\;dv.
\end{equation}
By the multiplication rule we obtain also $ \partial_t \left(1+\lambda\right)F=(1+\lambda)\partial_tF+F\partial_t\lambda $. We integrate \pef{expansion8} 
\begin{equation}\label{expansion12}
\int_{\mathbb{R}^3} dv\; \partial_tF= \frac{1}{1+\lambda}\int_{\mathbb{R}^3} dv\;\left( \mathbb{K}_1\left[F,\lambda\right]+\mathbb{K}_2\left[F,\lambda\right]\right)-\frac{\partial_t \lambda}{1+\lambda}\int_{\mathbb{R}^3} dv\; F=-\frac{\partial_t \lambda}{1+\lambda}\int_{\mathbb{R}^3} dv\; F.
\end{equation}
Moreover, \pef{expansion12} in \pef{expansion11} implies 
\begin{equation}\label{expansion13}
\partial_t \lambda= -\frac{(1-\lambda)^2(1+\lambda)}{(1+\lambda)+(1-\lambda)^2\int_{\mathbb{R}^3}F}\int_{\mathbb{R}^3} \mathbb{K}_1\left[F,\lambda\right]dv.
\end{equation}
Substituting $ \partial_t \lambda $ in \pef{expansion8} with \pef{expansion13} we conclude the derivation of the equation for $ F $ and we summarize the new kinetic system we already defined in Section 1.3 and which we will study in the next Sections as follows
\begin{equation}\label{newsystem1}
\begin{split}
\partial_t \lambda&= -\frac{(1-\lambda)^2(1+\lambda)}{(1+\lambda)+(1-\lambda)^2\int_{\mathbb{R}^3}F}\int_{\mathbb{R}^3} \mathbb{K}_1\left[F,\lambda\right]dv\\
\partial_t F&= \frac{1}{1+\lambda}\left(\mathbb{K}_1\left[F,\lambda\right]+\mathbb{K}_2\left[F,\lambda\right]\right)+\frac{(1-\lambda)^2F}{(1+\lambda)+(1-\lambda)^2\int_{\mathbb{R}^3}F}\int_{\mathbb{R}^3} \mathbb{K}_1\left[F,\lambda\right]dv\\
\end{split}
\end{equation}
\subsection{Stability condition for the generalized Chapman-Enskog expansion}
We want now to briefly look at the stability condition for the generalized Chapman-Enskog expansion. To this end we start with the usual kinetic system as in \pef{kinsystem}. We expect $ F^1 $ to change in time $ t $ of order $ 1 $, hence it shall change very slowly on the time scale $ \eps t $ for small $ \eps>0 $. Due to this we assume $ F^1(v,t) $ to be constant in time. Moreover, according to the model we assume $ \mathbb{K}_2\left[F^1,F^2\right] $ to be very small and $ \int_{\mathbb{R}^3}dv\;F^1\geq\int_{\mathbb{R}^3}dv\;F^2 $. Multiplying \pef{kinsystem} by $ \varepsilon $ and defining $ \tau:=t/\varepsilon $ the system becomes
\begin{equation}\label{justification1}
\begin{split}
\partial_\tau F^2&=-\int_{\mathbb{S}^2} dn\;\left[ F^2+ Q\left(F^2-F^1\right)\right]\\
\partial_\tau Q&=\int_{\mathbb{R}^3}dv\;\left[F^2+ Q\left(F^2-F^1\right)\right]
\end{split}
\end{equation}
We would like to show the following proposition, which ensures the stability of the expansion we made in Section 2.2.
\begin{prop}
Let $ F^1 $ be constant in time and $ \int_{\mathbb{R}^3}dv\;F^1\geq\int_{\mathbb{R}^3}dv\;F^2 $. Let $ F^2 $ and $ Q $ satisfy equation \pef{justification1}. Then as $ \tau\to\infty $, the density $ Q $ approaches to a constant and the particle density $ F^2 $ satisfies $ F^2(v)=\frac{Q}{1+Q}F_1 $. 
\end{prop}
\begin{proof}
Integrating in $ \mathbb{R}^3 $ or in $ \mathbb{S}^2 $, \pef{justification1} becomes
\begin{equation}\label{justification2}
\begin{split}
\partial_\tau\rho_2&=-\rho_2\left(I+1\right)+I a\\
\partial_\tau I=& I(\rho_2-a)+\rho_2
\end{split}
\end{equation}
where we used the notation $ a:=\int_{\mathbb{R}^3}dv\;F^1(v)=\, $constant, $ \rho_2(\tau):=\int_{\mathbb{R}^3}dv\;F^2(v) $ and $ I(\tau):=\int_{\mathbb{S}^2}dn\;Q(n) $. We now show that they converge in the long time to a stability solution, so for $ \tau\to\infty $ we have the relation $ \rho_2(I+1)=Ia $, which is only the case when both $ \rho_2 $ and $ I $ become constants. 

First of all we see from the equation \pef{justification2} that $ c_0:=\rho_2+I $ is a constant with $ c_0<I+a $. Hence we rewrite the equation for $ I $ as 
\begin{equation}\label{justification3}
\partial_\tau I=I(c_0-I-a)+c_0-I=-I^2-(1-c_0+a)I+c_0.
\end{equation}
But this is a simple ODE, whose solution $ I\to-\frac{(1-c_0+a)-\sqrt{(1-c_0+a)^2+4c_0}}{2}:=I^\infty>0 $ and therefore also $ \rho_2=c_0-I\to c_0-I^\infty=\frac{I^\infty a}{I^\infty+1}=:\rho_2^\infty $ converges to a constant whenever $ \tau $ goes to infinity. This implies the claim. Indeed for $ \tau\to\infty $ we solve the equation $\partial_\tau Q=Q(\rho_2-a)+\rho_2 $ in the following way. Recall that $ \rho_2-a\leq0 $ and that $ \rho_2 $ is bounded. Since $ \rho_2\to\rho_2^\infty $ as $ t\to\infty $, there exists some $ T>0 $ such that $ \rho_2\leq\rho_2^\infty-\varepsilon $ for some small $ \varepsilon>0 $. For all these reasons we can compute
\begin{equation}\label{justification4}
\begin{split}
Q(t)&=Q_0\exp\left(\int_{0}^{t} \left(\rho_2(s)-a\right)\; ds\right)+\int_{0}^t	\exp\left(\int_{s}^{t} \left(\rho_2(\tau)-a\right)\; d\tau\right)\rho_2(s)ds\\
&\underset{t\to\infty}{\longrightarrow}\lim\limits_{t\to\infty}\int_{0}^tds\;e^{-\frac{a}{I^\infty+1}(t-s)}\frac{I^\infty a}{I^\infty+1}\\
&= I^\infty.
\end{split}
\end{equation}
A similar computation for the particle density $ F^2 $ shows the desired claim, since \linebreak$ \partial_\tau F^2=-(1+I)F^2+IF^1 $
\begin{equation}\label{justification5}
F^2(t,v)\underset{t\to\infty}{\longrightarrow}\lim\limits_{t\to\infty}\int_{0}^t ds\;e^{-(1+I^\infty)(t-s)}I^\infty F^1=\frac{I^\infty}{I^\infty+1}F^1=\lim\limits_{t\to\infty}\frac{Q}{Q+1}F^1.
\end{equation}

Hence we have shown the desired result.
\end{proof}
\section{Well-posedness theory for the fast radiation limit}
We anticipated that one can prove a well-posedness theory for equation \pef{kinsystem}. Indeed, following the standard proof as in \cite{fournier,cooling,mischler} one can prove existence and uniqueness of weak solutions for the kinetic system. On the other hand, by mean of the ODE's theory in general Banach spaces one can prove as for example in \cite{Bressan,Gamba} a strong well-posedness theory. Finally, the long time convergence to the unique equilibrium is achieved with the standard method as in \cite{arkerydI,arkerydII}. One can read these results in the first part of \cite{tesi}. In Section 3 we will prove similar results for the kinetic equation \pef{newsystem1} we formally derived in Section 2. We start the analysis by looking at its properties and giving some definitions. Afterwards we will develop a well-posedness theory, which we recall to be important for the proof of Theorem \ref{thmrigorous} in Section 4.
\subsection{Properties of the kinetic system}
The solutions for the equation \pef{newsystem1} satisfy some conservation laws. Indeed, let $ F(t,v)\geq 0 $ and $ \lambda(t) $ be solution with $F\in L_2^1\left(\mathbb{R}^3\right)$. Then putting $ \xi_1=1 $ and $ \eta=0 $ in \pef{expansion7} we obtain the conservation of mass for $ F $
\begin{equation}\label{newmass}
\partial_t \int_{\mathbb{R}^3} dv\; (1+\lambda)F=0.
\end{equation}
Here we also used the properties of the collision terms when integrated. Similarly taking $ \xi_1=v_i $ for $ i=1,2,3 $ we deduce the conservation of momentum
\begin{equation}\label{newmomentum}
\partial_t \int_{\mathbb{R}^3} dv\; (1+\lambda)Fv_i=0.
\end{equation}
Finally, we also have the conservation of energy in the following sense. Let us consider $ \xi_1=\arrowvert v\arrowvert^2 $ and $ \eta=\frac{2\varepsilon_0}{1-\lambda} $, then we have
\begin{equation}\label{newenergy}
\partial_t \left(\int_{\mathbb{R}^3} dv\left[ \left(\arrowvert v\arrowvert^2(1+\lambda)F\right)+2\varepsilon_0\lambda F\right]\right)+ 2\varepsilon_0\partial_t\frac{\lambda}{1-\lambda}=0.
\end{equation}

Clearly we shall work on a different space than $ \mathcal{X} $.
\begin{definition}
	We define the Banach space $ X:=\left\{\left(\lambda, F\right): \lambda\in\mathbb{R},\; F\in L_2^1\left(\mathbb{R}^3\right)\right\} $ and the subset \linebreak $  X_+:=\left\{\left(\lambda, F\right)\in X: \lambda\in[0,1),\; F\geq0\right\} $. We will also consider the subspace $ \tilde{X}:=\left\{\left(\lambda, F\right)\in X: F\in L_4^1\left(\mathbb{R}^3\right)\right\} $. 
\end{definition}

With these spaces we can define the following operator, that will simplify the notation.
\begin{definition}\label{defL}
	We define the operator $ T:\tilde{X}\to X $ by
	\begin{equation*}
	T\begin{pmatrix}\lambda \\ F\end{pmatrix}=\begin{pmatrix} -\frac{(1-\lambda)^2(1+\lambda)}{(1+\lambda)+(1-\lambda)^2\int_{\mathbb{R}^3}F}\int_{\mathbb{R}^3} \mathbb{K}_1\left[F,\lambda\right]dv \\ \frac{1}{1+\lambda}\left(\mathbb{K}_1\left[F,\lambda\right]+\mathbb{K}_2\left[F,\lambda\right]\right)+\frac{(1-\lambda)^2F}{(1+\lambda)+(1-\lambda)^2\int_{\mathbb{R}^3}F}\int_{\mathbb{R}^3} \mathbb{K}_1\left[F,\lambda\right]dv	\end{pmatrix}.
	\end{equation*}
\end{definition}

\begin{remark}
	With this operator we can rewrite the kinetic system \pef{newsystem1} as\begin{equation*}
	\begin{split}
	\partial_t \lambda &=T_1(\lambda, F)\\
	\partial_t F&= T_2(\lambda, F)
	\end{split}
	\end{equation*}
\end{remark}
Next we define the notion of weak solutions, energy and entropy.
\begin{definition}\label{solution}
	We call a pair $ \left(\lambda,F\right)\in C\left([0,\infty), X\right) $ a weak solution to \pef{newsystem1} with initial values \linebreak$ \left(\lambda_0, F_0\right)\in X_+ $ if  $ \left(\lambda,F\right)\in X_+ $ for all $ t\geq 0 $ and if
	\begin{equation*}
	\lambda(t)=\int_0^t T_1(\lambda, F) \;d\tau+\lambda_0
	\end{equation*}
	and if for all $ \varphi\in C_c^\infty\left(\mathbb{R}^3\right) $
	\begin{equation*}
	\partial_t\int_{\mathbb{R}^3} dv\; F(t,v)\varphi(v)=\int_{\mathbb{R}^3} T_2(\lambda, F)\varphi(v)\;dv.
	\end{equation*}
\end{definition}

\begin{definition}\label{energyentropy}
	Given a solution to \pef{newsystem1} $ (\lambda, F) $ in the sense of Definition \ref{solution} we define the energy $ E(t) $ by
	\begin{equation}
	E(t)=\int_{\mathbb{R}^3} dv\left[ \left(1+\arrowvert v\arrowvert^2(1+\lambda)F\right)+2\varepsilon_0\lambda F\right]+ 2\varepsilon_0\frac{\lambda}{1-\lambda}.
	\end{equation}
	We already know that $ E(t)=E(0)=\text{constant} $ for strong solutions. As in the kinetic system \pef{kinsystem}, this holds for weak solutions with bounded fourth moment too. We also define the entropy of the new system \pef{newsystem1} by
	\begin{equation}\label{newentropy}
	\mathcal{H}(t)=\int_{\mathbb{R}^3} \; dv\left[(1+\lambda)F\log(F)+F\lambda\log(\lambda)\right]+\frac{\lambda
	}{1-\lambda}\log(\lambda)+\log(1-\lambda).
	\end{equation}
\end{definition}
One of the ways in which we will use the conservation of energy is in the proof of the boundedness of $ \lambda $. Indeed it provides $ 0\leq \lambda(t)\leq \lambda_0\left(E(0)\right)<1 $. The entropy can be derived from the Definition \ref{entropydef} under the assumption $ \frac{A_0}{B_0}=1 $. Moreover, it decreases, as we will show in Proposition \ref{prop1}.
\\

We will assume throughout this Section that the initial data satisfies
\begin{equation}\label{Assumption}
F_0(v)\geq Ce^{-\beta\arrowvert v\arrowvert^2}
\end{equation}
for some constant $ C>0 $ and $ \beta>0 $. This is a strong assumption, but it allows us to have a global well-posedness without using too many technical arguments. In this paper we will not focus on the most general well-posedness theory, but on the derivation of the effective kinetic equation \pef{newsystem1} for the fast radiation limit and on the rigorous proof of Theorem \ref{thmrigorous} in Section 4, for which a well-posedness theory is required.

With assumption \pef{Assumption} we can estimate for any solution to \pef{newsystem1} which satisfies $ \Arrowvert F\Arrowvert_{L_2^1}\leq E(0) $ that
\begin{equation}\label{Assumption1}
\begin{split}
\partial_t F(v)\geq& -(1+\lambda)\left(\int_{\mathbb{R}^3} dv_2 \int_{\mathbb{S}^2} d\omega\; B_{el}(v,v_2)F(v_2)\right)F(v)\\
& -\frac{2}{1+\lambda}\left(\int_{\mathbb{R}^3} d\ov_2 \int_{\mathbb{S}^2} d\omega\; B_{ne}^{12}(v,\ov_2)F(\ov_2)\right)F(v)\\
& -\frac{2\lambda}{1+\lambda}\left(\int_{\mathbb{R}^3} d\ov_3 \int_{\mathbb{S}^2} d\omega\; B_{ne}^{34}(\ov_3,v)F(\ov_3)\right)F(v)\\
\;\; & -C(E(0), \varepsilon_0)F(v)\\
\geq&-C(E(0), \varepsilon_0)\left(1+\arrowvert v\arrowvert\right)F(v),
\end{split}
\end{equation}
where we used the properties $ 0\leq\frac{(1-\lambda)^2}{(1+\lambda)+(1-\lambda)^2\int_{\mathbb{R}^3}F}\leq1 $ and
\begin{equation}\label{Assumption2}
\begin{split}
&\int_{\mathbb{R}^3}dv_1\int_{\mathbb{R}^3} dv_2 \int_{\mathbb{S}^2} d\omega\; (1+\lambda)B_{el}(v_1,v_2)F(v_2)F(v_1)\\
+&2\int_{\mathbb{R}^3}d\ov_1\int_{\mathbb{R}^3} d\ov_2 \int_{\mathbb{S}^2} d\omega\; B_{ne}^{12}(\ov_1,\ov_2)F(\ov_2)(\ov_2)\\
+&\lambda\int_{\mathbb{R}^3}d\ov_4\int_{\mathbb{R}^3} d\ov_3 \int_{\mathbb{S}^2} d\omega\; B_{ne}^{34}(\ov_3,\ov_4)F(\ov_3)F(\ov_4)\\
\leq&C(E(0), \varepsilon_0)
\end{split}
\end{equation}
coming from the estimates \pef{usual estimate 3}, \pef{usual estimate}, \pef{usual estimate 2} for the collisional kernels.

The estimate \pef{Assumption2} implies that we have the following estimate for the solution $ F $
\begin{equation}\label{Assumption4}
F(t,v)\geq Ce^{-\beta \arrowvert v\arrowvert^2-C(E(0),\varepsilon_0)\left(1+\arrowvert v\arrowvert t\right)}
\end{equation}
and so we conclude that for every $ t\geq 0 $ there exists some constant $ \tilde{\beta}>0 $ and $ \tilde{C} $
\begin{equation}\label{Assumption5}
F(t,v)\geq \tilde{C}e^{-\tilde{\beta} \arrowvert v\arrowvert^2}.
\end{equation}
This assumption permits us to justify the formal calculation for the time derivative of the entropy. Moreover, it shows that if the initial value satisfies $ F_0>0 $ the also for the solution holds $ F>0 $.
\begin{prop}\label{prop1}
	Assume that the initial value $ (\lambda_0, F_0) $ satisfies the assumption \pef{Assumption} and that its energy is bounded, i.e $ E(0)<\infty $. If $ F $ has bounded fourth moment, the solution $ (\lambda, F) $ satisfies $ \lambda(t)\in[0,1) $ for all $ t\geq 0 $. Moreover if for a strong solution $ \Arrowvert F\Arrowvert_{L_4^1}\leq \kappa $ for some $ 0<\kappa<\infty $ and for all $ t\geq 0 $, then $ \partial_t\mathcal{H}(t)\leq 0 $.
	\begin{proof}
		We start with the first claim: the well-definedness of $ \lambda $. Using the definition of energy and its conservation we see
		\begin{equation*}
		\frac{\lambda}{1-\lambda}\leq \frac{E(t)}{2\varepsilon_0}=\frac{E(0)}{2\varepsilon_0}
		\end{equation*}
		This implies that $ \lambda(t)<1 $ for all $ t $, since $ \lambda\leq\frac{E(0)}{2\varepsilon_0+E(0)}<1 $.
		
		Assuming now $ \lambda(t_0)=0 $ for some $ t_0\in[0,1) $, the equation \pef{newsystem1} implies 
		\begin{equation*}
		\begin{split}
		\partial_t \lambda(t_0)&=-\frac{1}{1+\int_{\mathbb{R}^3}F}\int_{\mathbb{R}^3} \mathbb{K}_1\left[F,0\right]dv\\
		&=-\frac{1}{1+\int_{\mathbb{R}^3}F}\int_{\mathbb{R}^{12}} d\ov_1 d\ov_2 d\ov_3 d\ov_4 \delta(\ov_1+\ov_2-\ov_3-\ov_4)\delta(\arrowvert\ov_1\arrowvert^2+\arrowvert\ov_2\arrowvert^2-\arrowvert\ov_3\arrowvert^2-\arrowvert\ov_4\arrowvert^2-2\varepsilon_0)\\
		&\;\;\;\;\;\;\;\;\;\;\;\;\;\;\;\;\;\;\;\;\;\;\;\;\;\;\;W_{ne}(\ov_1,\ov_2;\ov_3,\ov_4) \left(-F(\ov_2)F(\ov_1)\right)\\
		&\geq 0.
		\end{split}
		\end{equation*}
		This conclude the first claim, indeed we have just seen that $ \lambda(t)\in \left[0,\frac{E(0)}{2\varepsilon_0+E(0)}\right] $ for all $ t $.
		
		For the second claim, we shall simply differentiate the entropy and using the kinetic system \pef{newsystem1} we can conclude that it is non-increasing.
		\begin{equation}\label{3.11}
		\begin{split}
		\partial_t \mathcal{H}(t)=&\int_{\mathbb{R}^3} \; dv \left[(1+\lambda)\log(F)\partial_tF+ F\log(F)\partial_t\lambda+(1+\lambda)\partial_tF+\lambda\log(\lambda)\partial_tF+F\log(\lambda)\partial_t \lambda+F\partial_t\lambda\right]\\
		&\;\;\;\;+\frac{\partial_t\lambda
		}{(1-\lambda)^2}\log(\lambda)+\frac{\partial_t\lambda}{1-\lambda}-\frac{\partial_t\lambda}{1-\lambda}\\
		=&\int_{\mathbb{R}^3} \; dv \left[\partial_t\left((1+\lambda)F\right)+\partial_t\left((1+\lambda)F\right) \log(F)+\partial_t\left(\lambda F\right)\log(\lambda)\right]+\log(\lambda)\partial_t\frac{\lambda}{1-\lambda}\\
		=&\int_{\mathbb{R}^3} \; dv \left(\mathbb{K}_1\left[F,\lambda\right]+\mathbb{K}_2\left[F,\lambda\right]\right)\log(F)+\int_{\mathbb{R}^3} \; dv\, \mathbb{K}_2\left[F,\lambda\right]\log(\lambda)\\
		=&\frac{1}{4}\int_{\mathbb{R}^3}dv_1\int_{\mathbb{R}^3}dv_2\;\int_{\mathbb{S}^2}d\omega\;B_{el}(v_1,v_2)\left(1+\lambda\right)^2 \left(F(v_3)F(v_4)-F(v_1)F(v_2)\right)\log\left(\frac{F(v_1)F(v_2)}{F(v_3)F(v_4)}\right)\\
		&+\int_{\mathbb{R}^3}d\ov_1\int_{\mathbb{R}^3}d\ov_2\;\int_{\mathbb{S}^2}d\omega\;B^{12}_{ne}(\ov_1,\ov_2)\left(\lambda F(\ov_3)F(\ov_4)-F(\ov_1)F(\ov_2)\right)\log\left(\frac{F(\ov_1)F(\ov_2)}{\lambda F(\ov_3)F(\ov_4)}\right)\\	
		\leq 0	,
		\end{split}
		\end{equation}
		where we used that the integral of the sum of the collision terms is zero and equation \pef{expansion9}, the weak formulation of the kinetic equation and the symmetry for the elastic collision. Now the inequality
		\begin{equation*}
		(z-y)\log\left(\frac{y}{z}\right)\leq 0\;\;\;\;\;\;\forall z,y>0
		\end{equation*}
		implies that each integrand  on the right hand side of \pef{3.11} is negative, which implies that the entropy is decreasing.
		
		In order to justify more rigorously this calculation  we remark that $ \Arrowvert F\Arrowvert_{L_4^1}\leq \kappa $ for every $ t\geq 0 $ and that for every collisional kernel $ B\in \{B_{el}, B_{ne}^{ij}\} $ there exists some constant $ C>0 $ such that \linebreak$ 0\leq B\leq C\left( 1+ \arrowvert v\arrowvert+\arrowvert w\arrowvert\right) $. Hence we know that for every $ i=1,2 $ we have that $ \mathbb{K}_i\left[F,\lambda\right]~\in~ L_3^1\left(\mathbb{R}^3\right) $, which implies that $ T_1\left(\lambda, F\right)\in L^\infty $ and $ T_2\left(\lambda, F\right)\in L^1\left(\mathbb{R}^3\right) $. Given the assumption \pef{Assumption} we know that $ \arrowvert \log\left(F\right)\arrowvert\leq C \left( \arrowvert v\arrowvert^2+1\right) $ for some $ C $ and for all $ 0\leq s\leq t $. Therefore using the Leibniz rule we can compute for all $ \varepsilon>0 $
		\begin{equation}\label{assumption6}
		\partial_t F\log\left(F+\varepsilon\right)= T_2\left(\lambda, F\right)\log\left(F+\varepsilon\right)+\frac{F}{F+\varepsilon}T_2\left(\lambda, F\right)
		\end{equation}
		and also
		\begin{equation}\label{assumption8}
		\partial_t \lambda F\log\left( \lambda F+\varepsilon\right)= \left[FT_1\left(\lambda, F\right)+\lambda T_2\left(\lambda, F\right) \right]\log\left(\lambda F+\varepsilon\right)+\frac{\lambda F}{\lambda F+\varepsilon}\left[FT_1\left(\lambda, F\right)+\lambda T_2\left(\lambda, F\right) \right].
		\end{equation}	
		These imply using Fubini that for any fixed $ t$
		\begin{equation}\label{assumption7}
		\begin{split}
		\int_{\mathbb{R}^3} dv\;& F(t,v)\log\left(F+\varepsilon\right)(t,v)=\int_{\mathbb{R}^3} dv\; F_0(v)\log\left(F_0+\varepsilon\right)(v)\\	
		&+\int_{0}^tds\int_{\mathbb{R}^3}dv\;T_2\left(\lambda, F\right)(s,v)\log\left(F+\varepsilon\right)(s,v)
		+\frac{F(s,v)}{F(s,v)+\varepsilon} T_2\left(\lambda, F\right)(s,v)
		\end{split}
		\end{equation}
		and similarly
		\begin{equation}\label{assumption9}
		\begin{split}
		&\int_{\mathbb{R}^3} dv\; \lambda F(t,v)\log\left(\lambda F+\varepsilon\right)(t,v)=\int_{\mathbb{R}^3} dv\; \lambda_0F_0(v)\log\left(\lambda_0F_0+\varepsilon\right)(v)\\
		&+\int_{0}^tds\int_{\mathbb{R}^3}dv\;\bigg(\left[FT_1\left(\lambda, F\right)+\lambda T_2\left(\lambda, F\right) \right](s,v)\log\left(\lambda F+\varepsilon\right)(s,v)\\
		&\;\;\;\;\;\;\;\;\;\;\;\;\;\;\;\;\;\;\;\;\;\;\;\;\;\;\;\;\;+\frac{\lambda F(s,v)}{\lambda F(s,v)+\varepsilon}\left[FT_1\left(\lambda, F\right)+\lambda T_2\left(\lambda, F\right) \right](s,v)\bigg).
		\end{split}
		\end{equation}
		Here everything is uniformly integrable, and therefore letting then $ \varepsilon\to 0 $ we conclude the formula in \pef{3.11}, since moreover the function $ f(x)=\frac{x}{1-x}\log(x)+\log(1-x) $ is already well-defined and differentiable for $ x\in[0,1) $.\\
	\end{proof}
\end{prop}

We proceed now with the well-posedness theory. We will adapt here the standard method for the classical Boltzmann equation. We will use a weakly convergent sequence of approximating solutions. This is the method used in the seminal paper \cite{arkerydI}.
\subsection{The cut-off kinetic system}
We define some auxiliary cut-off kernels. We define for $ n\in\mathbb{N} $
\begin{equation}\label{cutoff1}
B_{el,n}(v,w):=\min\left(\arrowvert v-w\arrowvert,n\right),
\end{equation}
\begin{equation}\label{cutoff2}
B_{ne,n}^{12}(\ov_1,\ov_2):=B_{ne}^{12}(\ov_1,\ov_2)\rchi_{\arrowvert\ov_1-\ov_2\arrowvert\leq n}\left(\ov_1,\ov_2\right),
\end{equation}
\begin{equation}\label{cutoff3}
B_{ne,n}^{34}(\ov_3,\ov_4):=B_{ne}^{34}(\ov_3,\ov_4)\rchi_{\arrowvert\ov_3-\ov_4\arrowvert\leq \sqrt{n^2-4\varepsilon_0}}\left(\ov_3,\ov_4\right).
\end{equation}

It is not difficult to see that the following is true.
\begin{prop}\label{propcutoff}
	For the velocities of the nonelastic collisions the following is equivalent\begin{equation*}
	\arrowvert\ov_1-\ov_2\arrowvert\leq n\iff\arrowvert\ov_3-\ov_4\arrowvert\leq\sqrt{n^2-4\varepsilon_0} 
	\end{equation*}
\begin{proof}
We recall the definition of pre- and post-collisional nonelastic velocities
\begin{equation*}
\ov_{3,4}=\frac{\ov_1+\ov_2}{2}\pm\omega\sqrt{\frac{\arrowvert \ov_1-\ov_2\arrowvert^2}{4}-\varepsilon_0},
\end{equation*}
which implies that $ \arrowvert\ov_3-\ov_4\arrowvert=\sqrt{\arrowvert \ov_1-\ov_2\arrowvert^2-4\varepsilon_0} $. Moreover we also have the relation
\begin{equation*}
\ov_{1,2}=\frac{\ov_3+\ov_4}{2}\pm\omega'\sqrt{\frac{\arrowvert \ov_3-\ov_4\arrowvert^2}{4}+\varepsilon_0},
\end{equation*}
which again implies $ \arrowvert\ov_1-\ov_2\arrowvert=\sqrt{\arrowvert \ov_3-\ov_4\arrowvert^2+4\varepsilon_0} $.
Hence considering $ \arrowvert\ov_1-\ov_2\arrowvert\leq n $ respectively \linebreak$ \arrowvert\ov_3-\ov_4\arrowvert\leq\sqrt{n^2-4\varepsilon_0} $ in the above relation we conclude the proof.
\end{proof}
\end{prop} 
An important implication of Proposition \ref{propcutoff} is that the following sets are the same
\begin{equation}\label{cutoff4}
\begin{split}
&\left\{(v_1,v_2,v_3,v_4)\in\mathbb{R}^{12}: \arrowvert v_1- v_2\arrowvert\leq n,\;v_{3,4}=\frac{v_1+v_2}{2}\pm\omega\sqrt{\frac{\arrowvert v_1-v_2\arrowvert^2}{4}-\varepsilon_0},\;\omega\in\mathbb{S}^2 \right\}\\
=&\left\{(v_1,v_2,v_3,v_4)\in\mathbb{R}^{12}: \arrowvert v_3- v_4\arrowvert\leq \sqrt{n^2-4\varepsilon_0},\;v_{1,2}=\frac{v_3+v_4}{2}\pm\omega'\sqrt{\frac{\arrowvert v_3-v_4\arrowvert^2}{4}+\varepsilon_0},\;\omega'\in\mathbb{S}^2 \right\}\\
:=&\Omega_n.
\end{split}
\end{equation}

We can now define the cut-off collisions term $ \mathbb{K}_{i,n} $ for $ i=1,2 $ replacing of the usual hard sphere's kernels with the cut-off kernel as in \pef{cutoff1}, \pef{cutoff2} and \pef{cutoff3}. We also define a new operator $ T_n:\tilde{X}\to X $ by
\begin{equation}\label{T_n}
T_n\begin{pmatrix}\lambda \\F\end{pmatrix}=\begin{pmatrix}-\frac{(1-\lambda)^2(1+\lambda)}{(1+\lambda)+(1-\lambda)^2\int_{\mathbb{R}^3}F}\int_{\mathbb{R}^3} \mathbb{K}_{1,n}\left[F,\lambda\right]dv\\ \frac{1}{1+\lambda}\left(\mathbb{K}_{1,n}\left[F,\lambda\right]+\mathbb{K}_{2,n}\left[F,\lambda\right]\right)+\frac{(1-\lambda)^2F}{(1+\lambda)+(1-\lambda)^2\int_{\mathbb{R}^3}F}\int_{\mathbb{R}^3} \mathbb{K}_{1,n}\left[F,\lambda\right]dv \end{pmatrix}.
\end{equation}
\begin{lemma}\label{lemcutoff1}
A solution $ (\lambda_n, F_n)\in C^1\left([0,1), X\right) $ of the equation $ \partial_t\begin{pmatrix}\lambda_n \\F_n\end{pmatrix}=T_n\begin{pmatrix}\lambda_n\\F_n\end{pmatrix} $ still satisfies the conservation of mass, momentum and energy as written in \pef{newmass}, \pef{newmomentum} and \pef{newenergy}.
\begin{proof}
We first find a formulation for $ \int_{\mathbb{R}^3}\;dv\;\left[ \mathbb{K}_{1,n}\left[F_n,\lambda_n\right]\varphi^1(v)+\mathbb{K}_{2,n}\left[F_n,\lambda_n\right]\varphi^2(v)\right] $. For the elastic collision terms, it is not difficult to see using the symmetry of the kernels with respect to pre- and post-collisional velocities that
\begin{equation*}
\begin{split}
&\int_{\mathbb{R}^3}\;dv\; \mathbb{K}_{1,n}^{elastic}\left[F_n,\lambda_n\right]\varphi^1(v)\\
=&(1+\lambda_n)\int_{\mathbb{R}^3} dv_1 \int_{\mathbb{R}^3}\;dv_2\int_{\mathbb{S}^2} d\omega\; \min(\arrowvert v_1-v_2\arrowvert,n)\left(F_n(v_3)F_n(v_4)-F_n(v_1)F_n(v_2)\right)\varphi^1(v_1)\\
=&\frac{(1+\lambda_n)}{2}\int_{\mathbb{R}^3} dv_1 \int_{\mathbb{R}^3}\;dv_2\int_{\mathbb{S}^2} d\omega\;\min(\arrowvert v_1-v_2\arrowvert,n)F_n(v)F_n(v_2)\left(\varphi^1(v_3)+\varphi^1(v_4)-\varphi^1(v_1)-\varphi^1(v_2)\right)
\end{split}
\end{equation*}
and also
\begin{equation*}
\begin{split}
\int_{\mathbb{R}^3}\;dv\; \mathbb{K}_{2,n}^{elastic}\left[F_n,\lambda_n\right]\varphi^2(v)=&\frac{(1+\lambda_n)\lambda_n}{2}\int_{\mathbb{R}^3} dv_1 \int_{\mathbb{R}^3}\;dv_2\int_{\mathbb{S}^2} d\omega\;\min(\arrowvert v_1-v_2\arrowvert,n)\\
&\:\:\:\:\:\:\:\:\:\:\:\:\:\:\:\:\:\:\:\:\:\:\:\:F_n(v_1)F_n(v_2)\left(\varphi^2(v_3)+\varphi^2(v_4)-\varphi^2(v_1)-\varphi^2(v_2)\right).
\end{split}
\end{equation*}
Using Proposition \ref{propcutoff}, we see that for the nonelastic collision the weak formulation similar to the one in \pef{weak} applies, indeed
\begin{equation}\label{nonelasticcutoff4}
\begin{split}
&\int_{\mathbb{R}^3}\;dv\;\left[ \mathbb{K}_{1,n}^{nonelastic}\left[F_n,\lambda_n\right]\varphi^1(v)+\mathbb{K}_{2,n}^{nonelastic}\left[F_n,\lambda_n\right]\varphi^2(v)\right]\\
=&2\int_{\mathbb{R}^3} d\ov_1 \int_{\mathbb{R}^3}\;d\ov_2\int_{\mathbb{S}^2} d\omega\;B_{ne,n}^{12}\left(\lambda_nF_n(\ov_3)F_n(\ov_4)-F_n(\ov_1)F_n(\ov_2)\right)\varphi^1(\ov_1)\\
&+\int_{\mathbb{R}^3} d\ov_3 \int_{\mathbb{R}^3}\;d\ov_4\int_{\mathbb{S}^2} d\omega\;B_{ne,n}^{34}\left(F_n(\ov_1)F_n(\ov_2)-\lambda_nF_n(\ov_3)F_n(\ov_4)\right)\varphi^1(\ov_4)\\
&+\int_{\mathbb{R}^3} d\ov_3 \int_{\mathbb{R}^3}\;d\ov_4\int_{\mathbb{S}^2} d\omega\;B_{ne,n}^{34}\left(F_n(\ov_1)F_n(\ov_2)-\lambda_nF_n(\ov_3)F_n(\ov_4)\right)\varphi^2(\ov_3)\\
=&\int_{\Omega_n}d\ov_1 d\ov_2 d\ov_3 d\ov_4 \delta(\ov_1+\ov_2-\ov_3-\ov_4)\delta(\arrowvert\ov_1\arrowvert^2+\arrowvert\ov_2\arrowvert^2-\arrowvert\ov_3\arrowvert^2-\arrowvert\ov_4\arrowvert^2-2\varepsilon_0)W_{ne}\\
&\;\;\;\;\left(F_n(\ov_1)F_n(\ov_2)-\lambda_nF_n(\ov_3)F_n(\ov_4)\right)\left(\varphi^2(\ov_3)+\varphi^1(\ov_4)-\varphi^1(\ov_1)-\varphi^1(\ov_2)\right).\\
\end{split}
\end{equation}
Hence, testing with $ \varphi^i(v)=1,\;v,\;\arrowvert v\arrowvert^2+\delta_{i2}2\varepsilon_0 $ we can immediately conclude that \linebreak$ \int_{\mathbb{R}^3}\;dv\;\left[ \mathbb{K}_{1,n}\left[F_n,\lambda_n\right]\varphi^1(v)+\mathbb{K}_{2,n}\left[F_n,\lambda_n\right]\varphi^2(v)\right]=0 $.
This result implies easily the conservation of mass and momentum, indeed for a solution $ (\lambda_n, F_n) $ we have
\begin{equation*}
\partial_t(1+\lambda_n)F_n= F_n \partial_t\lambda_n+(1+\lambda_n)\partial_tF_n=T_{1,n}(\lambda_n, F_n)F_n+T_{2,n}(\lambda_n, F_n)(1+\lambda_n)=\mathbb{K}_{1,n}\left[F_n,\lambda_n\right]+\mathbb{K}_{2,n}\left[F_n,\lambda_n\right].
\end{equation*}
To conclude the conservation of energy we first notice through a straightforward calculation from the definition of the cut-off kinetic system that $ \partial_t(\lambda_n)=(1-\lambda_n)^2\int_{\mathbb{R}^3}dv\;\left(\partial_tF_n-\mathbb{K}_{1,n}\left[F_n,\lambda_n\right]\right) $. This together with the multiplication rule $ \partial_t(\lambda_nF_n)=\mathbb{K}_{1,n}\left[F_n,\lambda_n\right]+\mathbb{K}_{2,n}\left[F_n,\lambda_n\right]-\partial_tF_n $ implies that 
\begin{equation*}
\partial_t\left(\int_{\mathbb{R}^3}dv\;\lambda_nF_n\right)+\partial_t\frac{\lambda_n}{1-\lambda_n}=-\int_{\mathbb{R}^3}dv\;\mathbb{K}_{1,n}\left[F_n,\lambda_n\right]=\int_{\mathbb{R}^3}dv\;\mathbb{K}_{2,n}\left[F_n,\lambda_n\right].
\end{equation*} Putting all these considerations together we obtain that
\begin{equation}
\begin{split}
\partial_t E_n(t)&=\partial_t\left(\int_{\mathbb{R}^3} dv \left[\left(1+\arrowvert v\arrowvert^2\right)(1+\lambda_n)F_n+2\varepsilon_0\lambda_n F_n\right]\right)+ 2\varepsilon_0\partial_t\frac{\lambda_n}{1-\lambda_n}\\&=\int_{\mathbb{R}^3}\;dv\;\left[ \mathbb{K}_{1,n}\left[F_n,\lambda_n\right]\left(1+\arrowvert v\arrowvert^2\right)+\mathbb{K}_{2,n}\left[F_n,\lambda_n\right]\left(1+\arrowvert v\arrowvert^2+2\varepsilon_0\right)\right]=0.
\end{split}
\end{equation}
All the integrals appearing in these calculations are well defined, since we are considering a solution $ \left(\lambda_n, F_n\right)\in C^1\left([0,\infty), X\right) $ and since the collisional kernels are bounded, we also have that $ \mathbb{K}_{i,n}\left[F_n,\lambda_n\right]\in L_2^1\left(\mathbb{R}^3\right) $.
\end{proof}
\end{lemma}
With this Lemma and in particular with the conservation of energy, we see again that for a solution $ (\lambda_n, F_n)\in C^1\left([0,1), X\right) $, the first term satisfies $ \lambda_n\in\left[0,\frac{E_n(0)}{2\varepsilon_0+E_n(0)}\right]~\subset~[0,1) $. Next we look at what happens to the entropy for the cut-off system. We define therefore $ \mathcal{H}_n(t):=\mathcal{H}[\lambda_n, F_n](t) $ for a solution $ \left(\lambda_n, F_n\right) $ according to Definition \ref{energyentropy}. The following two lemmas summarize the properties of the entropy in this case.
\begin{lemma}\label{lemcutoff2}
Let $ (\lambda_n, F_n)\in C^1\left([0,1), X\right) $ be a solution of the equation 
$ \partial_t\begin{pmatrix}\lambda_n\\ F_n\end{pmatrix} =T_n \begin{pmatrix}\lambda_n\\ F_n\end{pmatrix}$ 
for the initial value $ (\lambda_0, F_0)\in X_+\cap\tilde{X} $ which satisfies \pef{Assumption}. Assume $ \mathcal{H}_n(t) $ is well-defined and assume $ \Arrowvert F_n\Arrowvert_{L_4^1}\leq \kappa $ for some $ 0<\kappa<\infty $ and for all $ t\geq0 $. Then
\begin{equation*}
\partial_t \mathcal{H}_n(t)=\int_{\mathbb{R}^3}\;dv\;\left[ \mathbb{K}_{1,n}\left[F_n,\lambda_n\right]\log(F_n)+\mathbb{K}_{2,n}\left[F_n,\lambda_n\right]\log(\lambda_nF_n)\right] \leq 0.
\end{equation*}
\begin{proof}
In Lemma \ref{lemcutoff1} we saw that also in the case of this cut-off system the following is true
\begin{equation*}
\partial_t(1+\lambda_n)F_n=\mathbb{K}_{1,n}\left[F_n,\lambda_n\right]+\mathbb{K}_{2,n}\left[F_n,\lambda_n\right] 
\end{equation*}
and 
\begin{equation*}
\partial_t(\frac{\lambda_n}{1-\lambda_n})=\int_{\mathbb{R}^3}dv\;\left(\partial_tF_n-\mathbb{K}_{1,n}\left[F_n,\lambda_n\right]\right)=-\int_{\mathbb{R}^3}dv\;\left(\partial_t\lambda_nF_n-\mathbb{K}_{2,n}\left[F_n,\lambda_n\right]\right).
\end{equation*}
We can therefore proceed as we did in Proposition \ref{prop1} and obtain that
\begin{equation}
\begin{split}
&\partial_t \mathcal{H}_n(t)=\int_{\mathbb{R}^3}\;dv\;\left[\mathbb{K}_{1,n}\left[F_n,\lambda_n\right]\log(F_n)+\mathbb{K}_{2,n}\left[F_n,\lambda_n\right]\log(\lambda_nF_n)\right]\\
=&\frac{(1+\lambda_n)^2}{4}\int_{\mathbb{R}^3} dv_1\int_{\mathbb{R}^3}dv_2\int_{\mathbb{S}^2} d\omega\;\min(\arrowvert v_1-v_2\arrowvert,n)\left(F_n(v_3)F_n(v_4)-F_n(v_1)F_n(v_2)\right)\log\left(\frac{F_n(v_1)F_n(v_2)}{F_n(v_3)F_n(v_4)}\right)\\
&+\int_{\Omega_n}d\ov_1 d\ov_2 d\ov_3 d\ov_4 \delta(\ov_1+\ov_2-\ov_3-\ov_4)\delta(\arrowvert\ov_1\arrowvert^2+\arrowvert\ov_2\arrowvert^2-\arrowvert\ov_3\arrowvert^2-\arrowvert\ov_4\arrowvert^2-2\varepsilon_0)W_{ne}\\
&\;\;\;\;\left(\lambda_nF_n(\ov_3)F_n(\ov_4)-F_n(\ov_1)F_n(\ov_2)\right)\log\left(\frac{F_n(\ov_1)F_n(\ov_2)}{\lambda_nF_n(\ov_3)F_n(\ov_4)}\right)\\
\leq&0.
\end{split}
\end{equation}
\end{proof}
\end{lemma}
\begin{lemma}\label{lemcutoff3}
Let $ (\lambda_n, F_n)\in C^1\left([0,1), X\right) $ be a solution of the equation $ \partial_t\begin{pmatrix}\lambda_n \\F_n\end{pmatrix}=T_n\begin{pmatrix}\lambda_n\\F_n\end{pmatrix} $ for the initial value $ (\lambda_0, F_0)\in X_+\cap\tilde{X} $ which satisfies \pef{Assumption}. Assume $ \Arrowvert F_n\Arrowvert_{L_4^1}\leq \kappa $ for some $ 0<\kappa<\infty $ and for all $ t\geq0 $, $ F_0\log(F_0)\in L^1\left(\mathbb{R}^3\right) $ and $ E_n(0)=:E<\infty $. Then there exists a constant $ K $ depending on $ E $ such that $ \mathcal{H}_n(t)\geq K $ for all $ t $.
\begin{proof}
As we already have seen the assumptions imply $ \lambda_n\in\left[0,\frac{E}{2\varepsilon_0+E}\right]\subset[0,1) $. We now estimate each term of the entropy. So we start with $ (1+\lambda_n)\leq 2 $, also using the well-known inequality $ x \log(x)\geq -y+x\log(y) $ with $ x=F $ and $ y=e^{-\arrowvert v\arrowvert^2} $ we estimate
\begin{equation}
\int_{\mathbb{R}^3}dv\; F_n \log(F_n)\geq \int_{\mathbb{R}^3}dv\; F_n \log^-(F_n)\geq -\int_{\mathbb{R}^3}dv\;\left[ F_n\arrowvert v\arrowvert^2+e^{-\arrowvert v\arrowvert^2}\right]\geq -E-\int_{\mathbb{R}^3}dv\; e^{-\arrowvert v\arrowvert^2}:=-C_E
\end{equation}
It is well-known that $ \lambda_n\log(\lambda_n)\geq -\frac{1}{e} $ and also that $ \log(1-\lambda_n)\geq\log\left(1-\frac{E}{2\varepsilon_0+E}\right)=\log\left(\frac{2\varepsilon_0}{2\varepsilon_0+E}\right) $. Finally we claim: $ \frac{\lambda_n}{1-\lambda_n}\log(\lambda_n)\geq 1 $. In order to see that let us consider $ f(x)=\frac{x}{1-x}\log(x) $ for $ 0\leq x<1 $. This function is non-increasing, indeed by the well-known inequality $ \log(x)\leq x-1 $ we compute
\begin{equation*}
f'(x)=\frac{1}{(1-x)^2}\log(x)+\frac{1}{1-x}=\frac{1}{1-x}\left(\frac{1}{1-x}\log(x)+1\right)\leq 0.
\end{equation*}
Therefore since $ \lim\limits_{x\to 1}f(x)=-1 $ we conclude the claim.
All these estimates imply now the lower bound for the entropy
\begin{equation}
\begin{split}
\mathcal{H}_n(t)=&\int_{\mathbb{R}^3} \; dv \left[(1+\lambda_n)F_n\log(F_n)+F_n\lambda_n\log(\lambda_n)\right]+\frac{\lambda_n
}{1-\lambda_n}\log(\lambda_n)+\log(1-\lambda_n)\\
\geq&-2C_E-\frac{E}{e}-1-\log\left(\frac{2\varepsilon_0}{2\varepsilon_0+E}\right):=K.
\end{split}
\end{equation}
\end{proof} 
\end{lemma}
\begin{remark}
Lemma \ref{lemcutoff2} and Lemma \ref{lemcutoff3} imply together an important result for any strong solution \linebreak$ (\lambda_n, F_n)\in C^1\left([0,1), X\right) $ of $ \partial_t\begin{pmatrix}\lambda_n \\F_n\end{pmatrix}=T_n\begin{pmatrix}\lambda_n \\F_n\end{pmatrix} $ for the initial value $ (\lambda_0, F_0)\in X_+\cap\tilde{X} $ which satisfies \pef{Assumption}. If $ \Arrowvert F_n\Arrowvert_{L_4^1}\leq \kappa $ for some $ 0<\kappa<\infty $ and for all $ t\geq0 $, $ F_0\log(F_0)\in L^1\left(\mathbb{R}^3\right) $ and $ E_n(0)=:E<\infty $, then there exists a $ K\in\mathbb{R} $ such that $ K\leq\mathcal{H}_n(t)\leq\mathcal{H}(0)  $. This implies, as we will see in the next Lemma, that $ F_n\log(F_n) $ has uniformly bounded integral for any $ n $ and $ t $.
\end{remark}
\begin{lemma}\label{lemcutoff4}
Let $ (\lambda_n, F_n) $ and $ (\lambda_0, F_0) $ satisfy the same assumption as in Lemma \ref{lemcutoff3}. Then the sequence $ (\lambda_n, F_n) $ is bounded in $ C\left([0,\infty), X\right) $ and there exists some $ C\in\mathbb{R} $ independent of $ n $ and $ t $ such that $ \int_{\mathbb{R}^3}dv\;F_n\log(F_n)\leq C  $ for all $ n $ and $ t $. Moreover $ F_n\log(F_n)\in L^1\left(\mathbb{R}^3\right) $ for all $ n $ and $ t $.
\begin{proof}
The conservation of energy implies the first claim
\begin{equation*}
\sup\limits_{t\geq 0}\arrowvert\lambda_n(t)\arrowvert\leq \frac{E}{2\varepsilon_0+E}\;\;\text{ and }\;\;\sup\limits_{t\geq 0}\Arrowvert F_n(t)\Arrowvert_{L_2^1}\leq E.
\end{equation*}
The second claim follows from Lemma \ref{lemcutoff3}. Indeed
\begin{equation*}
\begin{split}
\int_{\mathbb{R}^3}dv\;F_n\log(F_n)\leq&\int_{\mathbb{R}^3}dv\;(1+\lambda_n)F_n\log(F_n)\\
=&\mathcal{H}_n(t)-\left[\int_{\mathbb{R}^3} \; dv F_n\lambda_n\log(\lambda_n)+\frac{\lambda_n
}{1-\lambda_n}\log(\lambda_n)+\log(1-\lambda_n)\right]\\
\leq& \mathcal{H}(0)+\frac{E}{e}+1+\log\left(\frac{1}{2\varepsilon_0+E}\right):=C.
\end{split}
\end{equation*}
Moreover, we already know that $ \int_{\mathbb{R}^3}dv\; F_n \log^-(F_n)\geq-C_E $, this together with the last estimate implies
\begin{equation*}
\int_{\mathbb{R}^3}dv\; F_n \log^+(F_n)=\int_{\mathbb{R}^3}dv\; F_n \log(F_n)-\int_{\mathbb{R}^3}dv\; F_n \log^-(F_n)\leq C+C_E.
\end{equation*} 
Since all these estimates are uniformly in time and $ n $, we conclude $ F_n\log(F_n)\in L^1\left(\mathbb{R}^3\right) $.
\end{proof}
\end{lemma}
The next and last Lemma of this section is very important for the existence of strong solution of the cut-off kinetic system.
\begin{lemma}\label{lemcutoff5}
Let  $ \begin{pmatrix}\lambda\\F\end{pmatrix}, \begin{pmatrix}\mu\\G\end{pmatrix}\in X_+ $ with bounded energies $ E(\lambda,F)\leq E_0 $ and $ E(\mu,G)\leq E_0 $, then there exists some $ C_n $ (which can depend on $ n $) such that
\begin{equation*}
\left\Arrowvert T_n\begin{pmatrix}\lambda\\ F\end{pmatrix}-T_n\begin{pmatrix}\mu \\G\end{pmatrix}\right\Arrowvert_X\leq C_n\left\Arrowvert\begin{pmatrix}\lambda \\F\end{pmatrix}-\begin{pmatrix}\mu \\G\end{pmatrix}\right\Arrowvert_X.
\end{equation*}
\begin{proof}
The proof is not difficult, but it requires several estimates. We start with some preliminary calculations. They will be useful also at the end of this paper. For most of the next estimates we use that for given functions $ F,\,G,\,H\,K $ and velocities $ v,\,w,\,u,\,\nu $ the following holds true
\begin{equation}\label{equality}
\begin{split}
\left(H(\nu)F(u)\right.&\left.-F(v)F(w)-K(\nu)G(u)+G(v)G(w)\right)\\&=\left(H(\nu)-K(\nu)\right)F(u)+K(\nu)\left(F(u)-G(u)\right)-\left(F(v)-G(v)\right)F(w)-G(v)\left(F(w)-G(w)\right).
\end{split}
\end{equation}
We can proceed with all preliminary estimates. First if all, using \pef{equality} together with the triangle inequality and the well-known inequality $\min(\arrowvert v_1-v_2\arrowvert,n) \leq \arrowvert v_1-v_2\arrowvert\leq(1+\arrowvert v_1\arrowvert^2)^{1/2}(1+\arrowvert v_2\arrowvert^2)^{1/2} $ we compute with the help of the monotonicity of the $ L^1_k $-norms in the exponent k the following
\begin{equation}\label{lip1}
\begin{split}
&\int_{\mathbb{R}^3} dv\;\left\arrowvert \mathbb{K}_{1,n}^{elastic}\left[F,\lambda F\right]-\mathbb{K}_{1,n}^{elastic}\left[G,\lambda G\right]\right\arrowvert\\
\leq&4 \int_{\mathbb{R}^3} dv_1\int_{\mathbb{R}^3}dv_2 \left(1+\arrowvert v_1\arrowvert^2\right)\left(1+\arrowvert v_2\arrowvert^2\right)\left[\left\arrowvert F(v_1)-G(v_1)\right\arrowvert F_2+\left\arrowvert F(v_2)-G(v_2)\right\arrowvert G_1\right]\\
\leq& 8E_0 \Arrowvert F-G\Arrowvert_{L_2^1}.
\end{split}
\end{equation}
We obtain a similar result (the constant multiplying $ \Arrowvert F-G\Arrowvert_{L_2^1} $ will be slightly different) for the first nonelastic collision terms, since here the only properties we need is again that $ B_{ne,n}^{ij}\leq B_{ne}^{ij} $ and the estimate in \pef{usual estimate 2} as follows
\begin{equation}\label{lip}
\begin{split}
&\int_{\mathbb{R}^3} dv\;\left\arrowvert \mathbb{K}_{1,n}^{nonelastic}\left[F,\lambda F\right]-\mathbb{K}_{1,n}^{nonelastic}\left[G,\lambda G\right]\right\arrowvert\\
\leq&6 \frac{C_0}{2}(1+2\sqrt{\varepsilon_0}) \int_{\mathbb{R}^3} d\ov_1\int_{\mathbb{R}^3}d\ov_2 \left(1+\arrowvert \ov_1\arrowvert^2\right)\left(1+\arrowvert \ov_2\arrowvert^2\right)\left[\left\arrowvert F(\ov_1)-G(\ov_1)\right\arrowvert F(\ov_2)+\left\arrowvert F(\ov_2)-G(\ov_2)\right\arrowvert G(\ov_1)\right]\\
\leq& 6 C_0(1+2\sqrt{\varepsilon_0})E_0 \Arrowvert F-G\Arrowvert_{L_2^1}.
\end{split}
\end{equation}
This means that there exists a constant $ \overline{C} $ depending on $ E_0 $ and on $ \varepsilon_0 $ such that
\begin{equation}\label{lip2}
\int_{\mathbb{R}^3} dv\;\left\arrowvert \mathbb{K}_{1,n}\left[F,\lambda F\right]-\mathbb{K}_{1,n}\left[G,\lambda G\right]\right\arrowvert\leq \overline{C}\Arrowvert F-G\Arrowvert_{L_2^1}.
\end{equation}
This immediately implies another useful inequality
\begin{equation}\label{lip3}
\int_{\mathbb{R}^3} dv\;\left\arrowvert \mathbb{K}_{1,n}\left[F,\lambda F\right]\right\arrowvert\leq \overline{C} E_0 .
\end{equation}
Technically more involved is the estimate for
\begin{equation*}
\int_{\mathbb{R}^3} dv\;\left(1+\arrowvert v\arrowvert^2\right)\left[\left\arrowvert \mathbb{K}_{1,n}\left[F,\lambda F\right]-\mathbb{K}_{1,n}\left[G,\lambda G\right]\right\arrowvert+\left\arrowvert\mathbb{K}_{2,n}\left[F,\lambda F\right]-\mathbb{K}_{2,n}\left[G,\lambda G\right]\right\arrowvert\right]
\end{equation*}

We proceed similarly as before using \pef{equality} and estimating the cut-off kernel by $ n $. In the change of variable we use this time the relation between pre- and post-collisional velocities as in \pef{elasticrel1} and we calculate
\begin{equation}\label{lip4}
\begin{split}
&\int_{\mathbb{R}^3} dv\;\left(1+\arrowvert v\arrowvert^2\right)\left[\left\arrowvert \mathbb{K}_{1,n}^{elastic}\left[F,\lambda F\right]-\mathbb{K}_{1,n}^{elastic}\left[G,\lambda G\right]\right\arrowvert+\left\arrowvert\mathbb{K}_{2,n}^{elastic}\left[F,\lambda F\right]-\mathbb{K}_{2,n}^{elastic}\left[G,\lambda G\right]\right\arrowvert\right]\\
\leq& 4n \int_{\mathbb{R}^3} dv_1\int_{\mathbb{R}^3}dv_2\left(1+\arrowvert v_1\arrowvert^2\right)\left[\left\arrowvert F(v_1)-G(v_1)\right\arrowvert F_2+\left\arrowvert F(v_2)-G(v_2)\right\arrowvert G_1\right]\\
&+4n\int_{\mathbb{R}^3} dv_1\int_{\mathbb{R}^3}dv_2 \left(1+\arrowvert v_3\arrowvert^2+1+\arrowvert v_4\arrowvert^2\right) \left[\left\arrowvert F(v_3)-G(v_3)\right\arrowvert F_4+\left\arrowvert F(v_4)-G(v_4)\right\arrowvert G_3\right]\\
\leq& 8n E_0 \Arrowvert F-G\Arrowvert_{L_2^1} + 16n E_0 \Arrowvert F-G\Arrowvert_{L_2^1}\leq 24n E_0\Arrowvert F-G\Arrowvert_{L_2^1}.
\end{split}
\end{equation}
For the nonelastic collisions we proceed in a similar way, using the inequality $ B_{ne,n}^{ij}\leq n $, and obtain a similar result as the estimate in \pef{lip4}
\begin{equation}\label{lip4bis}
\begin{split}
\int_{\mathbb{R}^3}& dv\;\left(1+\arrowvert v\arrowvert^2\right)\left[\left\arrowvert \mathbb{K}_{1,n}^{nonelastic}\left[F,\lambda F\right]-\mathbb{K}_{1,n}^{nonelastic}\left[G,\lambda G\right]\right\arrowvert+\left\arrowvert\mathbb{K}_{2,n}^{nonelastic}\left[F,\lambda F\right]-\mathbb{K}_{2,n}^{nonelastic}\left[G,\lambda G\right]\right\arrowvert\right]\\
\leq&2nC_0(1+2\sqrt{\varepsilon_0})\left[\int_{\mathbb{R}^3} d\ov_1\int_{\mathbb{R}^3} d\ov_2 \;\left(1+\arrowvert \ov_1\arrowvert^2+1+\arrowvert\ov_2\arrowvert^2\right)\left(\left|F(\ov_1)-G(\ov_1)\right|F(\ov_2)+\left|F(\ov_2)-G(\ov_2)\right|G(\ov_1)\right)\right. \\\;\; &\left.+(1+\varepsilon_0)\int_{\mathbb{R}^3}d\ov_3\int_{\mathbb{R}^3} d\ov_4 \;\left(1+\arrowvert \ov_3\arrowvert^2+1+\arrowvert \ov_4\arrowvert^2\right)\left(\left| F(\ov_3)-G(\ov_3)\right|F(\ov_4)+\left|F(\ov_4)-G(\ov_4)\right|G(\ov_3)\right)\right]\\
\leq& 16nC_0(1+2\sqrt{\varepsilon_0})(1+\varepsilon_0)E_0\Arrowvert F-G\Arrowvert_{L_2^1}.
\end{split}
\end{equation}
In all these calculations we estimated $ \lambda $ by $ 1 $. All these results together imply the existence of a constant $ \overline{C}_n>0 $ which depends on $ n $ and for which
\begin{equation}\label{lip5}
\int_{\mathbb{R}^3} dv\;\left(1+\arrowvert v\arrowvert^2\right)\left[\left\arrowvert \mathbb{K}_{1,n}\left[F,\lambda F\right]-\mathbb{K}_{1,n}\left[G,\lambda G\right]\right\arrowvert+\left\arrowvert\mathbb{K}_{2,n}\left[F,\lambda F\right]-\mathbb{K}_{2,n}\left[G,\lambda G\right]\right\arrowvert\right]\leq \overline{C}_n\Arrowvert F-G\Arrowvert_{L_2^1}
\end{equation}
holds.\\

Other additional estimates involve the difference of the operator $ T_n $ applied this time to the pair $ (G,\lambda G) $ and $ (G,\mu G) $. With a similar computation as for equation \pef{lip2} (now many terms of the form $ \left\arrowvert G(v)-G(v)\right\arrowvert $ cancel out) we conclude the existence of a constant $ \tilde{C}>0 $ independent of $ n $ but which depends on $ E_0 $ such that
\begin{equation}\label{lip8}
\int_{\mathbb{R}^3} dv\;\left\arrowvert \mathbb{K}_{1,n}\left[G,\lambda G\right]-\mathbb{K}_{1,n}\left[G,\mu G\right]\right\arrowvert\leq \tilde{C}E_0\arrowvert\lambda-\mu\arrowvert.
\end{equation}
This can be also explicitly computed by mean of the definition of collision operator, the relation \pef{equality}, the triangle inequality and finally we estimate the kernels as for equations \pef{lip1} and \pef{lip2} by
\begin{equation}
\begin{split}
&\int_{\mathbb{R}^3} dv\;\left\arrowvert \mathbb{K}_{1,n}^{elastic}\left[G,\mu G\right]-\mathbb{K}_{1,n}^{elastic}\left[G,\mu G\right]\right\arrowvert\\
\leq&\arrowvert \lambda-\mu\arrowvert\int_{\mathbb{R}^3} dv_1\int_{\mathbb{R}^3}dv_2\int_{\mathbb{S}^2} d\omega\;\min(\arrowvert v_1-v_2\arrowvert,n)\left\arrowvert G(v_3)G(v_4)-G(v_1)G(v_2)\right\arrowvert\\
\leq&2\arrowvert \lambda-\mu\arrowvert E_0^2
\end{split}
\end{equation}
and by
\begin{equation}
\begin{split}
&\int_{\mathbb{R}^3} dv\;\left\arrowvert \mathbb{K}_{1,n}^{nonelastic}\left[F,\lambda F\right]-\mathbb{K}_{1,n}^{nonelastic}\left[G,\mu G\right]\right\arrowvert\\
\leq&2\arrowvert \lambda-\mu\arrowvert\int_{\mathbb{R}^3} d\ov_1\int_{\mathbb{R}^3}d\ov_2\int_{\mathbb{S}^2} d\omega\;B_{ne,n}^{12}(\ov_1,\ov_2)\left\arrowvert G(\ov_3)G(\ov_4)\right\arrowvert\\
&+\arrowvert \lambda-\mu\arrowvert\int_{\mathbb{R}^3} d\ov_3\int_{\mathbb{R}^3}d\ov_4\int_{\mathbb{S}^2} d\omega\;B_{ne,n}^{34}(\ov_3,\ov_4) \left\arrowvert G(\ov_3)G(\ov_4)\right\arrowvert \\
\leq&3 \frac{C_0}{2}(1+2\sqrt{\varepsilon_0})E_0^2 \arrowvert \lambda-\mu\arrowvert.
\end{split}
\end{equation}
In an analogous way as we did for equation \pef{lip5} we can find a constant $ \tilde{C}_n $ dependent on $ n $ and on $ E_0 $ such that 
\begin{equation}\label{lip10}
\int_{\mathbb{R}^3} dv\;\left(1+\arrowvert v\arrowvert^2\right)\left\arrowvert \mathbb{K}_{1,n}\left[G,\lambda G\right]-\mathbb{K}_{1,n}\left[G,\mu G\right]\right\arrowvert+\left\arrowvert\mathbb{K}_{2,n}\left[G,\lambda G\right]-\mathbb{K}_{2,n}\left[G,\mu G\right]\right\arrowvert\leq \tilde{C}_nE_0\arrowvert\lambda-\mu\arrowvert.
\end{equation}
Again this can be seen easily doing similar computation as for \pef{lip4}
\begin{equation}\label{key1}
\begin{split}
&\int_{\mathbb{R}^3} dv\;\left(1+\arrowvert v\arrowvert^2\right)\left[\left\arrowvert \mathbb{K}_{1,n}^{elastic}\left[G,\lambda G\right]-\mathbb{K}_{1,n}^{elastic}\left[G,\mu G\right]\right\arrowvert+\left\arrowvert\mathbb{K}_{2,n}^{elastic}\left[G,\lambda G\right]-\mathbb{K}_{2,n}^{elastic}\left[G,\mu G\right]\right\arrowvert\right]\\
\leq & n\left(2\arrowvert\lambda-\mu\arrowvert+\arrowvert \lambda^2-\mu^2\arrowvert\right) \int_{\mathbb{R}^3} dv_1\int_{\mathbb{R}^3}dv_2\left(1+\arrowvert v_1\arrowvert^2\right)G(v_1)G(v_2)\\
&+n\left(2\arrowvert\lambda-\mu\arrowvert+\arrowvert \lambda^2-\mu^2\arrowvert\right)\int_{\mathbb{R}^3} dv_3\int_{\mathbb{R}^3}dv_4 \left(1+\arrowvert v_3\arrowvert^2+1+\arrowvert v_4\arrowvert^2\right) G(v_3)G(v_4)\\
\leq& 12n E_0^2 \arrowvert\lambda-\mu\arrowvert.
\end{split}
\end{equation}
For the nonelastic part all the terms depending on $ \ov_1 $ and $ \ov_2 $ eliminate, due to the fact that they are not multiplied by $ \lambda $ or $ \mu $ and so we can conclude the following
\begin{equation}\label{key2}
\begin{split}
&\int_{\mathbb{R}^3} dv\;\left(1+\arrowvert v\arrowvert^2\right)\left[\left\arrowvert \mathbb{K}_{1,n}^{nonelastic}\left[G,\lambda G\right]-\mathbb{K}_{1,n}^{nonelastic}\left[G,\mu G\right]\right\arrowvert+\left\arrowvert\mathbb{K}_{2,n}^{nonelastic}\left[G,\lambda G\right]-\mathbb{K}_{2,n}^{nonelastic}\left[G,\mu G\right]\right\arrowvert\right]\\
&\leq 2n\arrowvert \lambda-\mu\arrowvert(2+\varepsilon_0)\int_{\mathbb{R}^3}d\ov_3\;\int_{\mathbb{R}^3} d\ov_4 \int_{\mathbb{S}^2} d\omega\; \left(1+\arrowvert \ov_3\arrowvert^2+1+\arrowvert \ov_4^2\arrowvert\right)\left(\lambda -\mu\right)G(\ov_3)G(\ov_4)\\
&\leq 4n E_0^2(2+\varepsilon_0)\arrowvert \lambda-\mu\arrowvert.
\end{split}
\end{equation}

Equipped with estimates \pef{lip2}, \pef{lip5}, \pef{lip8} and \pef{lip10}, the proof of Lemma \ref{lemcutoff5} can be concluded. We can bound the norm of the difference of the maps $ T_n $ by the sum of four terms
\begin{equation}\label{lip6}
\begin{split}
\left\Arrowvert T_n\begin{pmatrix}\lambda \\F\end{pmatrix}-T_n\begin{pmatrix}\mu \\G\end{pmatrix}\right\Arrowvert_X\leq&\left\arrowvert T_{1,n}\begin{pmatrix}\lambda \\F\end{pmatrix}-T_{1,n}\begin{pmatrix}\lambda \\G\end{pmatrix}\right\arrowvert+\left\arrowvert T_{1,n}\begin{pmatrix}\lambda \\G\end{pmatrix}-T_{1,n}\begin{pmatrix}\mu \\G\end{pmatrix}\right\arrowvert\\
&+\left\Arrowvert T_{2,n}\begin{pmatrix}\lambda \\F\end{pmatrix}-T_{2,n}\begin{pmatrix}\lambda \\G\end{pmatrix}\right\Arrowvert_{L_2^1}+\left\Arrowvert T_{2,n}\begin{pmatrix}\lambda \\G\end{pmatrix}-T_{2,n}\begin{pmatrix}\mu\\G\end{pmatrix}\right\Arrowvert_{L_2^1}\\
=&I+II+III+IV.
\end{split}
\end{equation}
For the first term we need \pef{lip2} and \pef{lip3}. Then we see that there exists some $ C_1>0 $ depending on $ E_0 $ but not on $ n $ such that
\begin{equation}\label{lip7}
\begin{split}
I\leq&\frac{(1-\lambda)^2(1+\lambda)}{(1+\lambda)+(1-\lambda)^2\int_{\mathbb{R}^3}F} \int_{\mathbb{R}^3} dv\;\left\arrowvert \mathbb{K}_{1,n}\left[F,\lambda F\right]-\mathbb{K}_{1,n}\left[G,\lambda G\right]\right\arrowvert\\
&+(1-\lambda)^2(1+\lambda) \int_{\mathbb{R}^3} dv\;\left\arrowvert\mathbb{K}_{1,n}\left[G,\lambda G\right]\right\arrowvert\left\arrowvert\frac{(1+\lambda)+(1-\lambda)^2\int_{\mathbb{R}^3}G-(1+\lambda)+(1-\lambda)^2\int_{\mathbb{R}^3}F}{\left((1+\lambda)+(1-\lambda)^2\int_{\mathbb{R}^3}F\right)\left((1+\lambda)+(1-\lambda)^2\int_{\mathbb{R}^3}G\right)}\right\arrowvert\\
\leq&2 \int_{\mathbb{R}^3} dv\;\left\arrowvert \mathbb{K}_{1,n}\left[F,\lambda F\right]-\mathbb{K}_{1,n}\left[G,\lambda G\right]\right\arrowvert+2\int_{\mathbb{R}^3} dv\;\left\arrowvert\mathbb{K}_{1,n}\left[G,\lambda G\right]\right\arrowvert\Arrowvert F-G\Arrowvert_{L_2^1}\\\leq& C_1 \Arrowvert F-G\Arrowvert_{L_2^1}.
\end{split}
\end{equation}
For the second term we use the result in \pef{lip3} and in \pef{lip8}. We find in this way a constant $ C_2>0 $ independent of $ n $ which bind the second term as follows
\begin{equation}\label{lip9}
\begin{split}
II\leq& \frac{(1-\lambda)^2(1+\lambda)}{(1+\lambda)+(1-\lambda)^2\int_{\mathbb{R}^3}G} \int_{\mathbb{R}^3} dv\;\left\arrowvert \mathbb{K}_{1,n}\left[G,\lambda G\right]-\mathbb{K}_{1,n}\left[G,\mu G\right]\right\arrowvert\\
&+\int_{\mathbb{R}^3} dv\;\left\arrowvert\mathbb{K}_{1,n}\left[G,\mu G\right]\right\arrowvert\left\arrowvert\frac{(1-\lambda)^2(1+\lambda)}{(1+\lambda)+(1-\lambda)^2\int_{\mathbb{R}^3}G}-\frac{(1-\mu)^2(1+\mu)}{(1+\mu)+(1-\mu)^2\int_{\mathbb{R}^3}G}\right\arrowvert\\
\leq& 2 \tilde{C} E_0 \arrowvert\lambda-\mu\arrowvert+\overline{C}E_0\left[(1+\lambda)(1+\mu)\arrowvert(1-\lambda)^2-(1-\mu)^2\arrowvert+(1-\lambda)^2(1-\mu)^2\int_{\mathbb{R}^3}G\arrowvert\lambda-\mu\arrowvert\right]\\
\leq&2 \tilde{C} E_0 \arrowvert\lambda-\mu\arrowvert+\overline{C}E_0(6+E_0)\arrowvert\lambda-\mu\arrowvert\\
\leq&C_2\arrowvert\lambda-\mu\arrowvert.
\end{split}
\end{equation}
Now we consider the estimates for the $ L_2^1 $-norms. By mean of equations \pef{lip2} and \pef{lip5} we can compute for the third term
\begin{equation}\label{lip11}
\begin{split}
III\leq& \frac{1}{1+\lambda}\int_{\mathbb{R}^3} dv\;\left(1+\arrowvert v\arrowvert^2\right)\left\arrowvert \mathbb{K}_{1,n}\left[F,\lambda F\right]-\mathbb{K}_{1,n}\left[G,\lambda G\right]\right\arrowvert+\left\arrowvert\mathbb{K}_{2,n}\left[F,\lambda F\right]-\mathbb{K}_{2,n}\left[G,\lambda G\right]\right\arrowvert\\
&+\frac{(1-\lambda)^2}{(1+\lambda)+(1-\lambda)^2\int_{\mathbb{R}^3}F}\int_{\mathbb{R}^3} dv\;\left\arrowvert \mathbb{K}_{1,n}\left[F,\lambda F\right]\right\arrowvert\int_{\mathbb{R}^3} dv\;\left(1+\arrowvert v\arrowvert^2\right)\left\arrowvert F-G\right\arrowvert\\
&+\frac{(1-\lambda)^2}{(1+\lambda)+(1-\lambda)^2\int_{\mathbb{R}^3}F}\int_{\mathbb{R}^3} dv\;\left\arrowvert \mathbb{K}_{1,n}\left[F,\lambda F\right]-\mathbb{K}_{1,n}\left[G,\lambda G\right]\right\arrowvert\int_{\mathbb{R}^3} dv\;\left(1+\arrowvert v\arrowvert^2\right)\left\arrowvert G\right\arrowvert\\
&+\left\arrowvert\frac{(1-\lambda)^2}{(1+\lambda)+(1-\lambda)^2\int_{\mathbb{R}^3}F}-\frac{(1-\lambda)^2}{(1+\lambda)+(1-\lambda)^2\int_{\mathbb{R}^3}G}\right\arrowvert\int_{\mathbb{R}^3} dv\;\left\arrowvert \mathbb{K}_{1,n}\left[G,\lambda G\right]\right\arrowvert\int_{\mathbb{R}^3} dv\;\left(1+\arrowvert v\arrowvert^2\right)\left\arrowvert G\right\arrowvert\\
\leq&\overline{C}_n \Arrowvert F-G\Arrowvert_{L_2^1}+\overline{C} E_0 \Arrowvert F-G\Arrowvert_{L_2^1}+ \overline{C} E_0 \Arrowvert F-G\Arrowvert_{L_2^1}+\overline{C} E_0^2 \Arrowvert F-G\Arrowvert_{L_2^1}\\
\leq& C_3 \Arrowvert F-G\Arrowvert_{L_2^1},
\end{split}
\end{equation}
for some $ C_3>0 $ depending on $ n $ and $ E_0 $. 

Finally for the last term we proceed in a similar way, using again equations \pef{lip2}, \pef{lip3} and also \pef{lip8} and \pef{lip10}. We therefore find a constant $ C_4 $ depending on $ n $ and $ E_0 $ such that the fourth term is bounded from above by
\begin{equation}\label{lip12}
\begin{split}
IV\leq&\frac{1}{1+\lambda}\int_{\mathbb{R}^3} dv\;\left(1+\arrowvert v\arrowvert^2\right)\left[\left\arrowvert \mathbb{K}_{1,n}\left[G,\lambda G\right]-\mathbb{K}_{1,n}\left[G,\mu G\right]\right\arrowvert+\left\arrowvert\mathbb{K}_{2,n}\left[G,\lambda G\right]-\mathbb{K}_{2,n}\left[G,\mu G\right]\right\arrowvert\right]\\
&+\left\arrowvert\frac{1}{1+\lambda}-\frac{1}{1+\mu}\right\arrowvert\int_{\mathbb{R}^3} dv\;\left(1+\arrowvert v\arrowvert^2\right)\left[\left\arrowvert\mathbb{K}_{1,n}\left[G,\mu G\right]\right\arrowvert+\left\arrowvert\mathbb{K}_{2,n}\left[G,\mu G\right]\right\arrowvert\right]\\
&+\frac{(1-\lambda)^2}{(1+\lambda)+(1-\lambda)^2\int_{\mathbb{R}^3}G}\int_{\mathbb{R}^3} dv\;\left\arrowvert \mathbb{K}_{1,n}\left[G,\lambda G\right]-\mathbb{K}_{1,n}\left[G,\mu G\right]\right\arrowvert\int_{\mathbb{R}^3} dv\;\left(1+\arrowvert v\arrowvert^2\right)\left\arrowvert G\right\arrowvert\\
&+\left\arrowvert\frac{(1-\lambda)^2}{(1+\lambda)+(1-\lambda)^2\int_{\mathbb{R}^3}G}-\frac{(1-\mu)^2}{(1+\mu)+(1-\mu)^2\int_{\mathbb{R}^3}G}\right\arrowvert\int_{\mathbb{R}^3} dv\;\left\arrowvert \mathbb{K}_{1,n}\left[G,\mu G\right]\right\arrowvert\int_{\mathbb{R}^3} dv\;\left(1+\arrowvert v\arrowvert^2\right)\left\arrowvert G\right\arrowvert\\
\leq& \tilde{C}_nE_0\arrowvert\lambda-\mu\arrowvert+ \overline{C}_n E_0 \arrowvert\lambda-\mu\arrowvert+E_0 \tilde{C} \arrowvert\lambda-\mu\arrowvert+ \overline{C} E_0^2 \arrowvert(1+\mu)(1-\lambda)^2-(1+\lambda)(1-\mu)^2\arrowvert\\
\leq& C_4 \arrowvert\lambda-\mu\arrowvert,
\end{split}
\end{equation}
where in the last step we used that $ \arrowvert(1+\mu)(1-\lambda)^2-(1+\lambda)(1-\mu)^2\arrowvert\leq 6\arrowvert\lambda-\mu\arrowvert $.

We are ready now to conclude the proof of this Lemma. Indeed the estimates \pef{lip7}, \pef{lip9}, \pef{lip11} and \pef{lip12} imply the existence of a constant $ C_n $ depending on $ n $ such that
\begin{equation}\label{lip13}
\left\Arrowvert T_n\begin{pmatrix}\lambda\\F\end{pmatrix}-T_n\begin{pmatrix}\mu \\G\end{pmatrix}\right\Arrowvert_X\leq \left(C_2+C_4 \right)\arrowvert\lambda-\mu\arrowvert+\left(C_1+C_3\right) \Arrowvert F-G\Arrowvert_{L_2^1}\leq C_n \left\Arrowvert\begin{pmatrix}\lambda \\F\end{pmatrix}-\begin{pmatrix}\mu \\G\end{pmatrix}\right\Arrowvert_X.
\end{equation}
\end{proof}
\end{lemma} 
This Lemma ends Section 3.2 and we can now prove the existence of weak solutions for the system \pef{newsystem1}.
\subsection{Existence of weak solutions}
We are now ready to prove the existence of weak solutions. Because of the uniform bound of the entropy, the cut-off system of Section 3.2 gives us a sequence of functions which converges to the desired weak solutions. However, we still need one result before the actual theorem. We namely need, that for all $ n\geq M $ the $ L^1_4 $-norm is uniformly bounded. This is proved in the next Lemma, which can be proved adapting the theory of the classical Boltzmann equation.
\begin{lemma}\label{lemexistence}
Let $ \left(\lambda_n, F_n\right)\in C^1\left([0,\infty), X\right) $ be the unique solution of the approximating kinetic system $ \partial_t\begin{pmatrix}\lambda_n\\ F_n\end{pmatrix} =T_n \begin{pmatrix}\lambda_n\\ F_n\end{pmatrix}$ with initial value $ (\lambda_0,F_0)\in X_+ $ which satisfies \pef{Assumption}. Assume also that $ \Arrowvert F_0\Arrowvert_{L_4^1}<\infty $. Then there exist constants $ M $ and $ C $ such that for every $ n\geq M $ it holds
\begin{equation}\label{lemexistence1}
\sup_{t\geq 0}\Arrowvert F_n\Arrowvert_{L_4^1}\leq C\:\:\:\:\:\:\text{ and }\:\:\:\:\:\:\int_{0}^t \Arrowvert F_n (., \tau)\min\left(\arrowvert.\arrowvert, n\right)\Arrowvert_{L_4^1}\leq C(1+t).
\end{equation}
\begin{proof} We adapt to our system the proof of Lemma 4.2 in \cite{mischler}.
We will proceed in two steps. We first prove that there exist $ M $ and $ C $ such that
\begin{equation}\label{3lem}
\sup_{t\geq 0}\Arrowvert F_n\Arrowvert_{L_3^1}\leq C\:\:\:\:\:\:\text{ and }\:\:\:\:\:\:\int_{0}^t \Arrowvert F_n (., \tau)\min\left(\arrowvert.\arrowvert, n\right)\Arrowvert_{L_3^1}\leq C(1+t)
\end{equation}
and then with the help of this result, we will prove the Lemma.

Since there exist constants $ c_3 $ and $ c_4 $ such that $ \left(1+\arrowvert v\arrowvert^2\right)^{3/2}\leq c_3\left(1+\arrowvert v\arrowvert^3\right) $ and $ \left(1+\arrowvert v\arrowvert^2\right)^{2}\leq c_4\left(1+\arrowvert v\arrowvert^4\right) $, we will show the result for the weighted norm $ \Arrowvert F\Arrowvert_s:=\int_{\mathbb{R}^3} dv\left(F^1(v)+F^2(v)\right)\arrowvert v\arrowvert^s $ for $ s=3,4 $. Indeed, since the solutions have the $ L_2^1 $-norm bounded, the initial mass is conserved, and so we would conclude the theorem.

We need now some estimates for the elastic and nonelastic kernels. First of all for the elastic kernels we obtain the following relation
\begin{equation}\label{mischler1}
\frac{1}{4}\min\left(\arrowvert v_1\arrowvert, n\right)-\arrowvert v_2\arrowvert\leq \min\left(\arrowvert v_1-v_2\arrowvert, n\right)\leq 4 \left(\arrowvert v_2\arrowvert+\min\left(\arrowvert v_1\arrowvert,n\right)\right).
\end{equation}
We recall the well-known Povzner inequality for the elastic collisions, as given e.g. in \cite{villani} and \cite{povzner}
\begin{equation}\label{povzner}
\left(\arrowvert v_3\arrowvert^s+\arrowvert v_4\arrowvert^s-\arrowvert v_1\arrowvert^s-\arrowvert v_2 \arrowvert^s\right)\leq -K(\theta) \arrowvert v_1\arrowvert^{s}+ C\left(\arrowvert v_1\arrowvert^{s-1}\arrowvert v_2 \arrowvert+\arrowvert v_1\arrowvert\;\arrowvert v_2\arrowvert^{s-1}\right),
\end{equation} 
where $ s>2 $, $ \theta $ is the angle between $ v_1-v_2 $ and $ v_3-v_4 $ and $ K(\theta)\geq 0 $ and strictly positive for $ 0<\theta\leq \frac{\pi}{2} $. Therefore, using \pef{povzner} we have the following inequality
\begin{equation}\label{mischler2}
\int_{\mathbb{S}^2} d\omega \;\left(\arrowvert v_3\arrowvert^3+\arrowvert v_4\arrowvert^3-\arrowvert v_1\arrowvert^3-\arrowvert v_2\arrowvert^3\right)\leq C_1 \left(\arrowvert v_1\arrowvert^2\arrowvert v_2\arrowvert+\arrowvert v_1\arrowvert\arrowvert v_2\arrowvert^2\right)-c_1\arrowvert v_1\arrowvert^3.
\end{equation}
Therefore putting \pef{mischler1} and \pef{mischler2} together and using that $ 0\leq \lambda_n<1 $ we compute
\begin{equation}\label{mischler3}
\begin{split}
&\frac{(1+\lambda_n)^2}{2}\int_{\mathbb{R}^3}dv_1\int_{\mathbb{R}^3}dv_2\int_{\mathbb{S}^2}d\omega\; B_{el,n}(v_1,v_2) F_n(v_1)F_n(v_2)\left(\arrowvert v_3\arrowvert^3+\arrowvert v_4\arrowvert^3-\arrowvert v_1\arrowvert^3-\arrowvert v_2\arrowvert^3\right)\\
\leq& C_2\left(\Arrowvert F_n\Arrowvert_{L_2^1}^2+\Arrowvert F_n\Arrowvert_{L_1^1}\int_{\mathbb{R}^3}dv\left(F_n(v)\arrowvert v\arrowvert^2\min\left(\arrowvert v\arrowvert,n\right)\right)\right)\;+C_3\Arrowvert F_n\Arrowvert_3 \Arrowvert F_n\Arrowvert_{L_1^1}\\
&-C_4\Arrowvert F_n\Arrowvert_{L^1}\int_{\mathbb{R}^3}dv\left(F_n(v)\right)\arrowvert v\arrowvert^3\min\left(\arrowvert v\arrowvert,n\right),\\
\end{split}
\end{equation}
for some constants $ C_2,\,C_3,\,C_4>0 $ depending on $ C_1 $ and $ c_1 $.

We proceed now with some estimates for the terms of the nonelastic collisions. We notice first of all that for the kernels we have following relation due to \pef{usual estimate}
\begin{equation}\label{mischler4}
B_{ne,n}^{12}(\ov_1,\ov_2)\leq \frac{C_0}{2}\arrowvert\ov_1-\ov_2\arrowvert\rchi_{\arrowvert\ov_1-\ov_2\arrowvert\leq n}\left(\ov_1,\ov_2\right)\leq \frac{C_0}{2}\min\left(\arrowvert\ov_1-\ov_2\arrowvert,n\right)\leq 2C_0\left(\arrowvert \ov_2\arrowvert+\min\left(\arrowvert \ov_1\arrowvert,n\right)\right)
\end{equation}
and also since $ \sqrt{n^2-4\varepsilon_0}\leq n $ we estimate
\begin{equation}\label{mischler5}
\begin{split}
B_{ne,n}^{34}(\ov_3,\ov_4)\leq&\frac{C_0}{2}\left(\arrowvert\ov_3-\ov_4\arrowvert+2\sqrt{\varepsilon_0}\right)\rchi_{\arrowvert\ov_3-\ov_4\arrowvert\leq n}\left(\ov_3,\ov_4\right)\\
\leq& \frac{C_0}{2}\left(\min\left(\arrowvert\ov_3-\ov_4\arrowvert,n\right)+2\sqrt{\varepsilon_0}\right)\leq 2C_0\left(\arrowvert \ov_4\arrowvert+\min\left(\arrowvert \ov_3\arrowvert,n\right)+2\sqrt{\varepsilon_0}\right).
\end{split}
\end{equation}
Moreover we have also some estimates for the nonelastic velocities. We can use the well-known relation $ x^3+y^3 \leq \left(x^2+y^2\right)^{3/2} $ for $ x,y\geq0 $ and also that there exists some $ C>0 $ such that the following holds for $ x,y\geq 0 $
\begin{equation}
\begin{split}
\left(\left(x+y\right)^{3/2}-x^{3/2}-y^{3/2}\right)\leq C\left(x^{1/2}y+x y^{1/2}\right).
\end{split}
\end{equation} By the relation for the velocities of the non elastic collisions \pef{nonelasticrel1} we see on one hand that
\begin{equation}\label{mischler6}
\arrowvert \ov_3\arrowvert^3+\arrowvert \ov_4\arrowvert^3-\arrowvert \ov_1\arrowvert^3-\arrowvert \ov_2\arrowvert^3\leq C\left(\arrowvert \ov_1\arrowvert\;\arrowvert \ov_2\arrowvert^2+\arrowvert \ov_1\arrowvert^2\arrowvert \ov_2\arrowvert\right).
\end{equation}
And on the other hand there exists also a constant $ C_{\eps_0} $ depending on $ \eps_0 $ such that
\begin{equation}\label{mischler7}
\arrowvert \ov_1\arrowvert^3+\arrowvert \ov_2\arrowvert^3-\arrowvert \ov_3\arrowvert^3-\arrowvert \ov_4\arrowvert^3\leq C_{\varepsilon_0}\left(1+\arrowvert\ov_3\arrowvert\;\arrowvert \ov_4\arrowvert^2+\arrowvert \ov_3\arrowvert^2\arrowvert \ov_4\arrowvert+\arrowvert \ov_3\arrowvert+\arrowvert \ov_4\arrowvert+\arrowvert \ov_3\arrowvert^2+\arrowvert \ov_4\arrowvert^2\right).
\end{equation}
Therefore using \pef{mischler4} and \pef{mischler6}, we compute
\begin{equation}\label{mischler8}
\begin{split}
&B_{ne,n}^{12}(\ov_1,\ov_2)\left(\arrowvert \ov_3\arrowvert^3+\arrowvert \ov_4\arrowvert^3-\arrowvert \ov_1\arrowvert^3-\arrowvert \ov_2\arrowvert^3\right)\\
\leq&C_0^1 \left(\arrowvert \ov_1\arrowvert\;\arrowvert \ov_2\arrowvert^3+\arrowvert \ov_1\arrowvert^2\arrowvert \ov_2\arrowvert^2+\arrowvert \ov_1\arrowvert\min\left(\arrowvert \ov_1\arrowvert,n\right)\arrowvert \ov_2\arrowvert^2+\arrowvert \ov_1\arrowvert^2\min\left(\arrowvert \ov_1\arrowvert,n\right)\arrowvert \ov_2\arrowvert\right).
\end{split}
\end{equation}
While \pef{mischler5} and \pef{mischler7} imply
\begin{equation}\label{mischler9}
\begin{split}
&B_{ne,n}^{34}(\ov_3,\ov_4)\left(\arrowvert \ov_1\arrowvert^3+\arrowvert \ov_2\arrowvert^3-\arrowvert \ov_3\arrowvert^3-\arrowvert \ov_4\arrowvert^3\right)\\
\leq& C_0^2\left(1+2\arrowvert\ov_4\arrowvert+\arrowvert\ov_3\arrowvert+\min\left(\arrowvert\ov_3\arrowvert,n\right)+\arrowvert\ov_3\arrowvert\arrowvert\ov_4\arrowvert+\min\left(\arrowvert\ov_3\arrowvert,n\right)\arrowvert\ov_4\arrowvert+2\arrowvert\ov_4\arrowvert^2+\arrowvert\ov_3\arrowvert^2\right.\\
&\left.\:\:\:\:\:\:\:\:\:\:+\arrowvert\ov_3\arrowvert\min\left(\arrowvert\ov_3\arrowvert,n\right)2\arrowvert\ov_3\arrowvert^2\arrowvert\ov_4\arrowvert+\arrowvert\ov_3\arrowvert\arrowvert\ov_4\arrowvert^2+\min\left(\arrowvert\ov_3\arrowvert,n\right)\arrowvert\ov_4\arrowvert^2+\arrowvert\ov_3\arrowvert^2\arrowvert\ov_4\arrowvert^2\right.\\
&\left.\:\:\:\:\:\:\:\:\:\:+\min\left(\arrowvert\ov_3\arrowvert,n\right)\arrowvert\ov_3\arrowvert\arrowvert\ov_4\arrowvert^2+\arrowvert\ov_3\arrowvert\arrowvert\ov_4\arrowvert^3+\arrowvert\ov_4\arrowvert^3+\min\left(\arrowvert\ov_3\arrowvert,n\right)\arrowvert\ov_3\arrowvert^2\arrowvert\ov_4\arrowvert+\min\left(\arrowvert\ov_3\arrowvert,n\right)\arrowvert\ov_3\arrowvert^2\right).
\end{split}
\end{equation}
Now putting these estimates \pef{mischler8} and \pef{mischler9} in the weak formulation \pef{nonelasticcutoff4} and evaluating the delta functions with respect to $ (\ov_1,\ov_2) $ or $ (\ov_3,\ov_4) $ we conclude
\begin{equation}\label{mischler10}
\begin{split}
&\int_{\Omega_n}d\ov_1 d\ov_2 d\ov_3 d\ov_4 \delta(\ov_1+\ov_2-\ov_3-\ov_4)\delta(\arrowvert\ov_1\arrowvert^2+\arrowvert\ov_2\arrowvert^2-\arrowvert\ov_3\arrowvert^2-\arrowvert\ov_4\arrowvert^2-2\varepsilon_0)W_{ne}\\
&\;\;\;\;\left(F_n(\ov_1)F_n(\ov_2)+\lambda_n F_n^2(\ov_3)F_n(\ov_4)\right)\left(\arrowvert \ov_3\arrowvert^3+\arrowvert \ov_4\arrowvert^3-\arrowvert \ov_1\arrowvert^3-\arrowvert \ov_2\arrowvert^3\right)\\
\leq& \tilde{C}\left(\Arrowvert F_n\Arrowvert_{L_2^1}^2+\Arrowvert F_n\Arrowvert_{L_1^1}\Arrowvert F_n\Arrowvert_3+ \Arrowvert F_n\Arrowvert_{L^1}\int_{\mathbb{R}^3}dv\left(F_n(v)\arrowvert v\arrowvert^2\min\left(\arrowvert v\arrowvert,n\right)\right) \right).
\end{split}
\end{equation}
Here we also used the monotonicity of the $ L^1_k $-norms estimating all moments less then 2 by the $ L^1_2 $-norm.

Therefore putting together the estimates \pef{mischler3} and \pef{mischler10} we conclude the existence of constants $ \tilde{C}_2 $, $ \tilde{C}_3 $ and $ C_4 $ such that the following holds
\begin{equation}\label{lemexistence2}
\begin{split}
&\frac{1}{1+\lambda_n}\int_{\mathbb{R}^3}dv\;\arrowvert v\arrowvert^3\left(\mathbb{K}_1\left[F_n,\lambda_n\right]+\mathbb{K}_2\left[F_n,\lambda_n\right]\right)\\ 
\leq& \tilde{C}_2\left(\Arrowvert F_n\Arrowvert_{L_2^1}^2+ \Arrowvert F_n\Arrowvert_{L_1^1}\int_{\mathbb{R}^3}dv\;F_n(v)\arrowvert v\arrowvert^{2}\min\left(\arrowvert v\arrowvert,n\right) \right)+\tilde{C}_3\Arrowvert F_n\Arrowvert_{L_1^1}\Arrowvert F_n\Arrowvert_3\\
&-C_4\Arrowvert F_n\Arrowvert_{L^1}\int_{\mathbb{R}^3}dv\;F_n(v)\arrowvert v\arrowvert^3\min\left(\arrowvert v\arrowvert,n\right).
\end{split}
\end{equation}
We now evaluate the term
\begin{equation}\label{lemexistence3}
\frac{(1-\lambda_n)^2}{(1+\lambda_n)+(1-\lambda_n)^2\int_{\mathbb{R}^3}F}\int_{\mathbb{R}^3}dv\; \mathbb{K}_1\left[F_n,\lambda_n\right]\int_{\mathbb{R}^3}dv\;F_n \arrowvert v\arrowvert^3.
\end{equation}
We can use \pef{lip3} and find a constant $ \overline{C} $ which depends on the initial energy, which we know to be conserved, such that
\begin{equation}\label{lemexistence4}
\int_{\mathbb{R}^3}dv\;\left\arrowvert \mathbb{K}_1\left[F_n,\lambda_n\right]\right\arrowvert\leq \overline{C} \Arrowvert F_n\Arrowvert_{L_1^1}.
\end{equation}
We obtain this time the $ L_1^1 $-norm due to a sharper estimate for the collisional kernels in both elastic and nonelastic terms as follows \begin{equation}\label{lemexistence5}
\arrowvert v_1-v_2\arrowvert\leq (1+\arrowvert v_2\arrowvert^2)(1+\arrowvert v_1\arrowvert^2)^{1/2}\leq (1+\arrowvert v_2\arrowvert^2)(1+\arrowvert v_1\arrowvert).
\end{equation}
Thus we easily see that
\begin{equation}\label{lemexistence6}
\frac{(1-\lambda_n)^2}{(1+\lambda_n)+(1-\lambda_n)^2\int_{\mathbb{R}^3}F}\int_{\mathbb{R}^3}dv\; \mathbb{K}_1\left[F_n,\lambda_n\right]\int_{\mathbb{R}^3}dv\;F_n \arrowvert v\arrowvert^3\leq \overline{C}\Arrowvert F_n\Arrowvert_{L_1^1}\Arrowvert F_n\Arrowvert_{3}.
\end{equation}
Therefore there exist constants $ \tilde{C}_2 $, $ \overline{C}_3 $ and $ C_4 $ such that for the time derivative of the third moment the following holds
\begin{equation}\label{lemexistence7}
\begin{split}
\partial_t\Arrowvert F_n\Arrowvert_{3}\leq &2\tilde{C}_2\left(\Arrowvert F_n\Arrowvert_{L_{2}^1}\Arrowvert F_n\Arrowvert_{L_2^1}+ \Arrowvert F_n\Arrowvert_{L_1^1}\int_{\mathbb{R}^3}dv\;F_n(v)\arrowvert v\arrowvert^{2}\min\left(\arrowvert v\arrowvert,n\right) \right)+\overline{C}_3\Arrowvert F_n\Arrowvert_{L_1^1}\Arrowvert F_n\Arrowvert_3\\
&-C_4\Arrowvert F_n\Arrowvert_{L^1}\int_{\mathbb{R}^3}dv\;F_n(v)\arrowvert v\arrowvert^3\min\left(\arrowvert v\arrowvert,n\right).
\end{split}
\end{equation}
Moreover, for all $ \delta>0 $ there exists a constant $ C_\delta $ such that
\begin{equation}\label{mischler12}
\int_{\mathbb{R}^3}dv\left(F_n(v)\right)\arrowvert v\arrowvert^2\left(1+\min\left(\arrowvert v\arrowvert,n\right)\right)\leq C_\delta\Arrowvert F_n\Arrowvert_{L_1^1}+\delta\int_{\mathbb{R}^3}dv\left(F_n(v)\right)\arrowvert v\arrowvert^3\min\left(\arrowvert v\arrowvert,n\right).
\end{equation}
This can be seen using the Young's Inequality. Indeed for any $ \delta>0 $ there exist constants $ K_\delta^i>0 $ such that
\begin{equation}\label{delta1}
\arrowvert v\arrowvert^2=\arrowvert v\arrowvert^{1/2}\arrowvert v\arrowvert^{3/2}\leq K^1_\delta \arrowvert v\arrowvert+\delta \arrowvert v\arrowvert^3\leq K_\delta^1\arrowvert v\arrowvert+\delta+\delta\arrowvert v\arrowvert^3\min\left(\arrowvert v\arrowvert, n\right).
\end{equation}
and similarly we estimate also
\begin{equation}\label{delta2}
\arrowvert v\arrowvert^2\min\left(\arrowvert v\arrowvert,n\right)=\min\left(\arrowvert v\arrowvert,n\right)^{1/3}\arrowvert v\arrowvert^2\min\left(\arrowvert v\arrowvert,n\right)^{2/3}\leq K_\delta^2\min\left(\arrowvert v\arrowvert,n\right)+\delta\arrowvert v\arrowvert^3\min\left(\arrowvert v\arrowvert,n\right).
\end{equation}
These computations yield the existence of a constant $ C_\delta>0 $ such that
\begin{equation}
\arrowvert v\arrowvert^2\left(1+\min\left(\arrowvert v\arrowvert,n\right)\right)\leq C_\delta \left(1+\arrowvert v\arrowvert\right)+\delta \arrowvert v\arrowvert^3\min\left(\arrowvert v\arrowvert,n\right).
\end{equation}
Moreover for all $ n\geq M $ the following is also true
\begin{equation}\label{mischler13}
\arrowvert v\arrowvert^3\leq \frac{1}{M}\arrowvert v\arrowvert^3\min\left(\arrowvert v\arrowvert,n\right)+M^3.
\end{equation}
The last equations \pef{mischler12} and \pef{mischler13} imply also with the help of the monotonicity of the $ L^1_k $-norms and the conservation of mass, momentum and energy the existence of a constant $ C_{\delta, M} $, which also depends on initial mass and energy, such that 
\begin{equation}\label{mischler14}
\begin{split}
\partial_t\Arrowvert F_n\Arrowvert_3+&\left(\frac{C_4}{2} \Arrowvert F_n\Arrowvert_{L^1}-C_5\left(\Arrowvert F_n\Arrowvert_{L^1}+\Arrowvert F_n\Arrowvert_{L_2^1}\right)\left(\delta+\frac{1}{M}\right)\right)\\
&\left(\Arrowvert F_n\Arrowvert_3+\int_{\mathbb{R}^3}dv\left(F_n(v)\right)\arrowvert v\arrowvert^3\min\left(\arrowvert v\arrowvert,n\right)\right)\\
\leq&C_{\delta,M}.
\end{split}
\end{equation}
Taking now $ \delta $ small enough and $ M $ large enough Gr\"onwall's Lemma implies the estimate \pef{3lem}.\\

We can now proceed with the proof of the Lemma. We do similar computations, using now the fact, that $ \Arrowvert F_n\Arrowvert_{L_3^1} $ is uniformly bounded.

Again the Povzner inequality \pef{povzner} implies
\begin{equation}\label{mischler15}
\int_{\mathbb{S}^2} d\omega \;\left(\arrowvert v_3\arrowvert^4+\arrowvert v_4\arrowvert^4-\arrowvert v_1\arrowvert^4-\arrowvert v_2\arrowvert^4\right)\leq C_1 \left(\arrowvert v_1\arrowvert^3\arrowvert v_2\arrowvert+\arrowvert v_1\arrowvert\arrowvert v_2\arrowvert^3\right)-c_1\arrowvert v_1\arrowvert^4.
\end{equation}
This, together with the estimates in \pef{mischler1} implies similarly as in the proof of \pef{mischler3}
\begin{equation}\label{mischler16}
\begin{split}
&\frac{(1+\lambda_n)^2}{2}\int_{\mathbb{R}^3}dv_1\int_{\mathbb{R}^3}dv_2\int_{\mathbb{S}^2}d\omega\; B_{el,n}(v_1,v_2) F_n(v_1)F_n(v_2)\left(\arrowvert v_3\arrowvert^4+\arrowvert v_4\arrowvert^4-\arrowvert v_1\arrowvert^4-\arrowvert v_2\arrowvert^4\right)\\
\leq& C_2\left(\Arrowvert F_n\Arrowvert_{L_2^1}\Arrowvert F_n\Arrowvert_{L_3^1}+\Arrowvert F_n\Arrowvert_{L_1^1}\int_{\mathbb{R}^3}dv\left(F_n(v)\right)\right)\arrowvert v\arrowvert^3\min\left(\arrowvert v\arrowvert,n\right)\;+C_3\Arrowvert F_n\Arrowvert_4 \Arrowvert F_n\Arrowvert_{L_1^1}\\
&-C_4\Arrowvert F_n\Arrowvert_{L^1}\int_{\mathbb{R}^3}dv\left(F_n(v)\right)\arrowvert v\arrowvert^4\min\left(\arrowvert v\arrowvert,n\right).\\
\end{split}
\end{equation}
Since $ x^4+y^4\leq (x^2+y^2)^2 $ for the nonelastic velocities we do the following computations 
\begin{equation}\label{mischler17}
\arrowvert \ov_3\arrowvert^4+\arrowvert \ov_4\arrowvert^4-\arrowvert \ov_1\arrowvert^4-\arrowvert \ov_2\arrowvert^4\leq 2\arrowvert \ov_1\arrowvert^2\arrowvert \ov_2\arrowvert^2
\end{equation}
and also
\begin{equation}\label{mischler18}
\arrowvert \ov_1\arrowvert^4+\arrowvert \ov_2\arrowvert^4-\arrowvert \ov_3\arrowvert^4-\arrowvert \ov_4\arrowvert^4\leq C_{\eps_0}\left(1+\arrowvert \ov_3\arrowvert^2+\arrowvert \ov_4\arrowvert^2+\arrowvert \ov_3\arrowvert^2\arrowvert \ov_4\arrowvert^2\right).
\end{equation}
Therefore with a similar computation as in \pef{mischler10}, using the estimates for the non elastic kernels \pef{mischler4} and \pef{mischler5} we also obtain
\begin{equation}\label{mischler19}
\begin{split}
&\int_{\Omega_n}d\ov_1 d\ov_2 d\ov_3 d\ov_4 \delta(\ov_1+\ov_2-\ov_3-\ov_4)\delta(\arrowvert\ov_1\arrowvert^2+\arrowvert\ov_2\arrowvert^2-\arrowvert\ov_3\arrowvert^2-\arrowvert\ov_4\arrowvert^2-2\varepsilon_0)W_{ne}\\
&\;\;\;\;\left(F_n(\ov_1)F_n(\ov_2)-\lambda_nF_n(\ov_3)F_n(\ov_4)\right)\left(\arrowvert \ov_3\arrowvert^4+\arrowvert \ov_4\arrowvert^4-\arrowvert \ov_1\arrowvert^4-\arrowvert \ov_2\arrowvert^4\right)\\
\leq& \tilde{C}\Arrowvert F_n\Arrowvert_{L_3^1}\Arrowvert F_n\Arrowvert_{L_2^1}.
\end{split}
\end{equation}
This time we do not have more terms, since as we can see in \pef{mischler17} and \pef{mischler18} we do not have power bigger than $ 2 $ and so multiplying by terms of power $ 1 $, we can not obtain anything more.
The estimate \pef{lemexistence6} holds also when we consider the fourth moment, and thus putting now estimates \pef{mischler16}, \pef{mischler19} and \pef{lemexistence6} together we can infer that
\begin{equation}\label{mischler20}
\begin{split}
\partial_t\Arrowvert F\Arrowvert_4
\leq& \tilde{C}_2\left(\Arrowvert F_n\Arrowvert_{L_3^1}\Arrowvert F_n\Arrowvert_{L_2^1}+ \Arrowvert F_n\Arrowvert_{L_1^1}\int_{\mathbb{R}^3}dv\left(F_n^1(v)+F_n^2(v)\right)\arrowvert v\arrowvert^3\min\left(\arrowvert v\arrowvert,n\right) \right)+\tilde{C}_4\Arrowvert F_n\Arrowvert_{L_1^1}\Arrowvert F_n\Arrowvert_3\\
&-C_4\Arrowvert F_n\Arrowvert_{L^1}\int_{\mathbb{R}^3}dv\left(F_n(v)\right)\arrowvert v\arrowvert^4\min\left(\arrowvert v\arrowvert,n\right).
\end{split}
\end{equation}
By means of of the Young's Inequality we compute similarly as before in \pef{delta1} and \pef{delta2} for $ \delta>0 $ that
\begin{equation}
\arrowvert v\arrowvert^3=\arrowvert v\arrowvert^{1/3}\arrowvert v\arrowvert^{8/3}\leq K^1_\delta \arrowvert v\arrowvert+\delta\arrowvert v\arrowvert^{4}\leq K^1_\delta \arrowvert v\arrowvert+\delta+\delta\arrowvert v\arrowvert^{4}\min\left(\arrowvert v\arrowvert,n\right)
\end{equation}
and also
\begin{equation}
\arrowvert v\arrowvert^3\min\left(\arrowvert v\arrowvert,n\right)=\min\left(\arrowvert v\arrowvert,n\right)^{1/4}\arrowvert v\arrowvert^3\min\left(\arrowvert v\arrowvert,n\right)^{3/4}\leq K^2_\delta\min\left(\arrowvert v\arrowvert,n\right)+\delta\arrowvert v\arrowvert^4\min\left(\arrowvert v\arrowvert,n\right).
\end{equation}
Again, also in this case, these estimates imply that for every $ \delta $ there exists a constant $ C_\delta $ such that
\begin{equation*}
\int_{\mathbb{R}^3}dv\left(F_n^1(v)+F_n^2(v)\right)\arrowvert v\arrowvert^3\left(1+\min\left(\arrowvert v\arrowvert,n\right)\right)\leq C_\delta\Arrowvert F_n\Arrowvert_{L_1^1}+\delta\int_{\mathbb{R}^3}dv\left(F_n(v)\right)\arrowvert v\arrowvert^4\min\left(\arrowvert v\arrowvert,n\right).
\end{equation*}
Moreover similarly as in \pef{mischler13}, for all $ n\geq M $ we have $ \arrowvert v\arrowvert^4\leq \frac{1}{M}\arrowvert v\arrowvert^4\min\left(\arrowvert v\arrowvert,n\right)+M^4 $. Hence, arguing as for \pef{mischler14} we obtain the differential inequality
\begin{equation}\label{mischler21}
\begin{split}
\partial_t\Arrowvert F_n\Arrowvert_4+&\left(\frac{C_4}{2} \Arrowvert F_n\Arrowvert_{L^1}-C_5\left(\Arrowvert F_n\Arrowvert_{L^1}+\Arrowvert F_n\Arrowvert_{L_2^1}\right)\left(\delta+\frac{1}{M}\right)\right)\\
&\left(\Arrowvert F_n\Arrowvert_4+\int_{\mathbb{R}^3}dv\left(F_n(v)\right)\arrowvert v\arrowvert^4\min\left(\arrowvert v\arrowvert,n\right)\right)\\
\leq&C_{\delta,M}
\end{split}
\end{equation}
for some constant $ C_{\delta,M} $. Taking again $ \delta $ small enough and $ M $ large enough Gr\"onwall implies the Lemma.

\end{proof} 
\end{lemma}
Finally we are ready for the proof of the existence of weak solutions of the new kinetic system \pef{newsystem1}.
\begin{theorem}\label{infinityexistence1}
Let $ (\lambda_0,F_0)\in X_+\cap\tilde{X} $ with $ F_0\log(F_0)\in L^1\left(\mathbb{R}^3\right) $ and let $ E(\lambda_0,F_0):=E_0<\infty $. Assume $ (\lambda_0,F_0)\in \tilde{X} $ satisfies \pef{Assumption}. Let $ T>0 $ then there exist a weak solution $ (\lambda, F)\in C\left([0,T], X\right)$ of the kinetic system \pef{newsystem} for the initial value $ (\lambda_0, F_0) $.

\begin{remark}
Since $ T $ is arbitrary, this theorem implies the existence of a weak solution \linebreak$ (\lambda, F)\in C\left([0,\infty), X\right)$.
\end{remark}

\begin{proof}
The proof follow the usual strategy for this kind of kinetic equations. Such proof structure are the one we can see for example in \cite{arkerydI, arkerydII}.
First of all we start constructing a sequence of functions in $ C^1\left([0,\infty), X\right) $ which converge in some sense a solution of the desired equation. Let $ n\in\mathbb{N} $. We consider the cut-off equation $  \partial_t\begin{pmatrix}\lambda_n\\ F_n\end{pmatrix}~=~T_n\begin{pmatrix}\lambda_n\\ F_n\end{pmatrix} $ with initial value $ (\lambda_0, F_0) $ as in the assumption. Lemma \ref{lemcutoff5} together with the ODE theory implies the existence of unique solutions $ (\lambda_n,F_n)\in C^1\left([0,\infty), X\right) $. Indeed the operator is locally Lipschitz and due to the bounded kernel, we also have energy conservation, which implies the existence of a global solution. Moreover the solution is non-negative and lies in $ X_+ $, due to the construction of the fixed point. Indeed we can write the operator $ T_n $ in its positive and negative part as follows
\begin{equation}\label{Positivity}
T_{i,n}\begin{pmatrix}\lambda\\F\end{pmatrix}= T'_{i,n}\begin{pmatrix}\lambda\\F\end{pmatrix}-\left(\lambda\delta_{i1}+F\delta_{i2}\right) H_{i,n}\begin{pmatrix}\lambda\\F\end{pmatrix},
\end{equation}
where we consider
\begin{equation}\label{Positivity2}
\begin{split}
H_{1,n}\begin{pmatrix}\lambda\\F\end{pmatrix}&=\frac{\lambda+1}{(1+\lambda)+(1-\lambda)^2\int_{\mathbb{R}^3}F}2\int_{\mathbb{R}^3}d\ov_1\int_{\mathbb{R}^3}d\ov_2\int_{\mathbb{S}^2}d\omega B_{ne,n}^{12}(\ov_1,\ov_2)F(\ov_1)F(\ov_2)\\
&+\frac{\lambda^3+1}{(1+\lambda)+(1-\lambda)^2\int_{\mathbb{R}^3}F}2\int_{\mathbb{R}^3}d\ov_3\int_{\mathbb{R}^3}d\ov_4\int_{\mathbb{S}^2}d\omega B_{ne,n}^{34}(\ov_3,\ov_4)F(\ov_3)F(\ov_4)\\
\end{split}
\end{equation}
and 
\begin{equation}\label{Positivity3}
\begin{split}
H_{2,n}\begin{pmatrix}\lambda\\F\end{pmatrix}&=\frac{1}{(1+\lambda)}\left[\int_{\mathbb{R}^3}dv_2\int_{\mathbb{S}^2}d\omega B_{el,n}(v_1,v_2)F(v_2)+\int_{\mathbb{R}^3}d\ov_2\int_{\mathbb{S}^2}d\omega B_{ne,n}^{12}(\ov_1,\ov_2)F(\ov_2)\right.\\&\left.\:\:\:\:\:\:\:\:\:\:\:\:\:\:\:\:\:\:\:\:\:+2\lambda \int_{\mathbb{R}^3}d\ov_4\int_{\mathbb{S}^2}d\omega B_{ne,n}^{34}(\ov_3,\ov_4)F(\ov_4)\right]\\
&+\frac{1-\lambda^2}{(1+\lambda)+(1-\lambda)^2\int_{\mathbb{R}^3}F}2\int_{\mathbb{R}^3}d\ov_1\int_{\mathbb{R}^3}d\ov_2\int_{\mathbb{S}^2}d\omega B_{ne,n}^{12}(\ov_1,\ov_2)F(\ov_1)F(\ov_2).\\
\end{split}
\end{equation}
In this way we see that the map we use for the fixed-point argument preserves positivity. It is indeed given by
\begin{equation}\label{Positivity4}
\begin{split}
A_n\begin{pmatrix}\lambda\\F\end{pmatrix}=&\begin{pmatrix}\lambda_0\exp\left(-\int_0^tH_{1,n}(\lambda, F)ds\right)\\F_0\exp\left(-\int_0^tH_{2,n}(\lambda, F)ds\right)\end{pmatrix}+\begin{pmatrix}\int_{0}^t\exp\left(-\int_s^tH_{1,n}(\lambda, F)d\tau\right)T'_{1,n}(\lambda, F)ds\\\int_{0}^t\exp\left(-\int_s^tH_{2,n}(\lambda, F)d\tau\right)T'_{2,n}(\lambda, F)ds\end{pmatrix}.
\end{split}
\end{equation}

Hence we have global strong solutions and we will work with the sequence $ (\lambda_n,F_n)_n $.\\

From Lemma \ref{lemcutoff4} we know that the sequence $ (\lambda_n, F_n) $ is uniformly bounded in $ C\left([0,T], X\right) $. Moreover, $ (\lambda_n(t))_n$ is uniformly equicontinuous. Indeed the computation in equation \pef{lip3} implies that for all $ n $
\begin{equation}\label{existence1}
\arrowvert T_{1,n}\left(\lambda_n, F_n\right)\arrowvert\leq 2\int_{\mathbb{R}^3}dv\;\left\arrowvert \mathbb{K}_{1,n}\left[F_n,\lambda_N\right]\right\arrowvert\leq 2\overline{C} E_0.
\end{equation}
Therefore the definition of weak solution yields the claimed uniformly equicontinuity since \linebreak $ \arrowvert \lambda_n(t_1)-\lambda_n(t_2)\arrowvert\leq 2\overline{C} E_0\arrowvert t_1-t_2\arrowvert $. Applying Arzela-Ascoli's Theorem we find a subsequence $ \lambda_{n_k} $ and a function $ \lambda~\in~ C\left([0,T], \mathbb{R}\right) $ with $ \lambda\in[0,1) $ such that
\begin{equation}\label{existence2}
\lim\limits_{k\to\infty}\lambda_{n_k}\to\lambda\:\:\:\text{ in }C\left(\left[0,T\right], \Arrowvert . \Arrowvert_{sup}\right).
\end{equation} 

There exist also a subsequence $ F_{n_k} $ and a function $ F\in C\left([0,T], L_2^1\left(\mathbb{R}^3\right)\right) $, $ F\geq 0 $ such that $ F_{n_k}\rightharpoonup F $ for all $ t\in[0,T] $ in $ L^1\left(\mathbb{R}^3\right) $. By Lemma \ref{lemexistence} there exists some constant $ C>0 $ such that $ \Arrowvert F_n\Arrowvert_{L_4^1}\leq C $ for all $ t\in[0,\infty) $ and for all $ n\geq M $, so Lemma \ref{lemcutoff4} implies the existence of a constant $ C $ independent of $ n $ such that $ \int_{\mathbb{R}^3}dv\;F_n\log(F_n)\leq C $ for all $ n\geq M $ and so $ F_n\log(F_n)\in L^1\left(\mathbb{R}^3\right) $. The uniform boundedness of the $ L^1_4 $-norm implies also, as we remarked on page 15, that $ \mathbb{K}_{i,n}\left[F_n,\lambda_n\right]\in L^1_3\left(\mathbb{R}^3\right) $ uniformly for all $ n\geq M $. Therefore, the following estimate holds true
\begin{equation}\label{equicontinuity}
\sup_{t\leq T} \left\Arrowvert \partial_tF\right\Arrowvert_{L_2^1}\leq\sup_{t\leq T}\left( \left\Arrowvert \mathbb{K}_{1,n}\left[F_n,\lambda_n\right]+\mathbb{K}_{2,n}\left[F_n,\lambda_n\right]\right\Arrowvert_{L_{2}^1}+\left\Arrowvert F_n \int_{\mathbb{R}^3}dv\; \mathbb{K}_{1,n}\left[F_n,\lambda_n\right]\right\Arrowvert_{L_2^1}\right)\leq C
\end{equation}
for some constant $ C>0 $ depending on the uniform bound of the $ L^1_4 $-norm and on the initial mass and energy. Thus, Dunford-Pettis Theorem and the uniformly equicontinuity of $ F_n\in C\left([0,T], L^1_2\left(\mathbb{R}^3\right)\right) $ imply the weak-$ L^1 $ compactness of this sequence. Indeed the sequence $ F_n $ is uniformly bounded in $ L^1_2\left(\mathbb{R}^3\right) $ by the conservation of initial energy. Moreover, since the $ L^1_4 $-norm is bounded, as we have seen in Lemma \ref{lemexistence}, the sequence is tight
\begin{equation}
\int_{\arrowvert v\arrowvert\geq R} dv F_n(v)\leq \int_{\arrowvert v\arrowvert\geq R} dv F_n(v)\frac{\left(1+\arrowvert v\arrowvert^2\right)^2}{\left( 1+ R^2\right)^2}\leq \frac{C}{\left( 1+R^2\right)^2}\underset{R\to\infty}{\longrightarrow}0.
\end{equation}
Finally, \pef{lemcutoff4} implies the existence of a constant $ K_0>0 $, which depends only on the initial energy and entropy, such that $ \int_{\mathbb{R}^3}dv\; F_n \log^+(F_n)\leq K_0 $, so that the sequence is also uniformly integrable
\begin{equation}\int_{F_n>e^N}dv F_n(v)\leq \frac{1}{N}\int_{F_n>e^N}dv F_n(v)\log^+\left(F_n(v)\right)\leq \frac{K_0}{N}\underset{N\to\infty}{\longrightarrow}0.		
\end{equation}
Thus, we conclude the weak $ L^1 $-convergence for a subsequence $F_{n_k} $ (which can be chosen common to the one for $ \lambda_n $).

Moreover we easily see that for $ 0\leq l\leq 2 $ the following holds
\begin{equation}
\int_{\arrowvert v\arrowvert\geq R} dv F_n(v)\left(1+\arrowvert v\arrowvert^l\right)\leq K\int_{\arrowvert v\arrowvert\geq R}dv F_n(v)\left(1+\arrowvert v\arrowvert^2\right)^{l/2}\leq \frac{K\; C}{\left(1+R^2\right)^{2-l/2}}\underset{R\to\infty}{\longrightarrow}0.
\end{equation}
This implies 
\begin{equation}\label{existence3}
\int_{\mathbb{R}^3}dv\; F_{n_k}(v)\varphi(v)\underset{k\to\infty}{\longrightarrow}\int_{\mathbb{R}^3}dv\; F(v)\varphi(v)
\end{equation}
for all $ \varphi \frac{1}{1+\arrowvert v\arrowvert^2}\in L^{\infty}\left(\mathbb{R}^3\right) $. Hence, $ F $ satisfies the conservation of mass and momentum and also $ F\in L_2^1\left(\mathbb{R}^3\right) $. Moreover by the uniform convergence of $ \lambda_{n_k} $ we see that $ (\lambda, F) $ still satisfies the conservation of energy. Now it only remains to pass to the limit in the integrals. Equation \pef{existence3} implies the weak $ L^1 $-convergence of
\begin{equation}\label{existence4}
BF_{n_k}(v)F_{n_k}(w)\rightharpoonup BF(v)F(w)
\end{equation}
for a kernel $ B\in\{B_{el}, B_{ne}^{12}, B_{ne}^{34}\} $. It is not difficult to see that $ \rchi_{\arrowvert v-w\arrowvert> R}\leq \rchi_{(1+\arrowvert v\arrowvert^2)\geq R^2}+\rchi_{(1+\arrowvert w\arrowvert^2)\geq R^2} $. Defining $ R_k\in\{n_k, \sqrt{n_k^2-4\varepsilon_0}\} $ and $ R_k^1:=R_k-1 $ we see therefore that for $ k\to \infty $
\begin{equation}\label{existence5}
\begin{split}
0\leq&\int_{\mathbb{R}^3}dv\;\int_{\mathbb{R}^3}dw\;\int_{\mathbb{S}^2}d\omega\; (B-B_{n_k})F_{n_k}(v)F_{n_k}(w)\leq \int_{\arrowvert v-w\arrowvert\geq R_k}dv\;dw\;\int_{\mathbb{S}^2}d\omega\;BF_{n_k}(v)F_{n_k}(w)\\
\leq& C \left[\int_{\arrowvert v\arrowvert\geq R_{k}^1}dv\;(1+\arrowvert v\arrowvert^2)F_{n_k}(v)\right]\leq C \tilde{C} \frac{1}{\left(1+\left(R_k^1\right)^2\right)}\underset{k\to \infty}{\longrightarrow} 0.
\end{split}
\end{equation}
where the constant $ C $ depends on the initial energy and the constant $ \tilde{C} $ is the uniform bound of the $ L^1_4 $-norm of the solutions $ F_{n_k} $ for all $ n_k\geq M$. Since $ B\geq B_{n_k} $, we can conclude from this computation that $ \Arrowvert \left(B-B_{n_k}\right)F_{n_k}F_{n_k}\Arrowvert_{L^1}\to 0 $ as $ k\to\infty $.
This result, together with the uniform convergence of $ \lambda_{n_k} $ and the weak $ L^1 $-convergence of $ F_{n_k} $ implies that $ (\lambda, F)\in C\left([0,T], X\right) $ is the desired weak solution. Moreover, since the $ L_4^1 $-norm of the sequence $ F_{n_k} $ is uniformly bounded, also the solution $ F $ has this property. Indeed for all $ R>0 $
\begin{equation}\label{existence6}
\int_{\arrowvert v\arrowvert\leq R} dv F(v)\left(1+\arrowvert v\arrowvert^2\right)^2=\lim\limits_{k\to\infty}\int_{\arrowvert v\arrowvert\leq R} dv F_{n_k}(v)\left(1+\arrowvert v\arrowvert^2\right)^2\leq \tilde{C}
\end{equation}
\end{proof}
\end{theorem}
\subsection{Uniqueness of weak solutions}
Now that we have shown the existence of weak solutions, we can show that they are unique.
\begin{theorem}\label{infinityuniqueness}
Under the same assumption as in Theorem \ref{infinityexistence1} there is at most one weak solution $ \left(\lambda, F\right) $ for the same initial value problem with $ F\in L^{\infty}\left([0,T], L^1_4\left(\mathbb{R}^3\right)\right) $.

\begin{proof}
Let $ T, (\lambda_0, F_0) $ be as in the assumption. Let us assume that $ (\lambda, F) $ and $ (\mu, G) $ are two weak solutions for the same initial value $ (\lambda_0,F_0) $. Then using the computation in Lemma \ref{lemcutoff5} we see 
\begin{equation*}
\arrowvert \lambda(t)-\mu(t)\arrowvert\leq\left\arrowvert\int_{0}^{t} d\tau\; \left(T_1\left[\lambda, F\right]-T_1\left[\mu, G\right]\right)\right\arrowvert\leq\max(C_1,C_2)\int_{0}^t\left(\arrowvert\lambda-\mu\arrowvert+\Arrowvert F-G\Arrowvert_{L_2^1}\right).
\end{equation*}

We claim that there exists a constant $ C_T>0 $ depending on $ T $ and on $ \sup_{t\leq T}\left(\Arrowvert F\Arrowvert_{L_4^1},\Arrowvert G\Arrowvert_{L_4^1}\right) $ such that 
\begin{equation}\label{uniqueness1}
\int_{\mathbb{R}^3}dv\;(1+\arrowvert v\arrowvert^2)\text{sign}(F-G)\left(\mathbb{K}_1\left[F,\lambda\right]+\mathbb{K}_2\left[F,\lambda\right]-\mathbb{K}_1\left[G,\lambda\right]-\mathbb{K}_2\left[G,\lambda\right]\right)\leq \tilde{C}_T\Arrowvert F-G\Arrowvert_{L_2^1}.
\end{equation}
Indeed, defining $ \phi(v)= \text{sign} (F-G)(v)(1+\arrowvert v\arrowvert^2) $ we can do the following computation, first for the elastic terms and the for the nonelastic one
\begin{equation}\label{uniq1}
\begin{split}
&\int_{\mathbb{R}^3}dv\;(1+\arrowvert v\arrowvert^2)\text{sign}(F-G)\left(\mathbb{K}_1^{elastic}\left[F,\lambda\right]+\mathbb{K}_2^{elastic}\left[F,\lambda\right]-\mathbb{K}_1^{elastic}\left[G,\lambda\right]-\mathbb{K}_2^{elastic}\left[G,\lambda\right]\right)\\
=&\frac{(1+\lambda)^2}{2}\int_{\mathbb{R}^3}dv_1\int_{\mathbb{R}^3}dv_2\int_{\mathbb{S}^2}d\omega\; \arrowvert v_1-v_2\arrowvert \left(F(v_1)F(v_2)-G(v_1)G(v_2)\right)\left(\phi(v_3)+\phi(v_4)-\phi(v_1)-\phi(v_2)\right)\\
\leq&4\int_{\mathbb{R}^3}dv_1\int_{\mathbb{R}^3}dv_2\int_{\mathbb{S}^2}d\omega\;(1+\arrowvert v_1\arrowvert^2)^{1/2}(1+\arrowvert v_2\arrowvert^2)^{1/2} \arrowvert F(v_1)-G(v_1)\arrowvert G(v_2)(1+\arrowvert v_2\arrowvert^2)\\
&+4\int_{\mathbb{R}^3}dv_1\int_{\mathbb{R}^3}dv_2\int_{\mathbb{S}^2}d\omega\;(1+\arrowvert v_1\arrowvert^2)^{1/2}(1+\arrowvert v_2\arrowvert^2)^{1/2}  F(v_1)\arrowvert F(v_2)-G(v_2)\arrowvert(1+\arrowvert v_1\arrowvert^2)\\
\leq& 8  \sup_{t\leq T}\left(\Arrowvert F\Arrowvert_{L_4^1},\Arrowvert G\Arrowvert_{L_4^1}\right)\Arrowvert F-G\Arrowvert_{L_2^1},
\end{split}
\end{equation}
where we used in addition to the estimate estimate for the elastic kernel \pef{usual estimate 3} the relation \pef{equality} we already saw in Lemma \ref{lemcutoff5}. We apply also the triangle inequality, the definition of sign function and the relation between the elastic velocities.
In a very similar way we can bound the nonelastic terms applying this time the estimate for the nonelastic kernels and the relation between the nonelastic velocities. Using the weak formulation we see again
\begin{equation}\label{uniq2}
\begin{split}
&\int_{\mathbb{R}^3}dv\;(1+\arrowvert v\arrowvert^2)\text{sign}(F-G)\left(\mathbb{K}_1^{nonelastic}\left[F,\lambda\right]+\mathbb{K}_2^{nonelastic}\left[F,\lambda\right]-\mathbb{K}_1^{nonelastic}\left[G,\lambda\right]-\mathbb{K}_2^{nonelastic}\left[G,\lambda\right]\right)\\
\leq& C_0(1+2\sqrt{\varepsilon_0})\int_{\mathbb{R}^3}\int_{\mathbb{R}^3}d\ov_1d\ov_2\;\arrowvert F(\ov_1)-G(\ov_1)\arrowvert G(\ov_2)(1+\arrowvert \ov_1\arrowvert^2)^{1/2}(1+\arrowvert \ov_2\arrowvert^2)^{3/2}\\
&+C_0(1+2\sqrt{\varepsilon_0})\int_{\mathbb{R}^3}\int_{\mathbb{R}^3}d\ov_1d\ov_2\;F(\ov_1)\arrowvert F(\ov_2)-G(\ov_2)\arrowvert(1+\arrowvert\ov_1\arrowvert^2)^{3/2}(1+\arrowvert \ov_2\arrowvert^2)^{1/2}\\
&+C_0(1+2\sqrt{\varepsilon_0})(1+\varepsilon_0)\int_{\mathbb{R}^3}\int_{\mathbb{R}^3}d\ov_3d\ov_4\;\arrowvert  F(\ov_3)- G(\ov_3)\arrowvert F(\ov_4)\left(1+\arrowvert\ov_3\arrowvert^2\right)^{1/2}\left(1+\arrowvert\ov_4\arrowvert^2\right)^{3/2}\\
&+C_0(1+2\sqrt{\varepsilon_0})(1+\varepsilon_0)\int_{\mathbb{R}^3}\int_{\mathbb{R}^3}d\ov_3d\ov_4\; G(\ov_3)\arrowvert F(\ov_4)-G(\ov_4)\arrowvert\left(1+\arrowvert\ov_3\arrowvert^2\right)^{3/2}\left(1+\arrowvert\ov_4\arrowvert^2\right)^{1/2}\\
&\leq 4C_0(1+2\sqrt{\varepsilon_0})(1+\varepsilon_0)\Arrowvert F-G\Arrowvert_{L_2^1}\;\sup_{t\leq T}\left(\Arrowvert\mathbb{F}\Arrowvert_{L_4^1},\Arrowvert\mathbb{G}\Arrowvert_{L_4^1}\right).\\\end{split}
\end{equation}
We also see that, since $ \sup_{t\leq T}\left(\Arrowvert F\Arrowvert_{L_4^1},\Arrowvert G\Arrowvert_{L_4^1}\right)<\infty $ in equation \pef{lip2} we can actually find a constant $ \tilde{C}_T $ which depends on the initial energy $ E(0) $ and on $ \sup_{t\leq T}\left(\Arrowvert F\Arrowvert_{L_4^1},\Arrowvert G\Arrowvert_{L_4^1}\right) $ such that 
\begin{equation}\label{uniqueness4}
IV\leq \tilde{C}_T \arrowvert\lambda-\mu\arrowvert.
\end{equation}
This is true because for any collision kernel $ B\in\{B_{el},\; B_{ne}^{12},\;B_{ne}^{34}\} $ we have the estimate \linebreak$ B(v,w)\leq K\left(1+\arrowvert v\arrowvert^2\right)^{1/2}\left(1+\arrowvert w\arrowvert^2\right)^{1/2} $ so that for the key estimates \pef{key1} and \pef{key2} we can use
\begin{equation}
\int_{\mathbb{R}^3}dv\int_{\mathbb{R}^3}dw\int_{\mathbb{S}^2}d\omega\;B(v,w)G(v)G(w)(1+\arrowvert v\arrowvert^2)\leq \Arrowvert G\Arrowvert_{L_2^1}\Arrowvert G\Arrowvert_{L_4^1}\leq E(0)\sup_{t\leq T}\left(\Arrowvert F\Arrowvert_{L_4^1},\Arrowvert G\Arrowvert_{L_4^1}\right).
\end{equation}
Thus, changing the first term of the third estimate $ III $ in equation \pef{lip11} with the above term \pef{uniqueness1} and using also the estimate \pef{uniqueness4} together with an approximating argument for $ \phi $ we see that
\begin{equation}\label{uniqueness2}
\begin{split}
&\partial_t\int_{\mathbb{R}^3}dv\;(1+\arrowvert v\arrowvert^2)\text{sign}(F-G) (F-G)\\
&=\partial_t\int_{\mathbb{R}^3}dv\;(1+\arrowvert v\arrowvert^2)\text{sign}(F-G)\left(T_2\left[\lambda, F\right]-T_2\left[\mu, G\right]\right)\leq\overline{C}_T\left(\arrowvert\lambda-\mu\arrowvert+\Arrowvert F-G\Arrowvert_{L_2^1}\right),
\end{split}
\end{equation}
where $ \overline{C}_T $ depends on the initial energy $E(0) $ and on $ \sup_{t\leq T}\left(\Arrowvert F\Arrowvert_{L_4^1},\Arrowvert G\Arrowvert_{L_4^1}\right) $. These results imply the existence of a constant $ C=\max(C_1,C_2,\overline{C}_T) $ such that
\begin{equation}\label{uniqueness3}
\left\Arrowvert\begin{pmatrix}\lambda\\F\end{pmatrix}-\begin{pmatrix}\mu \\G\end{pmatrix}\right\Arrowvert_X\leq\int_{0}^t d\tau\;C\left\Arrowvert\begin{pmatrix}\lambda \\F\end{pmatrix}-\begin{pmatrix}\mu \\G\end{pmatrix}\right\Arrowvert_X.
\end{equation} 
The integral form of Gr\"onwall's Lemma implies the claim.
\end{proof}
\end{theorem}

\section{The rigorous proof of the derivation of the kinetic equation in the fast radiation limit}
In Section 2 we formally developed the equation \pef{newsystem1}, which describes the behavior of the system immediately after the initial time. In Section 3 a well-posedness theory was proven. In this new section we are ready to prove Theorem \ref{thmrigorous}. We want to prove that the solutions $ (\ol,\oF) $ of the formal derived system \pef{newsystem1} are the limit of the solutions of \pef{epsequation} when $ \eps $ goes to zero. The theorem will be stated and proven in Section 4.3.

For the classical Chapman-Enskog expansion of the Boltzmann equation around the equilibrium, i.e. the Maxwellian, a good convergence result was already shown. In this case the Chapman-Enskog expansion leads in the hydrodynamic limit to the Euler equation (or to the Navier-Stokes equation). For example in \cite{caflisch} it is proven that the smooth solution to the $ \eps $-Boltzmann equation converges for small times to the Maxwellian constructed with the smooth solutions to the corresponding Euler equation. Similar is also the result in \cite{saint-raymond}, which deals with the more general notion of renormalized solutions.

In our case we are dealing with a slightly different situation. The manifold we considered for the generalized Chapman-Enskog expansion is infinitely dimensional. Also the derived kinetic system \pef{newsystem1} is different from the Euler Equation. We will proceed therefore in a similar approach as in the literature but with different tools. We will consider an initial data $ \FF_0 $ in a small neighborhood of $ \mathcal{M}_+ $ and the solution $ \FF^\eps $ of \pef{epsequation} to that initial value. The most important step is that in the neighborhood of the manifold of the steady states $ \mathcal{M}_+ $ we can decompose each vector as the sum of an element in $ \mathcal{M}_+ $ and a remainder. So we will prove that $ \FF^\eps $ converges to $ \oFF $, for the latter being given by the solution of the derived system \pef{newsystem1} with the initial value being the projection of $ \FF_0 $ into $ \mathcal{M}_+ $.
\subsection{The decomposition lemma}
We start with the decomposition lemma, to this end we define a new space 
\begin{equation*}
\mathcal{Y}=\left\{ (F,\alpha,\theta,\lambda)\in L^1_2\left(\mathbb{R}^3\right)\times L^1_2\left(\mathbb{R}^3\right)\times L^1\left(\mathbb{S}^2\right)\times \mathbb{R}\right\}.
\end{equation*}

\begin{lemma}\label{parametrization}
Let $ \FF\in X $ with $ \FF>0 $. Let $ \delta>0 $. Assume there exists some $ \FF_0\in\mathcal{M}_+ $ such that $ \Arrowvert \FF-\FF_0\Arrowvert_\mathcal{X}<\eps$ for some  $\eps>0$ small enough. Then there exist $ \oFF\in\mathcal{M} $ and $ W=\begin{pmatrix}\alpha\\-\alpha\\\theta\end{pmatrix}\in\mathcal{X} $ with $\int_{\mathbb{R}^3}\alpha=\int_{\mathbb{S}^2}\theta$ such that $ \FF=\oFF+W $, $ \Arrowvert W\Arrowvert_\mathcal{X}< \delta $ and $ \ol\in(0,1) $. Moreover, $ \oFF\in\mathcal{M}_+ $.
\begin{proof}
This is an application of the Implicit Function Theorem for Banach Spaces. Let \linebreak$ A=L^1_2\left(\mathbb{R}^3\right)\times L^1_2\left(\mathbb{R}^3\right)\times L^1\left(\mathbb{S}^2\right)\times(-\infty,1) \subset\mathcal{Y} $, which is an open subset of $ \mathcal{Y} $.  We denote by $ H: \mathcal{X}\times A\to\mathcal{Y} $ the function given by
\begin{equation}\label{H}
H\begin{pmatrix}
\FF\\F\\\alpha\\\theta\\\lambda
\end{pmatrix}= \begin{pmatrix}F^1-F-\alpha\\F^2-\lambda F+\alpha\\Q-\frac{\lambda}{1-\lambda}-\theta\\\int_{\mathbb{R}^3}\alpha-\int_{\mathbb{S}^2}\theta	\end{pmatrix}.
\end{equation}
$ H $ is well-defined in $ \mathcal{Y} $. Moreover clearly $ H(\FF_0, 0, 0, \lambda_0)=0 $ by the definition of the manifold $ \mathcal{M}_+ $. By a simple computation using the Taylor expansion we calculate for $ h=(\overline{h},h_4,h_5,h_6,h_7)\in\mathcal{X}\times\mathcal{Y} $
\begin{equation}\label{frechet1}
H(\FF_0+\overline{h},F_0+h_4,h_5,h_6,\lambda+h_7)-H(\FF_0,F_0,0,0,\lambda_0)=\begin{pmatrix}\overline{h}_1-h_4-h_5\\\overline{h}_2-\lambda_0h_4+h_5-F_0h_7\\\overline{h}_3-\frac{1}{(1-\lambda_0)^2}\\\int_{\mathbb{R}^3}h_5-\int_{\mathbb{S}^2}h_6\end{pmatrix}+\begin{pmatrix}0\\-h_7h_4\\\mathcal{O}(h_7^2)\\0\end{pmatrix}.
\end{equation}
This implies that $ H $ is Frèchet differentiable in $ (\FF_0.F_0,0,0,\lambda_0) $. Denoting by $ D_\mathcal{Y}H(\FF_0,F_0,0,0,\lambda_0) $ the Frèchet derivative of $ H $ when the first entrance is constant $ \FF_0 $ we see that it is a bounded linear operator from $ \mathcal{Y} $ to $ \mathcal{Y} $ defined by
\begin{equation}\label{frechet2}
D_\mathcal{Y}H(\FF_0,F_0,0,0,\lambda_0)[(f,g,\omega,x)]= \begin{pmatrix} -f-g\\-\lambda_0g+g-F_0x\\-\frac{1}{(1-\lambda_0)^2}x-\omega\\\int_{\mathbb{R}^3}g-\int_{\mathbb{S}^2}\omega\end{pmatrix}
\end{equation}
This is also an invertible operator. The injectivity is easy to see: we assume that for some element in $ \mathcal{Y} $ we have $ \Arrowvert D_\mathcal{Y}H(\FF_0,F_0,0,0,\lambda_0)[(f,g,\omega,x)]\Arrowvert_\mathcal{Y}=0  $. This implies that each row of the vector defined in \pef{frechet2} must be equal $ 0 $ in $\mathcal{Y}$. Therefore we have first $ f=-g $ almost everywhere, which implies also $ f=\frac{F_0}{1+\lambda_0}x $ almost everywhere. From the relation $ \frac{1}{(1-\lambda_0)^2}x=-\omega $ for almost every $ n\in\mathbb{S}^2 $ we conclude, since $ x\in\mathbb{R} $, that also $ \omega $ has to be constant almost everywhere. Hence, putting the new definitions for $ g=-\frac{F_0}{1+\lambda_0}x $ and the one for $ \omega $ into $ \int_{\mathbb{R}^3}g-\int_{\mathbb{S}^2}\omega=0 $ we conclude that $ x=0 $. This is because by assumption $ \frac{\int_{\mathbb{R}^3}F_0}{1+\lambda_0}>0 $ and $ \frac{1}{(1-\lambda_0)^2}>0 $ and because of the computation
\begin{equation}\label{frechet3}
0=\int_{\mathbb{R}^3}g-\int_{\mathbb{S}^2}\omega= -\frac{\int_{\mathbb{R}^3}F_0}{1+\lambda_0}x-\frac{1}{(1-\lambda_0)^2}x.
\end{equation}
Since we have proved that $ x=0 $, this implies $ \omega=0 $ almost everywhere and also $ g=f=0 $ almost everywhere. Thus, the claimed injectivity holds true.

For the surjectivity we have to do a similar computation. Let $ \left(F,\;\alpha,\;\theta,\;\lambda\right)\in\mathcal{Y} $. We want to find a vector $ (f,g,\omega,x)\in\mathcal{Y} $ such that $\left\Arrowvert \begin{pmatrix} -f-g\\-\lambda_0g+g-F_0x\\-\frac{1}{(1-\lambda_0)^2}x-\omega\\\int_{\mathbb{R}^3}g-\int_{\mathbb{S}^2}\omega\end{pmatrix}-\begin{pmatrix}F\\\alpha\\\theta\\\lambda\end{pmatrix}\right\Arrowvert_\mathcal{Y}=0 $. From this equation we have first
\begin{equation}\label{frechet4}
-g=F+f\;\;\;\text{a.e}
\end{equation}
and therefore also 
\begin{equation}\label{frechet5}
f=\frac{-\alpha-F_0x-F}{1+\lambda_0}\;\;\;\text{a.e}.
\end{equation}
From the equation $ -\frac{1}{(1-\lambda_0)^2}x-\omega=\theta $ for almost every $ n\in\mathbb{S}^2 $ we see that $ \omega+\theta= -\frac{1}{(1-\lambda_0)^2}x $ almost everywhere, which implies that $ \omega+\theta $ is constant. So we can conclude with the following computation
\begin{equation}\label{frechet6}
\lambda=\int_{\mathbb{R}^3}g-\int_{\mathbb{S}^2}\omega=\frac{\int_{\mathbb{R}^3}\alpha-\lambda_0F}{1+\lambda_0}-\int_{\mathbb{S}^2}\theta+x\left(\frac{\int_{\mathbb{R}^3}F_0}{1+\lambda_0}+\frac{1}{(1-\lambda_0)^2}\right).
\end{equation}
By the assumption on $ \FF_0 $ the coefficient of $ x $ is positive, which implies the resolvibility of this equation. Having found the value of $ x $ we can find the desired vector $ (f,g,\omega,x)\in\mathcal{Y} $.

We have just proved the assumption for the implicit function theorem. Therefore there exist open neighborhoods $ V\subset\mathcal{X} $ of $ \FF_0 $ and $ U\subseteq A\subset\mathcal{Y} $ of $ (F_0,0,0,\lambda_0) $ and a unique continuous function $ \Phi:V\to U $ such that $ \Phi(\FF_0)=(F_0,0,0,\lambda_0) $ and such that for every $ (\FF,F,\alpha,\theta,\lambda)\in V\times U $ we have $ H(\FF,F,\alpha,\theta,\lambda)=0 $ if and only if $ \Phi(\FF)=(F,\alpha, \theta,\lambda) $. This implies that every $ \FF\in V $ can be uniquely written as \linebreak$ \FF=\oFF+W $ for $ \oFF\in\mathcal{M} $ and $ W=(\alpha,-\alpha,\theta)\in\mathcal{X} $ with $ \int_{\mathbb{R}^3}\alpha=\int_{\mathbb{S}^2}\theta $. Now let us consider \linebreak$ \tilde{U}=U\cap\left(L^1_2(\mathbb{R}^3)\times B^{L^1_2}_{\delta/3}(0)\times B^{L^1}_{\delta/3}(0)\times (0,1)\right)$, which is still an open neighborhood of $ (F_0,0,0,\lambda_0) $. Here we denote by $ B^Z_{r}(z_0) $ the ball of radius $ r $ and center $ z_0\in Z $ with respect to the norm of the Banach space $ Z $. Then taking $ \eps>0 $ small enough such that
$ B^\mathcal{X}_\eps\left(\FF_0\right)\subset \Phi^{-1}\left(\tilde U\right)\subset V $ we can conclude the lemma. By the continuity of $ \Phi $ the preimage of an open set is still an open set, in this case also a neighborhood of $ \FF_0 $.	

Now that we have the decomposition, it is not difficult to see that $ \oFF>0 $. The assumption that $ \FF>0 $ implies that $ \alpha>0 $ on the set where $ F\leq 0 $.  We already have seen that $ \ol\in(0,1) $. Hence, on the set where $ F\leq 0 $ we obtain $ \ol\,\oF-\alpha<0 $. But this is possible only on a set of measure zero. Therefore we conclude the claim.
\end{proof}
\end{lemma}
\bigskip
This Lemma can be generalized also for continuous functions with value in the Banach Space $\mathcal{X}$.
\begin{corollary}\label{parametrization2}
Let $ \FF\in C\left([0,T], \mathcal{X}\right) $ with $ \FF>0 $. Let $ \delta>0 $. Assume there exists some $ \FF_0\in\mathcal{M}_+ $ such that $ \Arrowvert \FF-\FF_0\Arrowvert_\mathcal{X}<\eps$ for all $ t\in[0,T] $ for some  $\eps>0$ small enough. Then there exist $ \oFF\in C\left([0,T], \mathcal{X}\right) $ with $ \oFF(t)\in\mathcal{M} $ and $ W=\begin{pmatrix}\alpha\\-\alpha\\\theta\end{pmatrix}\in C\left([0,T], \mathcal{X}\right) $ with $\int_{\mathbb{R}^3}\alpha=\int_{\mathbb{S}^2}\theta$ such that $ \FF=\oFF+W $, $ \sup_{t\leq T}\Arrowvert W\Arrowvert_\mathcal{X}< \delta $ and $ \ol(t)\in(0,1) $ for all $ t\in[0,T] $. Moreover $ \oFF>0 $.
\begin{proof}
The proof is analogous as the one of Lemma \ref{parametrization}. Note that here $ \FF_0 $ is constant in time. In this case we consider the open set $ A=C\left([0,T],L^1_2\left(\mathbb{R}^3\right)\times L^1_2\left(\mathbb{R}^3\right)\times L^1\left(\mathbb{S}^2\right)\times(-\infty,1) \right)\subset C\left([0,T], \mathcal{Y}\right)  $ and the function $ H:C\left([0,T], \mathcal{X}\right)\times A\to C\left([0,T], \mathcal{Y}\right) $ is defined as before in \pef{H}. It is well-defined and Frèchet differentiable at $ (\FF_0,F_0,0,0,\lambda_0) $ and as in the Lemma \ref{parametrization} its derivative is a bounded invertible operator. By the Implicit Function Theorem we have again the existence of open neighborhoods $ V\subset C\left([0,T], \mathcal{X}\right) $ of $ \FF_0 $ and $ U\subseteq A\subset C\left([0,T], \mathcal{Y}\right) $ of $ (F_0,0,0,\lambda_0) $ and of a unique continuous function $ \Phi:V\to U $ with the same property as before. Defining now  $ \tilde{U}=U\cap\left(C\left([0,T],L^1_2(\mathbb{R}^3)\right)\times B^{C\left([0,T],L^1_2\right)}_{\delta/3}(0)\times B^{C\left([0,T],L^1\right)}_{\delta/3}(0)\times B^{C\left([0,T],(-\infty,1)\right)}_{1/2}(1/2)\right)$ we conclude as in Lemma \ref{parametrization}.
\end{proof}
\end{corollary}

Lemma \ref{parametrization} and Corollary \ref{parametrization2} are fundamental for the proof of Theorem \ref{thmrigorous}, which is the goal of this section. Before moving to that we need some useful notation and technical results. 
\subsection{Some technical results}
\begin{definition}
We denote by $ \kappa_0=\int_{\mathbb{R}^3} \left(F_0^1+F_0^2\right) $ the initial mass of the particles. 
\end{definition}
\begin{remark}
	We recall that both the solutions to the kinetic equation \pef{kinsystem} as to the derived kinetic equation\pef{newsystem1} conserve the initial mass. Moreover they conserve also the initial energy $ E_0 $.
\end{remark}
We also recall, that solutions to the derived kinetic equation \pef{newsystem1} for initial values with bounded $ L^1_4 $-norm have also uniformly bounded $ L^1_4 $-norm. Moreover it can be proved, that strong solutions to the equation \pef{kinsystem} for initial values bounded in $ L^1_4 $ have also the fourth moment bounded, by a constant which only depends on $ E_0 $ and $ \kappa_0 $. This can be proven using the ODE's theory in general Banach spaces. We will not prove the whole well-posedness theory here. However, the most important step for the uniform boundedness of the $ L^1_4 $-norm is given by Lemma \ref{kappa_4}, which is proved for the homogeneous Boltzmann Equation in Lemma 6.5 of \cite{Bressan}. Before taking examine that proof carefully, we state an helpful result.
\begin{lemma}\label{Bressan}\textnormal{(\cite{Bressan}, page 17-19)}
	
	For each $ s>2 $ and any $ \lambda\in\left(\frac{4}{s+2},\;1\right)  $, one can find a constant $ \alpha $ large enough so that the following holds. If $ \arrowvert\xi\arrowvert\geq\alpha\left(1+\arrowvert\xi_*\right) $, then
	\begin{equation}\label{6.4}
	\frac{2}{\pi\arrowvert\xi-\xi_*\arrowvert^2}\int_{S_{\xi\xi_*}}\left(1+\arrowvert\xi'\arrowvert^2\right)^{s/2}d\sigma\leq \lambda\left[\left(1+\arrowvert\xi\arrowvert^2\right)^{s/2}+\left(1+\arrowvert\xi_*\arrowvert^2\right)^{s/2}\right],
	\end{equation}
	where $ d\sigma $ denotes the surface area on $ S_{\xi\xi_*} $, which is the sphere with diameter the segment, joining $ \xi $ with $ \xi_* $.	
\end{lemma}
Now we prove the desired result for the uniform $ L^1_4 $-boundedness of the sequence $ \FF^\eps $ of strong solution to the equation \pef{epsequation}.
\begin{lemma}\label{kappa_4}
Let $ \FF $ be a solution to the kinetic system \pef{kinsystem}. Assume $ \Arrowvert F\Arrowvert_{L_2^1}\leq \kappa_2 $, $\Arrowvert F\Arrowvert_{L^1}=\kappa_0 $ and that $ \Arrowvert F\Arrowvert_{L_4^1} $ is bounded. Then there exists a constant $ \overline{\kappa} $ depending on $ \kappa_2 $ and on $ \kappa_0 $ such that 
\begin{equation}\label{6.5}
\frac{d}{dt}\Arrowvert F\Arrowvert_{L_4^1}\leq 0\;\;\;\;\;\;\;\;\; whenever\;\;\; \Arrowvert F\Arrowvert_{L_4^1}\geq\overline{\kappa}. 
\end{equation} 
\begin{proof}
We follow almost one-to-one the proof of Lemma 6.5 of \cite{Bressan}. The only difference is that we are working with a system and also with some nonelastic interactions. Using equations \pef{mischler17} and \pef{mischler18} together with the relation of the nonelastic velocities as in \pef{nonelasticrel1} and estimating the kernel $ B_{ne}^{ij}(\ov,\ow)\leq C_{\eps_0}\left( \arrowvert\ov\arrowvert+\arrowvert\ow\arrowvert\right) $ we see that there exists some constant $ \tilde{C} $ such that defining $ \varphi^i(v)=\left(1+\arrowvert v\arrowvert^2\right)^2 $ the following is true that for every such $ \FF $
\begin{equation}\label{6.5bis}
\begin{split}
&\int_{\mathbb{R}^{12}} d\ov_1 d\ov_2 d\ov_3 d\ov_4 \delta(\ov_1+\ov_2-\ov_3-\ov_4)\delta(\arrowvert\ov_1\arrowvert^2+\arrowvert\ov_2\arrowvert^2-\arrowvert\ov_3\arrowvert^2-\arrowvert\ov_4\arrowvert^2-2\epsilon_0)W_{ne}(\ov_1,\ov_2;\ov_3,\ov_4)\\
&\left[\oF_1^1\oF_2^1\left(\varphi^2(\ov_3)+\varphi^1(\ov_4)-\varphi^1(\ov_1)-\varphi^1(\ov_2)\right)+\oF_3^2\oF_4^1\left(\varphi^1(\ov_1)+\varphi^1(\ov_2)-\varphi^2(\ov_3)-\varphi^1(\ov_4)\right)	\right]\\
&\leq \tilde{C} \Arrowvert F\Arrowvert_{L_2^1}\Arrowvert F\Arrowvert_{L_4^1}.
\end{split}
\end{equation}
Let us take $ \lambda $ and $ \alpha $ according to Lemma 6.4 of \cite{Bressan} as written in Lemma \ref{Bressan}. Without loss of generality we assume $ \alpha>1 $. Let $ \rho=\sqrt{\frac{2\kappa_2}{\kappa_0}} $. By the conservation of mass we obtain that
\begin{equation}\label{bressan2}
\int_{\arrowvert v\arrowvert\leq \rho}dv\;\left[F^1(v)+F^2(v)\right]>\frac{\kappa_0}{2}.
\end{equation}
Moreover let us define $ \Gamma':=\left\{(v_1,v_2):\arrowvert v_2\arrowvert\leq \rho,\;\arrowvert v\arrowvert\geq\alpha(1+\rho)\right\}\subset \mathbb{R}^3\times\mathbb{R}^3 $. We denote also $ \Gamma:=\mathbb{R}^3\times\mathbb{R}^3\setminus\Gamma' $. It is important to notice that on $ \Gamma' $ the following holds 
\begin{equation}\label{bressan4}
\arrowvert v_1-v_2\arrowvert\geq \arrowvert v_1\arrowvert-\arrowvert v_2\arrowvert=\left(\arrowvert v_1\arrowvert+1\right)-\left(\arrowvert v_2\arrowvert+1\right)\geq \left(1+\arrowvert v_1\arrowvert^2\right)^{1/2}-\left(\arrowvert v_2\arrowvert+1\right).
\end{equation}In this way we do the following computation following the same calculation as in \cite{Bressan}
\begin{equation}\label{6.5tris}
\begin{split}
\frac{d}{dt}\bigg\Arrowvert \frac{F}{\kappa_0}&\bigg\Arrowvert_{L_4^1}\leq \frac{1}{2\kappa_0}\sum_{i=1}^2\sum_{j=1}^2\int_{\mathbb{R}^3}dv_1\int_{\Gamma\cup\Gamma'}dv_2\int_{\mathbb{S}^2}d\omega\; \arrowvert v_1-v_2\arrowvert F_1^iF_2^j\left(\varphi^i(v_3)+\varphi^i(v_4)-\varphi^i(v_1)-\varphi^i(v_2)\right)\\
&+\frac{1}{\kappa_0}\int_{\mathbb{R}^{12}} d\ov_1 d\ov_2 d\ov_3 d\ov_4 \delta(\ov_1+\ov_2-\ov_3-\ov_4)\delta(\arrowvert\ov_1\arrowvert^2+\arrowvert\ov_2\arrowvert^2-\arrowvert\ov_3\arrowvert^2-\arrowvert\ov_4\arrowvert^2-2\epsilon_0)W_{ne}(\ov_1,\ov_2;\ov_3,\ov_4)\\
&\;\;\;\;\;\;\;\;\;\;\left[\oF_1^1\oF_2^1\left(\varphi^2(\ov_3)+\varphi^1(\ov_4)-\varphi^1(\ov_1)-\varphi^1(\ov_2)\right)+\oF_3^2\oF_4^1\left(\varphi^1(\ov_1)+\varphi^1(\ov_2)-\varphi^2(\ov_3)-\varphi^1(\ov_4)\right)	\right]\\
\leq& \frac{C}{\kappa_0} \Arrowvert F\Arrowvert_{L_2^1}\Arrowvert F\Arrowvert_{L_4^1}-\left(1-\lambda\right)\frac{1}{2}\sum_{i=1}^2\sum_{j=1}^2\int\int_{\Gamma'}dv_1dv_2 \;\arrowvert v_1-v_2\arrowvert \varphi (v_1)F_1^iF_2^j\\
\leq &\frac{C}{\kappa_0} \Arrowvert F\Arrowvert_{L_2^1}\Arrowvert F\Arrowvert_{L_4^1}-\frac{\left(1-\lambda\right)}{2\kappa_0}\sum_{i=1}^2\sum_{j=1}^2\int\int_{\Gamma'}dv_1dv_2 \left(1+\arrowvert v_1\arrowvert^2\right)^{1/2}\left(1+\arrowvert v_1\arrowvert^2\right)^2\;F_1^iF_2^j\\
&+\frac{\left(1-\lambda\right)}{2\kappa_0}\sum_{i=1}^2\sum_{j=1}^2\int\int_{\Gamma'}dv_1dv_2 \left(\arrowvert v_2\arrowvert+1\right)\left(1+\arrowvert v_1\arrowvert^2\right)^2\;F_1^iF_2^j\\
\leq&\frac{C_1}{\kappa_0}\Arrowvert F\Arrowvert_{L_2^1}\Arrowvert F\Arrowvert_{L_4^1}-\frac{\left(1-\lambda\right)}{2\kappa_0}\sum_{i=1}^2\sum_{j=1}^2\int\int_{\Gamma'}dv_1dv_2 \left(1+\arrowvert v_1\arrowvert^2\right)^{1/2}\left(1+\arrowvert v_1\arrowvert^2\right)^2\;F_1^iF_2^j\\
\leq&\frac{C_1}{\kappa_0}\Arrowvert F\Arrowvert_{L_2^1}\Arrowvert F\Arrowvert_{L_4^1}-\frac{\left(1-\lambda\right)}{4}\int_{\arrowvert v_1\arrowvert\geq \alpha(1+\rho)}dv_1 \left(1+\arrowvert v_1\arrowvert^2\right)^{5/2}\;\left(F_1^1+F_1^2\right).\\
\end{split}
\end{equation}
We used here also the weak formulation with $ \varphi^i(v)=\left(1+\arrowvert v\arrowvert^2\right)^2 $, the estimate \pef{6.5bis} for the nonelastic part and the Lemma \ref{Bressan} together with the fact, that the measure of $ S_{\xi\xi_*} $ is equal to $ \pi \arrowvert \xi-\xi_*\arrowvert^2 $. To conclude the Lemma we apply Jensen's inequality for the convex function $ h(x)=x^{5/4} $ and the probability measure $ \frac{\left(F^1(v)+F^2(v)\right)}{\kappa_0}dv $. Hence we see
\begin{equation}\label{bressan5}
\begin{split}
\kappa_0^{-1/4}\left(\int_{\mathbb{R}^3}dv\;\left(1+\arrowvert v\arrowvert^2\right)^{2}\left(F^1(v)+F^2(v)\right)\right)^{5/4}\leq&\int_{\mathbb{R}^3}dv\;\left(1+\arrowvert v\arrowvert^2\right)^{5/2}\left(F^1(v)+F^2(v)\right)\\
\leq&\int_{\arrowvert v\arrowvert\geq \alpha(1+\rho)}dv\;\left(1+\arrowvert v\arrowvert^2\right)^{5/2}\left(F^1(v)+F^2(v)\right)\\
&+\kappa_0(1+\alpha^2(1+\rho)^2)^{5/2}.
\end{split}
\end{equation}
Therefore putting all estimates together we obtain
\begin{equation}\label{bressan6}
\frac{d}{dt}\Arrowvert F\Arrowvert_{L_4^1}\leq C_1\Arrowvert F\Arrowvert_{L_2^1}\Arrowvert F\Arrowvert_{L_4^1}-\frac{\left(1-\lambda\right)\kappa_0^{3/4}}{4} \Arrowvert F\Arrowvert_{L_4^1}^{5/4}+\frac{\left(1-\lambda\right)\kappa_0^2}{4}(1+\alpha^2(1+\rho)^2)^{5/2}.
\end{equation}
The right hand side of this estimate is non-positive provided $ \Arrowvert F\Arrowvert_{L_4^1} $ is large enough. This conclude the Lemma.		
\end{proof}
\end{lemma}

With Lemma \ref{kappa_4} we can conclude that for a constant $ \kappa_4>\overline{\kappa} $ we have that for any $ t\geq 0 $ the $ L^1_4 $-norm of the solution is uniformly bounded $ \Arrowvert F\Arrowvert_{L_4^1}\leq \kappa_4 $. Note that this computation never used the radiative term of the equation \pef{compactequation}. This means that for an initial value with fourth moment bounded by $ \kappa_4 $ for every $ \eps>0 $ any strong solution $ \FF^\eps $ has the same bound $ \Arrowvert F^\eps\Arrowvert_{L_4^1}\leq \kappa_4 $. It is very important to notice that this constant does not depend on $ \eps $.
\subsection{The convergence theorem}
We are ready now for the rigorous prove of the derivation of the kinetic equation \pef{newsystem1}.
\begin{theorem}\label{thmrigorous1}
Let $ \eps>0 $. Let $ \FF_0\in\mathcal{M}_+ $, $ \FF_0>0 $ satisfying assumption \pef{Assumption} with mass $ \kappa_0>0 $, energy $ E_0:=E\left(\FF_0\right)\leq \frac{1}{2} $ and $ F_0\in L^1_4 $.

Let $ \delta_0:= \frac{2\eps_0\kappa_0}{64(2\eps_0+E_0)}>0 $ and let $ \delta<\delta_0 $ small enough. We define $ \tFF_0=\FF_0+W_0\in X $ such that $ \FF_0^\eps>0 $, $ W_0=\begin{pmatrix}\alpha_0\\-\alpha_0\\\theta_0\end{pmatrix}\in\mathcal{X} $ with $ \int_{\mathbb{R}^3}\alpha_0=\int_{\mathbb{S}^2}\theta_0 $ and $ \Arrowvert W_0\Arrowvert_\mathcal{X}<\delta $. 

Let $ \FF^\eps\in C^1\left([0,\infty),\mathcal{X}\right) $ be the solution to the equation \pef{epsequation} for the initial value $ \tFF_0 $. 

Let $ \oFF=\begin{pmatrix}\oF\\\ol\;\oF\\\frac{\ol}{1-\ol}\end{pmatrix}\in\mathcal{M}_+ $ with $ (\ol,\oF)\in C\left([0,\infty), [0,1)\times L^1_2\left(\mathbb{R}^3\right)\right) $ being the solution of the kinetic equation \pef{newsystem1} for the initial value $ \FF_0 $.

Then there exists some $ t_0>0 $ such that for all $ s>0 $ we have
\begin{equation}\label{thmrigorous2}
\lim\limits_{\eps\to 0}\sup_{s\leq t\leq t_0}\Arrowvert \FF^\eps-\oFF\Arrowvert_\mathcal{X}= 0.
\end{equation}
\begin{proof}
We proceed in several steps. We first decompose $ \FF^\eps $ according to Lemma \ref{parametrization}. Then using the equations for the functions in the manifold and for the rest term we estimate the norm of $ \FF^\eps-\oFF $.\\

\hspace{-15pt}\textbf{STEP 1: Decomposition}\\
By the continuity of $ \Arrowvert \FF^\eps-\FF_0\Arrowvert_\mathcal{X} $ and by the assumption there exists some $ t_1>0 $ such that 
\begin{equation}\label{deco1}
\Arrowvert \FF^\eps-\FF_0\Arrowvert_\mathcal{X}<\delta_1<\delta_0 \;\;\;\;\;\;\text{ for all }t\leq t_1
\end{equation}
for some $ \delta_1>\delta $ small enough. $ \delta_1 $ is chosen so that by Lemma \ref{parametrization} and its Corollary \ref{parametrization2} there exist unique vectors $ \oFF^\eps, W_\eps\in C\left([0,t_1], \mathcal{X}\right) $ such that $ \FF^\eps=\oFF^\eps+W_\eps $ and $ \oFF^\eps\in\mathcal{M} $ and $ W_\eps=(\alpha_\eps,-\alpha_\eps,\theta_\eps) $ with $ \int_{\mathbb{R}^3} \alpha_\eps=\int_{\mathbb{S}^2}\theta_\eps $. Moreover they satisfy 
\begin{equation}\label{deco2}
\Arrowvert W_\eps\Arrowvert_\mathcal{X}\leq \delta_0 \;\;\;\;\;\;\text{ and }\;\;\;\;\;\lambda_\eps\in(0,1).
\end{equation}
We can also take $ \delta_1 $ small enough so that by the uniqueness of the decomposition we have that \linebreak$ \oFF^\eps(0)=\oFF(0)=\FF_0 $. Therefore there exists by continuity also some $ 0<t_0\leq t_1 $ such that 
\begin{equation}\label{deco3}
\Arrowvert \oFF^\eps-\oFF\Arrowvert_{\mathcal{X}}<\frac{1}{16}\frac{2\eps_0}{2\eps_0+E_0}.
\end{equation}
\\
\textbf{Some important remarks.}\\
We recall that since both initial data $ \tFF_0 $ and $ \FF_0 $ are strictly positive, then also $ \FF^\eps $ and $ \oF $ are strictly positive for every $ t\geq0 $ and every $ \eps $. This implies that the mass of $ \FF^\eps $ is actually the mass of $ \oFF^\eps $
\begin{equation}\label{remthm}
\kappa_0=\int_{\mathbb{R}^3} F_\eps^1(v)+F^2_\eps(v)\;dv=\int_{\mathbb{R}^3} \oF^\eps(v)+\alpha_\eps(v)+\ol_\eps\oF^\eps(v)-\alpha_\eps(v)\;dv=(1+\ol_\eps)\int_{\mathbb{R}^3}\oF^\eps(v)\;dv,
\end{equation}
which by the positivity of $ \ol_\eps $ implies $ \int_{\mathbb{R}^3}\oF_\eps(v)dv\in\left(\frac{\kappa_0}{2},\kappa_0\right) $. In the same way we can see that the energy of $ \oFF^\eps $ is the same as the energy of $ \FF^\eps $ which is constant $ E_0 $. Moreover the $ L^1_2 $-norm of both $ \oFF^\eps $ and $ W_\eps $ is bounded
\begin{equation}\label{remthm2}
\begin{split}
\int_{\mathbb{R}^3} dv \arrowvert \oF^\eps\arrowvert(1+\arrowvert v\arrowvert^2)\leq& (1+\ol_\eps)\int_{\mathbb{R}^3} dv \arrowvert \oF^\eps\arrowvert(1+\arrowvert v\arrowvert^2)=\int_{\mathbb{R}^3} dv \arrowvert \oF^\eps+\alpha_\eps+\ol_\eps\oF^\eps-\alpha_\eps\arrowvert(1+\arrowvert v\arrowvert^2)\\
\leq&\int_{\mathbb{R}^3} dv \arrowvert \oF^\eps+\alpha_\eps\arrowvert+\arrowvert\ol_\eps\oF^\eps-\alpha_\eps\arrowvert(1+\arrowvert v\arrowvert^2)=\Arrowvert F^\eps\Arrowvert_{L_2^1}\leq E_0.
\end{split}
\end{equation}
Similarly we compute $ \int_{\mathbb{R}^3}dv \arrowvert\alpha_\eps\arrowvert(1+\arrowvert v\arrowvert^2)\leq\int_{\mathbb{R}^3}dv \arrowvert\oF^\eps+\alpha_\eps\arrowvert+\arrowvert \oF^\eps\arrowvert(1+\arrowvert v\arrowvert^2)\leq 2E_0  $.
For simplicity we will call $ \kappa_4 $ the maximum between the uniform bound of the $ L^1_4 $-norm of $ \FF^\eps $ and $ \oFF $. We know already that this is a constant independent of $ \eps $. With the very same calculation as in the derivation of equation \pef{remthm2} we see that $ \int_{\mathbb{R}^3} dv \arrowvert \oF^\eps\arrowvert(1+\arrowvert v\arrowvert^2)^2\leq \kappa_4 $ and therefore also $ \int_{\mathbb{R}^3}dv \arrowvert\alpha_\eps\arrowvert(1+\arrowvert v\arrowvert^2)^2\leq 2\kappa_4 $.\\

\hspace{-15pt}\textbf{STEP 2: The equation for the decomposition.}
\\
In order to estimate the norm of $ \FF^\eps-\oFF $ we start looking for the equation that $ \oFF^\eps $ and $ W_\eps $ satisfy. We define two operators. Let $ \FF\in\mathcal{M} $, then we define for $ W\in\mathcal{X} $
\begin{equation}\label{LF}
L_{\FF}(W)=\begin{pmatrix}\int_{\mathbb{S}^2} dn\; \left[W^2(v)+\frac{\lambda}{1-\lambda}\left(W^2(v)-W^1(v)\right)+W^3(n)(\lambda-1)F(v)\right]\\-\int_{\mathbb{S}^2} dn\; \left[W^2(v)+\frac{\lambda}{1-\lambda}\left(W^2(v)-W^1(v)\right)+W^3(n)(\lambda-1)F(v)\right]\\\int_{\mathbb{R}^3} dv\; \left[W^2(v)+\frac{\lambda}{1-\lambda}\left(W^2(v)-W^1(v)\right)+W^3(n)(\lambda-1)F(v)\right]\end{pmatrix}
\end{equation}
and 
\begin{equation}\label{L}
\mathcal{L}(W)=\begin{pmatrix}\int_{\mathbb{S}^2} dn\; \left[W^3(n)\left(W^2(v)-W^1(v)\right)\right]\\-\int_{\mathbb{S}^2} dn\; \left[W^3(n)\left(W^2(v)-W^1(v)\right)\right]\\\int_{\mathbb{R}^3} dv\; \left[W^3(n)\left(W^2(v)-W^1(v)\right)\right]\end{pmatrix}.
\end{equation}
Now, since $ \mathbb{R}\left[\oFF^\eps,\oFF^\eps\right]=0$ using the definition of the decomposition and the just new defined operators we see that for the radiation part we have $ \mathcal{R}[\FF^\eps,\FF^\eps]=L_{\oFF^\eps}(W_\eps)+\mathcal{L}(W_\eps) $ for all $ t\in[0,t_0] $. We therefore work with the following equation
\begin{equation}\label{step2}
\partial_t\oFF^\eps+\partial_tW_\eps=\frac{1}{\eps}\left(L_{\oFF^\eps}(W_\eps)+\mathcal{L}(W_\eps)\right)+\KK\left[\oFF^\eps+W_\eps,\oFF^\eps+W_\eps\right].
\end{equation}
Our aim is now to find the equation which $ \oFF^\eps $ satisfies, and then subtracting it from equation \pef{step2} we obtain the equation satisfied by $ W_\eps $. Let therefore $ \varphi\in C^1\left([0,t_0], C_c^\infty\left(\mathbb{R}^3\right)\right) $ and let $ \eta\in C\left([0,t_0], \mathbb{R}\right) $ be both arbitrary. We can define a linear bounded operator on $ \mathcal{X} $ by
\begin{equation}\label{step22}
A_{\varphi,\eta}(f,g,h)=\int_{\mathbb{R}^3} \varphi f \;dv+\int_{\mathbb{R}^3} \varphi g+(1-\ol_\eps)^2\eta g \;dv+\int_{\mathbb{S}^2} (1-\ol_\eps)^2\eta h\;dn.
\end{equation}
For every arbitrary pair $ \varphi, \eta $ we have $ A_{\varphi,\eta}\left[L_{\oFF^\eps}(W)\right]=0 $ and $ A_{\varphi,\eta}\left[\mathcal{L}(W)\right]=0 $. Indeed it is easy to compute
\begin{equation}\label{step23}
\begin{split}A_{\varphi,\eta}\left[L_{\oFF^\eps}(W)\right]&=\int_{\mathbb{R}^3}\int_{\mathbb{S}^2} \varphi \left[W^2+\frac{\ol_\eps}{1-\ol_\eps}\left(W^2-W^1\right)+W^3(\ol_\eps-1)\oF^\eps\right]\;dn \;dv\\&-\int_{\mathbb{R}^3}\int_{\mathbb{S}^2} \varphi \left[W^2+\frac{\ol_\eps}{1-\ol_\eps}\left(W^2-W^1\right)+W^3(\ol_\eps-1)\oF^\eps\right]\;dn \;dv\\
&-\int_{\mathbb{R}^3}\int_{\mathbb{S}^2}(1-\ol_\eps)^2\eta \left[W^2+\frac{\ol_\eps}{1-\ol_\eps}\left(W^2-W^1\right)+W^3(\ol_\eps-1)\oF^\eps\right]\;dn \;dv\\& +\int_{\mathbb{R}^3}\int_{\mathbb{S}^2}(1-\ol_\eps)^2\eta \left[W^2+\frac{\ol_\eps}{1-\ol_\eps}\left(W^2-W^1\right)+W^3(\ol_\eps-1)\oF^\eps\right]\;dn \;dv.
\end{split}
\end{equation}
and similarly also
\begin{equation}\label{step24}
\begin{split}A_{\varphi,\eta}\left[\mathcal{L}(W)\right]=&\int_{\mathbb{R}^3}\int_{\mathbb{S}^2} \varphi \left[W^3(n)\left(W^2(v)-W^1(v)\right)\right]\;dn \;dv-\int_{\mathbb{R}^3}\int_{\mathbb{S}^2} \varphi \left[W^3(n)\left(W^2(v)-W^1(v)\right)\right]\;dn \;dv\\
-\int_{\mathbb{R}^3}\int_{\mathbb{S}^2}(1-\ol_\eps&)^2\eta \left[W^3(n)\left(W^2(v)-W^1(v)\right)\right]\;dn \;dv +\int_{\mathbb{R}^3}\int_{\mathbb{S}^2}(1-\ol_\eps)^2\eta \left[W^3(n)\left(W^2(v)-W^1(v)\right)\right]\;dn \;dv.
\end{split}
\end{equation}		
Now, using the definition of $ W_\eps $ we also see easily that 
\begin{equation}\label{step25}
\begin{split}A_{\varphi,\eta}\left[W_\eps\right]=&\int_{\mathbb{R}^3}\int_{\mathbb{S}^2} \varphi \alpha_\eps\;dn \;dv-\int_{\mathbb{R}^3}\int_{\mathbb{S}^2} dn\;\varphi \alpha_\eps\;dn \;dv\\
&-\int_{\mathbb{R}^3}\int_{\mathbb{S}^2}(1-\ol_\eps)^2\eta \alpha_\eps\;dn \;dv +\int_{\mathbb{R}^3}\int_{\mathbb{S}^2}(1-\ol_\eps)^2\eta\theta_\eps\;dn \;dv\\
=&(1-\ol_\eps)^2\eta\left(-\int_{\mathbb{R}^3} \alpha_\eps\;dv+\int_{\mathbb{S}^2}\theta_\eps\;dn\right)=0.
\end{split}
\end{equation}	
%
Applying this operator $ A_{\varphi,\eta} $ on both sides of equation \pef{step2} and using partial integration we see that we can simplify that relation as
\begin{equation}\label{step27}
A_{\varphi,\eta}\left[\partial_t\oFF^\eps\right]+A_{\varphi,\eta}\left[\partial_tW_\eps\right]=A_{\varphi,\eta}\left[\KK[\oFF^\eps+W_\eps,\oFF^\eps+W_\eps]\right].
\end{equation}
Since this result holds for every arbitrary $ \varphi,\eta $ we can conclude choosing first $ \eta=0 $ and $ \varphi $ arbitrary and then taking $ \varphi=0 $ and $ \eta=\frac{1}{(1-\ol_\eps)^2} $ that the following must be satisfied
\begin{equation}\label{step28}
\partial_t (1+\ol_\eps)\oF^\eps= \KK_1[\oFF^\eps+W_\eps,\oFF^\eps+W_\eps]+\KK_2[\oFF^\eps+W_\eps,\oFF^\eps+W_\eps]
\end{equation}
and 
\begin{equation}\label{step29}
\partial_t \left(\int_{\mathbb{R}^3} dv\; \ol_\eps\; \oF^\eps\right)+\partial_t\left(\frac{\ol_\eps}{1-\ol_\eps}\right)=\int_{\mathbb{R}^3} dv\;\KK_2[\oFF^\eps+W_\eps,\oFF^\eps+W_\eps].
\end{equation}
We can see that this is true by the following computations
\begin{equation}\label{step213}
A_{\varphi,0}\left[\partial_tW_\eps\right]=\int_{\mathbb{R}^3}dv\, \varphi \left(\partial_t\alpha_\eps-\partial_t\alpha_\eps\right)=0
\end{equation}
and using the structure of $ W_\eps $
\begin{equation}\label{step214}
A_{0,(1-\ol_\eps)^{-2}}\left[\partial_tW_\eps\right]=-\int_{\mathbb{R}^3}dv\,\partial_t\alpha_\eps+\int_{\mathbb{S}^2}dn\,\partial_t\theta_\eps=\partial_t\left(\int_{\mathbb{S}^2}dn\,\theta_\eps-\int_{\mathbb{R}^3}dv\,\alpha_\eps\right)=0.
\end{equation}
Equations \pef{step28} and \pef{step29} are very similar to the one we had in Section 2.2 in equations \pef{expansion8} and \pef{expansion9}. Therefore in the same way we conclude that $ \oFF^\eps $ must satisfy the following equation
\begin{equation}\label{step210}
\begin{split}
\partial_t \oF^\eps=& \frac{1}{1+\ol_\eps}\left(\mathbb{K}_1\left[\oFF^\eps+W_\eps,\oFF^\eps+W_\eps\right]+\mathbb{K}_2\left[\oFF^\eps+W_\eps,\oFF^\eps+W_\eps\right]\right)\\
&+\frac{(1-\ol_\eps)^2\oF^\eps}{(1+\ol_\eps)+(1-\ol_\eps)^2\int_{\mathbb{R}^3}\oF^\eps}\int_{\mathbb{R}^3} \mathbb{K}_1\left[\oFF^\eps+W_\eps,\oFF^\eps+W_\eps\right]dv\\
\partial_t \ol_\eps\;\oF^\eps=& \frac{\ol_\eps}{1+\ol_\eps}\left(\mathbb{K}_1\left[\oFF^\eps+W_\eps,\oFF^\eps+W_\eps\right]+\mathbb{K}_2\left[\oFF^\eps+W_\eps,\oFF^\eps+W_\eps\right]\right)\\
&-\frac{(1-\ol_\eps)^2\ol_\eps\;\oF^\eps}{(1+\ol_\eps)+(1-\ol_\eps)^2\int_{\mathbb{R}^3}\oF^\eps}\int_{\mathbb{R}^3} \mathbb{K}_1\left[\oFF^\eps+W_\eps,\oFF^\eps+W_\eps\right]dv\\
\partial_t\left(\frac{\ol_\eps}{1-\ol_\eps}\right) =& -\frac{(1+\ol_\eps)}{(1+\ol_\eps)+(1-\ol_\eps)^2\int_{\mathbb{R}^3}\oF^\eps}\int_{\mathbb{R}^3} \mathbb{K}_1\left[\oFF^\eps+W_\eps,\oFF^\eps+W_\eps\right]dv\\
\end{split}
\end{equation}
In order to simplify the notation we will say $ \partial_t\oFF^{\eps}=\mathbb{P}_{T_{\oFF^\eps}\mathcal{M}}\left(\KK\left[\oFF^\eps+W_\eps,\oFF^\eps+W_\eps\right]\right) $. With a similar notation, since we already know that $ \oFF $ satisfies
\begin{equation}\label{step211}
\begin{split}
\partial_t \oF=& \frac{1}{1+\ol}\left(\mathbb{K}_1\left[\oFF,\oF\right]+\mathbb{K}_2\left[\oFF,\oFF\right]\right)+\frac{(1-\ol)^2\oF}{(1+\ol)+(1-\ol)^2\int_{\mathbb{R}^3}\oF}\int_{\mathbb{R}^3} \mathbb{K}_1\left[\oFF,\oFF\right]dv\\
\partial_t \ol\;\oF=& \frac{\ol}{1+\ol}\left(\mathbb{K}_1\left[\oFF,\oFF\right]+\mathbb{K}_2\left[\oFF,\oFF\right]\right)-\frac{(1-\ol)^2\ol\;\oF}{(1+\ol)+(1-\ol)^2\int_{\mathbb{R}^3}\oF}\int_{\mathbb{R}^3} \mathbb{K}_1\left[\oFF,\oFF\right]dv\\
\partial_t\left(\frac{\ol}{1-\ol}\right) =& -\frac{(1+\ol)}{(1+\ol)+(1-\ol)^2\int_{\mathbb{R}^3}\oF}\int_{\mathbb{R}^3} \mathbb{K}_1\left[\oFF,\oFF\right]dv\\
\end{split}
\end{equation}
We will say $  \partial_t\oFF=\mathbb{P}_{T_{\oFF}\mathcal{M}}\left(\KK\left[\oFF,\oFF\right]\right) $.
We can now conclude this STEP 2 with the equation for the remainder $ W_\eps $ as follows
\begin{equation}\label{step212}
\partial_tW_\eps=\frac{1}{\eps}\left(L_{\oFF^\eps}(W_\eps)+\mathcal{L}(W_\eps)\right)+\KK\left[\oFF^\eps+W_\eps,\oFF^\eps+W_\eps\right]-\mathbb{P}_{T_{\oFF^\eps}\mathcal{M}}\left(\KK\left[\oFF^\eps+W_\eps,\oFF^\eps+W_\eps\right]\right).
\end{equation}
\textbf{STEP 3: The estimate for $ W_\eps $.}\\
In this step we want to estimate the norm of the remainder $ W_\eps $ and see that it will converge to zero for every positive time $ t\in(0,t_0] $. We aim to use Gr\"onwall's Lemma, therefore we proceed testing the right hand side of equation \pef{step212} with the vector $ (\phi(v),-\phi(v), \varphi(n)) $, where $ \phi(v)=\text{sign}(\alpha_\eps)\left(1+\arrowvert v\arrowvert^2\right) $ and $ \varphi(n)=\text{sign}(\theta_\eps) $. 
Using the estimates for the elastic and non elastic kernels \pef{usual estimate 3} and \pef{usual estimate 2} and the assumption on $ \ol_\eps $ it is not difficult to see that 
\begin{equation}\label{step3}
\int_{\mathbb{R}^3}dv\;\left\arrowvert\KK_1[\FF^\eps,\FF^\eps]+\KK_2[\FF^\eps,\FF^\eps]\right\arrowvert\left(1+\arrowvert v\arrowvert^2\right)\leq C(C_0,\eps_0)E_0\kappa_4
\end{equation} 
and similarly
\begin{equation}\label{step31}
\int_{\mathbb{R}^3}dv\;\left\arrowvert\KK_1[\FF^\eps,\FF^\eps]+\KK_2[\FF^\eps,\FF^\eps]\right\arrowvert\leq C(C_0,\eps_0)E_0^2.
\end{equation} 
For this reason there exists a constant $ C_1>0 $ which depends on the elastic and non elastic kernels, on the initial energy $ E_0 $ and on $ \kappa_4 $ but independent of $ \eps $ such that for every $ t\in[0,t_0] $ the following holds true
\begin{equation}\label{step32}\hspace*{-10pt}
\begin{split}
&\int_{\mathbb{R}^3}dv\;\phi(v) \left(\KK_1[\FF^\eps,\FF^\eps]+\KK_2[\FF^\eps,\FF^\eps]-\mathbb{P}^1_{T_{\oFF^\eps}\mathcal{M}}\left(\KK[\FF^\eps,\FF^\eps]\right)-\mathbb{P}^2_{T_{\oFF^\eps}\mathcal{M}}\left(\KK[\FF^\eps,\FF^\eps]\right)\right)-\int_{\mathbb{S}^2} dn\;\varphi(n)\mathbb{P}^3_{T_{\oFF^\eps}\mathcal{M}}\left(\KK[\FF^\eps,\FF^\eps]\right)\\&\leq C_1.
\end{split}
\end{equation}
It is also easy to compute the following estimate for $ \mathcal{L} $
\begin{equation}\label{step33}
\begin{split}
&\int_{\mathbb{R}^3}dv\;\phi(v)\mathcal{L}^1(W_\eps)-\int_{\mathbb{R}^3}dv\;\phi(v)\mathcal{L}^2(W_\eps)+\int_{\mathbb{S}^2} dn\;\varphi(n)\mathcal{L}^3(W_\eps)\\
=&2\int_{\mathbb{R}^3}dv\;\phi(v)\mathcal{L}^1(W_\eps)+\int_{\mathbb{S}^2} dn\;\varphi(n)\mathcal{L}^3(W_\eps)\\
=&-4\int_{\mathbb{R}^3} dv\; \arrowvert\alpha_\eps\arrowvert\left(1+\arrowvert v\arrowvert^2\right)\left(\int_{\mathbb{S}^2}dn\;\theta_\eps\right)-2\int_{\mathbb{R}^3} dv\; \alpha_\eps\left(\int_{\mathbb{S}^2}dn\;\arrowvert\theta_\eps\arrowvert\right)\\
\leq&2\Arrowvert W_\eps\Arrowvert_{\mathcal{X}}^2.
\end{split}
\end{equation} 
We can come to the conclusion for the estimate of the remainder looking at the behavior of the operator $ L_{\oFF^\eps}(W_\eps) $. Here we really want to bound that expression by a negative constant multiplied to the norm of $ W_\eps $. This would guarantee us the good decay for the remainder. Again we shall use also the properties of $ W_\eps $. We conclude
\begin{equation}\label{step34}
\begin{split}
&\int_{\mathbb{R}^3}dv\;\phi(v)L_{\oFF^\eps}^1(W_\eps)-\int_{\mathbb{R}^3}dv\;\phi(v)L_{\oFF^\eps}^2(W_\eps)+\int_{\mathbb{S}^2} dn\;\varphi(n)L_{\oFF^\eps}^3(W_\eps)\\
=&2\int_{\mathbb{R}^3}dv\;\phi(v)L_{\oFF^\eps}^1(W_\eps)+\int_{\mathbb{S}^2} dn\;\varphi(n)\mathcal{L}^3(W_\eps)\\
=&-2\frac{1+\ol_\eps}{1-\ol_\eps}\int_{\mathbb{R}^3} dv\; \arrowvert\alpha_\eps\arrowvert\left(1+\arrowvert v\arrowvert^2\right)-\left(\int_{\mathbb{S}^2}dn\;\text{sign}(\theta_\eps)\right)\frac{1+\ol_\eps}{1-\ol_\eps}\int_{\mathbb{R}^3} dv\; \alpha_\eps\\
&-2(1-\ol\eps)\left(\int_{\mathbb{R}^3}dv\;\oF^\eps\text{sign}(\alpha_\eps)\left(1+\arrowvert v\arrowvert^2\right)\right)\left(\int_{\mathbb{S}^2}dn\;\theta_\eps\right)-(1-\ol\eps)\left(\int_{\mathbb{R}^3}dv\;\oF^\eps\right)\left(\int_{\mathbb{S}^2}dn\;\arrowvert\theta_\eps\arrowvert\right)\\
=&I+II.
\end{split}
\end{equation}
Now we should estimate these terms. The first one is a consequence of taking the absolute value
\begin{equation}\label{step35}
I\leq -\frac{1+\ol_\eps}{1-\ol_\eps}\int_{\mathbb{R}^3} dv\; \arrowvert\alpha_\eps\arrowvert\left(1+\arrowvert v\arrowvert^2\right)\leq -\int_{\mathbb{R}^3} dv\; \arrowvert\alpha_\eps\arrowvert\left(1+\arrowvert v\arrowvert^2\right).
\end{equation}
For the second one using that the initial energy $ E_0 $ is bounded by $ \frac{1}{2} $ and that as we saw after the first step $ \int_{\mathbb{R}^3} dv \arrowvert \oF^\eps\arrowvert(1+\arrowvert v\arrowvert^2)\leq E_0 $. Since also $ \ol\in[0,\frac{E_0}{2\eps_0+E_0}) $ and $ \int_{\mathbb{R}^3}\oF_\eps(v)dv\in\left(\frac{\kappa_0}{2},\kappa_0\right) $ we compute
\begin{equation}\label{step37}
\begin{split}
II&\leq -2(1-\ol_\eps)\left(\int_{\mathbb{R}^3}dv\;\oF^\eps\left(1+\arrowvert v\arrowvert^2\right)\right)\left(\left(1-\frac{\kappa_0}{4}\right)\int_{\mathbb{R}^3}dv\;\alpha_\eps+\frac{\kappa_0}{4}\int_{\mathbb{S}^2}dn\;\theta_\eps\right)\\&\;\;\;\;-(1-\ol\eps)\left(\int_{\mathbb{R}^3}dv\;\oF^\eps\right)\left(\int_{\mathbb{S}^2}dn\;\arrowvert\theta_\eps\arrowvert\right)\\
\leq& \left(1-\frac{\kappa_0}{4}\right)\int_{\mathbb{R}^3}dv\;\arrowvert\alpha_\eps\arrowvert -2(1-\ol)\frac{\kappa_0}{4}\left(\int_{\mathbb{R}^3}dv\;\oF^\eps\left(1+\arrowvert v\arrowvert^2\right)\right)\int_{\mathbb{S}^2}dn\;\theta_\eps-(1-\ol)\left(\int_{\mathbb{R}^3}dv\;\oF^\eps\right)\left(\int_{\mathbb{S}^2}dn\;\arrowvert\theta_\eps\arrowvert\right)\\
&-(\ol-\ol_\eps)\frac{\kappa_0}{4}\left(\int_{\mathbb{R}^3}dv\;\oF^\eps\left(1+\arrowvert v\arrowvert^2\right)\right)\int_{\mathbb{S}^2}dn\;\theta_\eps-(\ol-\ol_\eps)\left(\int_{\mathbb{R}^3}dv\;\oF^\eps\right)\left(\int_{\mathbb{S}^2}dn\;\arrowvert\theta_\eps\arrowvert\right)\\
\leq&\left(1-\frac{\kappa_0}{4}\right)\int_{\mathbb{R}^3}dv\;\arrowvert\alpha_\eps\arrowvert - \frac{\kappa_0}{4}\frac{2\eps_0}{2\eps_0+E_0}\left(\int_{\mathbb{S}^2}dn\;\arrowvert\theta_\eps\arrowvert\right)\\&+\arrowvert\ol_\eps-\ol\arrowvert\frac{\kappa_0}{8}\left(2\int_{\mathbb{R}^3}dv\;\arrowvert\alpha_\eps\arrowvert\right)+\arrowvert\ol_\eps-\ol\arrowvert\kappa_0\left(\int_{\mathbb{S}^2}dn\;\arrowvert\theta_\eps\arrowvert\right).
\end{split}
\end{equation}
Combining these estimates and equation \pef{deco3} in STEP 1 we see that for every $ t\in[0,t_0] $ we have
\begin{equation}\label{step38}
I+II\leq -\frac{\kappa_0}{8}\frac{2\eps_0}{2\eps_0+E_0}\Arrowvert W_\eps\Arrowvert_{\mathcal{X}}+\kappa_0\Arrowvert \oFF^\eps-\oFF\Arrowvert_{\mathcal{X}}\Arrowvert W_\eps\Arrowvert_{\mathcal{X}}\leq -\frac{\kappa_0}{16}\frac{2\eps_0}{2\eps_0+E_0}\Arrowvert W_\eps\Arrowvert_{\mathcal{X}}.
\end{equation}
Putting together equations \pef{step32}, \pef{step33} and \pef{step38} and using the result of STEP 1 that \linebreak$ \Arrowvert W_\eps\Arrowvert_{\mathcal{X}}\leq\delta_0~:=~ \frac{2\eps_0\kappa_0}{64(2\eps_0+E_0)}  $, we conclude
\begin{equation}\label{step39}
\partial_t \Arrowvert W_\eps\Arrowvert_{\mathcal{X}}\leq -\frac{C_2}{\eps}\Arrowvert W_\eps\Arrowvert_{\mathcal{X}}+C_1 \;\;\;\;\;\;\;\text{ for all }t\in[0,t_0],
\end{equation}
where $ C_1=\frac{\kappa_0}{32}\frac{2\eps_0}{2\eps_0+E_0} $ is independent of $ \eps $ and $ t $.

Using that $ \Arrowvert W_0\Arrowvert_{\mathcal{X}}<\delta<\delta_0 $ Gr\"onwall's Lemma implies the good decay of the remainder as $ \eps\to 0 $
\begin{equation}\label{step310}
\Arrowvert W_\eps\Arrowvert \leq \Arrowvert W_0\Arrowvert_{\mathcal{X}} e^{-\frac{C_2}{\eps}t}+\frac{C_1}{C_2}\eps\left(1-e^{-\frac{C_2}{\eps}t}\right)\leq\delta e^{-\frac{C_2}{\eps}t}+\frac{C_1}{C_2}\eps.
\end{equation}
Therefore we conclude also for all $ s>0 $
\begin{equation}\label{step311}
\lim\limits_{\eps\to 0}\sup_{s\leq t\leq t_0}\Arrowvert W_\eps\Arrowvert =\lim\limits_{\eps\to 0} \delta e^{-\frac{C_2}{\eps}s}+\frac{C_1}{C_2}\eps=0.
\end{equation}
\textbf{STEP 4: The estimate for $ \oFF^\eps-\oFF $.}\\
We proceed estimating the norm of $ \oFF^\eps-\oFF $ using also in this case the Gr\"onwall's inequality. This step will then imply the theorem, as we will see in STEP 5.

We start recalling that as we have seen in STEP 1, namely $ \ol_\eps\in(0,1) $. Now, recalling that $ \oF>0 $ and $ \FF^\eps>0 $ almost everywhere, since the initial data is positive. It is important to notice, that by Lemma \ref{parametrization} also $ \oF^\eps>0 $.
We want first to estimate the time derivative of $ \Arrowvert \oFF^\eps-\oFF\Arrowvert_{\mathcal{X}} $. We start with some preliminary estimates. We define two functions $ \phi(v)=\text{sign}\left(\oF^\eps-\oF\right)\left(\arrowvert v\arrowvert^2\right) $ and $ \psi(v)=\text{sign}\left(\ol_\eps\oF^\eps-\ol\;\oF\right)\left(\arrowvert v\arrowvert^2\right) $. Then since both functions $ \oF^\eps $ and $ \oF $ are non-negative and their $ L^1_4 $-norm is bounded uniformly by $ \kappa_4 $, estimates \pef{uniq1} and \pef{uniq2} imply
\begin{equation}\label{step45}
\begin{split}
\int_{\mathbb{R}^3} dv \;\phi(v)\left[\KK_1[\oF^\eps,\ol_\eps]+\KK_2[\oF^\eps,\ol_\eps]-\KK_1[\oF,\ol_\eps]-\KK_2[\oF,\ol_\eps]\right]
\leq C(\eps_0, C_0) \kappa_4 \Arrowvert \oF^\eps-\oF\Arrowvert_{L_2^1}.
\end{split}
\end{equation}
where the constant is independent of the time and of $ \eps $ but depends only on the kernels. Using now that 
\begin{equation}\label{step410}
\begin{split}
&\ol_\eps\oF^\eps(v)\oF^\eps(w)-\ol\,\oF(v)\oF(w)\\=&\left(\ol_\eps\oF^\eps(v)-\ol\,\oF(v)\right)\oF^\eps(w)+\oF(v)\left(\ol_\eps\oF^\eps(w)-\ol\,\oF(w)\right)+\left(\ol_\eps-\ol\right)\oF(v)\oF^\eps(w),
\end{split}
\end{equation}
we can compute completely analogously as in \pef{uniq1} and \pef{uniq2} that there exists some constant \linebreak$ C=C(\eps_0,C_0,\kappa_4) $ independent of time and $ \eps $ such that
\begin{equation}\label{step411}
\begin{split}
&\int_{\mathbb{R}^3} dv \;\psi(v)\left[\ol_\eps\KK_1[\oF^\eps,\ol_\eps]+\ol_\eps\KK_2[\oF^\eps,\ol_\eps]-\ol\KK_1[\oF,\ol_\eps]-\ol\KK_2[\oF,\ol_\eps]\right]\\\leq &C\left( \Arrowvert \oF^\eps-\oF\Arrowvert_{L_2^1}+\Arrowvert \ol_\eps\oF^\eps-\ol\,\oF\Arrowvert_{L_2^1}+\left\arrowvert\frac{\ol_\eps}{1-\ol_\eps}-\frac{\ol}{1-\ol}\right\arrowvert\right).
\end{split}
\end{equation}
Now we are ready for looking at the last estimates. First of all we can write down the equation for $ \partial_t\left(\oF^\eps-\oF\right) $ as follows
\begin{equation}\label{step412}
\begin{split}
\partial_t\left(\oF^\eps-\oF\right)=&\frac{1}{1+\ol_\eps}\left(\KK_1+\KK_2\right)\left[\oF^\eps,\ol_\eps\right]-\frac{1}{1+\ol}\left(\KK_1+\KK_2\right)\left[\oF,\ol\right]\\
&+\frac{(1-\ol_\eps)^2\oF^\eps}{(1+\ol_\eps)+(1-\ol_\eps)^2\int_{\mathbb{R}^3}\oF^\eps}\int_{\mathbb{R}^3} \mathbb{K}_1\left[\oF^\eps,\ol_\eps\right]dv\\
&-\frac{(1-\ol)^2\oF}{(1+\ol)+(1-\ol)^2\int_{\mathbb{R}^3}\oF}\int_{\mathbb{R}^3} \mathbb{K}_1\left[\oF,\ol\right]dv\\
&+\frac{1}{1+\ol_\eps}\left(\KK_1+\KK_2\right)\left(\left[\oFF^\eps,W_\eps\right]+\left[W_\eps,\oFF^\eps\right]+\left[W_\eps,W_\eps\right]\right)\\
&+\frac{(1-\ol_\eps)^2\oF^\eps}{(1+\ol_\eps)+(1-\ol_\eps)^2\int_{\mathbb{R}^3}\oF^\eps}\int_{\mathbb{R}^3} \mathbb{K}_1\left(\left[\oFF^\eps,W_\eps\right]+\left[W_\eps,\oFF^\eps\right]+\left[W_\eps,W_\eps\right]\right)dv.
\end{split}
\end{equation}
And therefore it is easy to compute, knowing the result \pef{step45} and that $ \arrowvert\ol_\eps-\ol\arrowvert\leq \left\arrowvert\frac{\ol_\eps}{1-\ol_\eps}-\frac{\ol}{1-\ol}\right\arrowvert $, the following estimate
\begin{equation}\label{step413}
\begin{split}
\partial_t\Arrowvert \oF^\eps&-\oF\Arrowvert_{L_2^1}\leq\frac{1}{1+\ol_\eps}\int_{\mathbb{R}^3} dv \;\phi(v)\left[\KK_1\left[\oF^\eps,\ol_\eps\right]+\KK_2\left[\oF^\eps,\ol_\eps\right]-\KK_1\left[\oF,\ol_\eps\right]-\KK_2\left[\oF,\ol_\eps\right]\right]\\
&+\frac{1}{1+\ol_\eps}\int_{\mathbb{R}^3} dv \;\phi(v)\left[\KK_1\left[\oF,\ol_\eps\right]+\KK_2\left[\oF,\ol_\eps\right]-\KK_1\left[\oF,\ol\right]-\KK_2\left[\oF,\ol\right]\right]\\
&+\left(\frac{1}{1+\ol_\eps}-\frac{1}{1+\ol}\right)\int_{\mathbb{R}^3} dv \;\phi(v)\left[\KK_1\left[\oF,\ol\right]+\KK_2\left[\oF,\ol\right]\right]\\
&+\frac{(1-\ol_\eps)^2\int_{\mathbb{R}^3} \mathbb{K}_1\left[\oF^\eps,\ol_\eps\right]dv}{(1+\ol_\eps)+(1-\ol_\eps)^2\int_{\mathbb{R}^3}\oF^\eps}\int_{\mathbb{R}^3}dv\left(\oF^\eps-\oF\right)\phi(v)\\
&+\frac{(1-\ol_\eps)^2\int_{\mathbb{R}^3}dv\,\oF\phi(v)}{(1+\ol_\eps)+(1-\ol_\eps)^2\int_{\mathbb{R}^3}\oF^\eps}\left(\int_{\mathbb{R}^3} \mathbb{K}_1\left[\oF^\eps,\ol_\eps\right]-\mathbb{K}_1\left[\oF,\ol_\eps\right]dv+\int_{\mathbb{R}^3} \mathbb{K}_1\left[\oF,\ol_\eps\right]-\mathbb{K}_1\left[\oF,\ol\right]dv\right)\\
&+(1-\ol_\eps)^2\int_{\mathbb{R}^3}dv\,\oF\phi(v)\int_{\mathbb{R}^3} \mathbb{K}_1\left[\oF,\ol\right]dv\left(\frac{1}{(1+\ol_\eps)+(1-\ol_\eps)^2\int_{\mathbb{R}^3}\oF^\eps}-\frac{1}{(1+\ol)+(1-\ol)^2\int_{\mathbb{R}^3}\oF}\right)\\
&+\frac{1}{1+\ol_\eps}\int_{\mathbb{R}^3}dv\,\phi(v)\left(\KK_1+\KK_2\right)\left(\left[\oFF^\eps,W_\eps\right]+\left[W_\eps,\oFF^\eps\right]+\left[W_\eps,W_\eps\right]\right)\\
&+\frac{(1-\ol_\eps)^2\int_{\mathbb{R}^3}dv\,\phi(v)\oF^\eps}{(1+\ol_\eps)+(1-\ol_\eps)^2\int_{\mathbb{R}^3}\oF^\eps}\int_{\mathbb{R}^3} \mathbb{K}_1\left(\left[\oFF^\eps,W_\eps\right]+\left[W_\eps,\oFF^\eps\right]+\left[W_\eps,W_\eps\right]\right)dv\\
\leq& C(\eps_0,C_0,\kappa_4,\kappa_2)\left(\Arrowvert \oFF^\eps-\oFF\Arrowvert_{\mathcal{X}}+\Arrowvert W_\eps\Arrowvert_{\mathcal{X}}^{1/2}\right).
\end{split}
\end{equation}
This estimate holds for several reasons. The first term was estimated in \pef{step45}, the second and the third hold by the boundedness of the $ L^1_4 $-norm and by the equations \pef{uniq1} and \pef{uniq2}, the fourth uses the non-negativity of the integral of $ \oF^\eps $ and equation \pef{lip3}. For the fifth term we can use the estimate in \pef{lip2} together with the remark about the norm of $ \oF^\eps $ and estimate \pef{lip5}, while the sixth terms is an easy calculation. For the last two terms we use the estimates on the kernels as in \pef{usual estimate 3} and \pef{usual estimate} to see that 
\begin{equation}\label{step416}
\begin{split}
&\int_{\mathbb{R}^3}dv\,\left[\phi(v)\left(\KK_1+\KK_2\right)\left(\left[\oFF^\eps,W_\eps\right]+\left[W_\eps,\oFF^\eps\right]+\left[W_\eps,W_\eps\right]\right)+ \mathbb{K}_1\left(\left[\oFF^\eps,W_\eps\right]+\left[W_\eps,\oFF^\eps\right]+\left[W_\eps,W_\eps\right]\right)\right]dv\\
\leq & C(\eps_0, C_0)\left(\Arrowvert \oF^\eps\Arrowvert_{L_3^1}\Arrowvert W_\eps\Arrowvert_{L_3^1}+\Arrowvert W_\eps\Arrowvert_{L_3^1}^2+\Arrowvert \oF^\eps\Arrowvert_{L_2^1}\Arrowvert W_\eps\Arrowvert_{L_2^1}+\Arrowvert W_\eps\Arrowvert_{L_2^1}^2\right)\leq C(\eps_0, C_0)\kappa_4^{3/2}\Arrowvert W_\eps\Arrowvert_{L_2^1}^{1/2}
\end{split}
\end{equation}
where in the end we used the interpolation formula for the $ L^1_3 $-norm.\\

In a completely analogous way we can estimate also the estimate for $ \partial_t \Arrowvert \ol_\eps\oF^\eps-\ol\,\oF\Arrowvert_{L_2^1} $ and \linebreak$ \partial_t\left\arrowvert \frac{\ol_\eps}{1-\ol_\eps}-\frac{\ol}{1-\ol}\right\arrowvert $. In the first case, equation \pef{step411} will substitute the first term in \pef{step413} and we substitute $ \phi $ with $ \psi $ and multiply with $ \ol $ or $ \ol_\eps $ were needed. Then following exactly the same calculation we did in \pef{step413} we estimate 
\begin{equation}\label{step414}
\partial_t \Arrowvert \ol_\eps\oF^\eps-\ol\,\oF\Arrowvert_{L_2^1}\leq C(\eps_0,C_0,\kappa_4,\kappa_2)\left(\Arrowvert \oFF^\eps-\oFF\Arrowvert_{\mathcal{X}}+\Arrowvert W_\eps\Arrowvert_{\mathcal{X}}^{1/2}\right).
\end{equation}
We conclude with the estimate for the photon number density. Here it is even easier since we estimate everything by means of the absolute value following the estimates for the lines from five to eight and line ten of equation \pef{step413}. We conclude therefore that there exists some constant $ C=C(\eps_0,C_0,\kappa_4)$ independent of $ t $ and $ \eps_0 $ such that for all $ \eps>0 $ and $ t\in[0,t_0] $ the following holds true
\begin{equation}\label{step415}
\Arrowvert \oFF^\eps-\oFF\Arrowvert_{\mathcal{X}}\leq \int_{0}^t C \left(\Arrowvert \oFF^\eps-\oFF\Arrowvert_{\mathcal{X}}+\Arrowvert W_\eps\Arrowvert_{\mathcal{X}}^{1/2}\right).
\end{equation}
\newpage
\hspace{-15pt}\textbf{STEP 5: Conclusion.}\\
Now we are ready to conclude the claim of the theorem. Given equation \pef{step310} we see that equation \pef{step415} can be estimated by
\begin{equation}\label{step5}
\begin{split}
\Arrowvert \oFF^\eps-\oFF\Arrowvert_{\mathcal{X}}\leq& \int_{0}^t C \left(\Arrowvert \oFF^\eps-\oFF\Arrowvert_{\mathcal{X}}+\sqrt{\delta} e^{-\frac{C_2}{2\eps}t}+\sqrt{\frac{C_1}{C_2}}\sqrt{\eps}\right)dt\\
=& \frac{C\sqrt{2\delta\eps}}{C_2}\left(1-e^{-\frac{C_2}{2\eps}t}\right)+C\sqrt{\frac{C_1}{C_2}}\sqrt{\eps} t+ \int_{0}^t C \Arrowvert \oFF^\eps-\oFF\Arrowvert_{\mathcal{X}}dt\\
\leq & C\left(\frac{C\sqrt{2\delta\eps}}{C_2}+\sqrt{\frac{C_1}{C_2}}\sqrt{\eps}t\right)+\int_{0}^t C \Arrowvert \oFF^\eps-\oFF\Arrowvert_{\mathcal{X}}dt.
\end{split}
\end{equation}
Hence, the Gr\"onwall's Inequality implies 
\begin{equation}\label{step51}
\Arrowvert \oFF^\eps-\oFF\Arrowvert_{\mathcal{X}}\leq C\left(\frac{C\sqrt{2\delta\eps}}{C_2}+\sqrt{\frac{C_1}{C_2}}\sqrt{\eps}t\right)e^{Ct},
\end{equation}
which converges to zero uniformly in $ [0,t_0] $ as $ \eps\to0 $. Since $ \Arrowvert \FF^\eps-\oFF\Arrowvert_{\mathcal{X}}\leq \Arrowvert \oFF^\eps-\oFF\Arrowvert_{\mathcal{X}}+\Arrowvert W_\eps\Arrowvert_{\mathcal{X}} $ equations \pef{step311} and \pef{step51} imply the theorem.
\end{proof}
\end{theorem}
We can extend the result of the theorem for larger time. We can argue using the fact that for some positive time $ t_0 $ the solution $ \FF^\eps $ is in a small neighborhood of $ \mathcal{M}_+ $ and so we can again apply Theorem \ref{thmrigorous1}. We repeat then the same procedure.
\begin{corollary}
Let $ T>0 $. Under the assumption of Theorem \ref{thmrigorous1} the following holds for all $ s>0 $
\begin{equation}\label{corollary}
\lim\limits_{\eps\to 0}\sup_{s\leq t\leq T}\Arrowvert \FF^\eps-\oFF\Arrowvert_\mathcal{X}=0.
\end{equation}
\begin{proof}
We only need to check Step 1 in Theorem \ref{thmrigorous1} at time $ t_0 $. There exists some $ \eps_1>0 $ such that for all $ \eps\leq \eps_1 $ we have on one hand that 
\begin{equation}\label{corollary2}
\Arrowvert \FF^\eps (t_0)-\oFF^\eps(t_0)\Arrowvert_{\mathcal{X}}=\Arrowvert W_\eps(t_0)\Arrowvert_{\mathcal{X}}<\delta<\delta_0
\end{equation}
and on the other hand also 
\begin{equation}\label{corollary3}
\Arrowvert \oFF^\eps (t_0)-\oFF(t_0)\Arrowvert_{\mathcal{X}}<\tilde{\delta}.
\end{equation}
This holds for some $ \tilde{\delta}>0 $ small enough such that there exists some $ t_2=t_0+\nu>t_0 $ so that for every $ t\in[t_0-\nu,t_2]  $ we have
\begin{equation}\label{cor4}
\Arrowvert \FF^\eps (t)-\oFF^\eps(t_0)\Arrowvert_{\mathcal{X}}<\delta_1<\delta_0
\end{equation}
for some $ \delta_1 $ small enough. $ \delta_1 $ is chosen so that we have again a decomposition in $ \FF^\eps=\oFF^\eps+W_\eps $ as in the Theorem \ref{thmrigorous1}. The time $ t_2 $ is also chosen so that for all $ t\in[0,t_2] $ the following holds
\begin{equation}\label{cor5}
\Arrowvert \oFF^\eps (t)-\oFF(t)\Arrowvert_{\mathcal{X}}<\frac{1}{16}\frac{2\eps_0}{2\eps_0+E_0}.
\end{equation}
Then the corollary follows from all the steps in Theorem \ref{thmrigorous1} and from the uniform continuity of the $ \mathcal{X} $-norm in $ [0,T] $, i.e. we can repeat a finite number of time the above argument.
\end{proof}
\end{corollary}
\section{Summary and final comments}
In this paper we considered a simplified model describing the interaction of a gas with monochromatic radiation. We focused on the fast radiation limit, i.e. on the situation when the radiative processes are much more frequent than the collisions between gas molecules. Starting from a scaled version of the initial kinetic system (cf. equation \pef{epsequation}), we formally derived an effective kinetic equation which describes the behavior of the gas-photon system immediately after the initial time. To this end we analyzed the dynamic of the system near the slow manifold of steady states $ \mathcal{M}_+ $. We performed a perturbative expansion, which we called generalized Chapman-Enskog expansion due to its reminiscence to the classical Chapman-Enskog expansion of the Boltzmann equation, and we derived equation \pef{newsystem1}. After having developed a well-posedness theory for the resulting equation (cf. Theorem \ref{infinityexistence1} and Theorem \ref{infinityuniqueness}) we proved rigorously that he solutions $ \FF^\eps $ of \pef{epsequation} converge to the solutions of \pef{newsystem1} as $ \eps\to\infty $. The main technical idea for this is that every vector in a small neighborhood of the manifold of steady states can be written as the sum of a vector in $ \mathcal{M}_+ $ (the projection) and an orthogonal vector (cf. Lemma \ref{parametrization}). Decomposing in this way the sequence $ \FF^\eps $ we concluded the proof of Theorem \ref{thmrigorous1} observing that the sequence of orthogonal vectors converges to zero uniformly for every positive time, while the sequence of projections converges to the solution of the formal derived equation \pef{newsystem1}. 

An interesting question we did not consider in this paper is the long-time behavior of the solutions of \pef{newsystem1}. In the case, when the collision terms and the radiative terms are of the same order, it is possible to show that strong solutions of \pef{kinsystem} converge to the equilibrium whenever $ t\to\infty $. We refer to \cite{tesi}. The crucial result one needs for the proof of this claim, is that the entropy of strong solutions is bounded from above and from below and it is non-increasing, the so called H-Theorem. Also for the kinetic system \pef{newsystem1} we can define an equilibrium. As in the classical Boltzmann equation, at the equilibrium the gas densities is given by a Maxwellian.
\begin{prop}\label{equilibrium}
	The equilibrium solution to the system \pef{newsystem1} is given by the pair $ \begin{pmatrix}e^{-2k\varepsilon_0}\\Ae^{-k\arrowvert v-u\arrowvert^2}\end{pmatrix} $, where $ A,k>0 $ and $ u\in\mathbb{R}^3 $.
\end{prop}

Moreover, in the fast radiation limit it is still possible to show that the lower semi-continuity of the entropy yields the upper and lower uniform bound of the entropy for the solution to \pef{newsystem1} as in Definition \ref{solution}. 
\begin{prop}
	Under the assumptions of Theorem \ref{infinityexistence1} the entropy $ \mathcal{H}(t)=\mathcal{H}\left[\lambda, F\right](t) $ of the weak solution $ \left(\lambda, F\right)^\top\in C\left([0,\infty), X\right) $ satisfies
	\begin{equation*}
	K\leq \mathcal{H}(t)\leq \mathcal{H}(0)
	\end{equation*}
	for all $ t\geq 0 $ and for some $ K\in\mathbb{R} $.
\end{prop}
However, in order to have also for the weak solutions a version of the H-theorem as in Lemma \ref{lemcutoff1}, we shall pass to the limit in the equation for the dissipation of the cut-off approximating solutions. Unfortunately, the weak $ L^1 $-convergence is not strong enough for that.

Nonetheless, for strong solutions to special bounded kernels the long-time convergence to the equilibrium can be proved. Indeed, we can use the theory developed for the cut-off system in Section 3.
For given positive constants $ \kappa_0 $, $ \kappa_2 $ and $ \kappa_4 $ we can consider the set 
\begin{equation*}
\Omega:=\left\{ \left(\lambda, F\right)\in X: \lambda\in[0,1),\; F\geq 0,\; \Arrowvert F\Arrowvert_{L^1}\leq \kappa_0,\;E\left(\lambda,F\right)\leq \kappa_2\text{ and }\Arrowvert F\Arrowvert_{L_4^1}\leq\kappa_4\right\}.
\end{equation*}
We define $ M $ to be the $ M\in\mathbb{N} $ as in Lemma \ref{lemexistence} such that the solution $ F_M $ of the $ M $-cut-off kinetic system for the initial data $ (\lambda_0,F_0)\in\Omega $ satisfies $ \Arrowvert F_M\Arrowvert_{L_4^1}\leq C $ for all $ t\geq 0 $ (C depens only on  initial energy, mass and $ L^1_4$-norm). 
We can consider the kinetic equation \pef{newsystem1} with the bounded cut-off kernels as in \pef{cutoff1}, \pef{cutoff2} and \pef{cutoff3} for $ n=M $. As a consequence of Lemma \ref{lemcutoff5} and Theorem \ref{infinityexistence1} and adapting the proof of Theorem 3.6 in \cite{arkerydII} we can show that the solution $ \left(\lambda, F\right) $ converge to the unique equilibrium as $ t\to\infty $.
\begin{theorem}
	Let $ (\lambda_0, F_0)\in\Omega $, $ F_0\log(F_0)\in L^1\left(\mathbb{R}^3\right) $. Assume $ F_0 $ satisfies assumption \pef{Assumption}. Let $ (\lambda, F)\in C^1\left([0,\infty), X\right) $ be the unique strong solution of this kinetic equation with bounded kernels. Then $ (\lambda, F) $ converges to the unique equilibrium solution $ \begin{pmatrix}e^{-2k\varepsilon_0}\\Ae^{-k\arrowvert v-u\arrowvert^2}\end{pmatrix} $ which conserves the initial mass, momentum and energy. The convergence is pointwise for $ \lambda $ and weak in $ L^1\left(\mathbb{R}^3\right) $ for $ F $.
\end{theorem}
\nocite{*}
\bibliographystyle{siam}
\bibliography{bibliography}

\bigskip

\hspace{-18,5pt}
\textsc{Elena Demattè}\\
University of Bonn, Institute for Applied Mathematics\\
Endenicher Allee 60\\
D-53115 Bonn\\
GERMANY\\
\textit{E-mail address}: dematte@iam.uni-bonn.de
\end{document}